\documentclass[aos]{imsart}

\RequirePackage{amsthm,amsmath,amsfonts,amssymb}
\RequirePackage[numbers]{natbib}
\RequirePackage[colorlinks,citecolor=blue,urlcolor=blue]{hyperref}
\RequirePackage{graphicx}

\usepackage{amsmath}
\usepackage{bbm}
\usepackage{amssymb}
\usepackage{amsthm}
\usepackage{mathrsfs}
\usepackage{xcolor}
\usepackage{braket}
\usepackage{tikz} 
\usetikzlibrary{cd}
\usepackage{enumerate}
\usepackage{comment}
\usepackage{physics}
\usepackage{amsmath}
\usepackage{tikz}
\usepackage{mathdots}
\usepackage{yhmath}
\usepackage{cancel}
\usepackage{color}
\usepackage{siunitx}
\usepackage{array}
\usepackage{multirow}
\usepackage{amssymb}
\usepackage{gensymb}
\usepackage{tabularx}
\usepackage{extarrows}
\usepackage{booktabs}
\usetikzlibrary{fadings}
\usetikzlibrary{patterns}
\usetikzlibrary{shadows.blur}
\usetikzlibrary{shapes}

\startlocaldefs
\theoremstyle{plain}
\newtheorem{lemma}{Lemma}
\newtheorem{corollary}{Corollary}
\newtheorem{proposition}{Proposition}
\newtheorem{theorem}{Theorem}
\theoremstyle{remark}
\newtheorem{definition}{Definition}

\def\id{\mathbf 1}

\def\RR{\mathbb{R}}
\def\cc{\mathbb{C}}
\def\nn{\mathbb{N}}
\def\tr{{\rm Tr}}
\def\TT{{\cal T}}

\def\EE{{\cal E}}

\newcommand{\manifirr}{\mathscr V^{\rm irr}}
\newcommand{\orbifirr}{\mathscr P^{\rm irr}}
\newcommand{\manifprim}{\mathscr V^{\rm prim}}
\newcommand{\tsnonid}{ T^{\rm nonid}}
\newcommand{\tsident}{ T^{\rm id}}

\endlocaldefs

\begin{document}

\begin{frontmatter}
\title{Asymptotic statistical theory of irreducible quantum Markov chains}

\begin{aug}
\author[A,B]{\fnms{Federico}~\snm{Girotti}\ead[label=e1]{federico.girotti@polimi.it}},
\author[C]{\fnms{Jukka}~\snm{Kiukas}}
\and
\author[A,D]{\fnms{M\u{a}d\u{a}lin}~\snm{Gu\c{t}\u{a}}}
\address[A]{School of Mathematical Sciences, University of Nottingham, Nottingham, NG7 2RD, United Kingdom}

\address[B]{Department of Mathematics, Polytechnic University of Milan, Milan, Piazza L. da Vinci 32, 20133, Italy\printead[presep={,\ }]{e1}}

\address[C]{Department of Mathematics, Aberystwyth University, Aberystwyth,
SY23 3BZ, United Kingdom}

\address[D]{Centre for the Mathematics and Theoretical Physics of Quantum Non-equilibrium Systems, University of Nottingham, Nottingham, NG7 2RD, United Kingdom}

\end{aug}

\begin{abstract}
In this paper we investigate the asymptotic statistical theory of irreducible quantum Markov chains, focusing on identifiability properties and asymptotic convergence of associated quantum statistical models.
We show that the space of identifiable parameters for the stationary output is a stratified space called an orbifold, which is obtained as the quotient of the manifold of irreducible dynamics by a compact group of state preserving symmetries. We analyse the orbifold's geometric properties, the connection between periodicity and strata, and provide orbifold charts as the starting point for the local asymptotic theory. The quantum Fisher information rate of the system and output state is expressed in terms of a canonical inner product on the identifiable tangent space. We then show that the joint system–output model satisfies quantum local asymptotic normality while the stationary output model converges to a product between a quantum Gaussian shift model and a mixture of quantum Gaussian shift models, reflecting the underlying periodicity. These  strong convergence results provide the basis for constructing asymptotically optimal estimators of dynamical parameters. We provide an in-depth analysis of the model with smallest dimensions, consisting of two-dimensional system and environment units. 
\end{abstract}



\end{frontmatter}
\tableofcontents


\section{Introduction and summary of results}


In recent decades quantum statistics has evolved from a theoretical challenge into a foundational principle of quantum technology that  underpins advances in diverse areas such as state tomography \cite{Paris,Hayashi,Lvov09,GutaGill25,Huang20}, quantum enhanced sensing and metrology \cite{Fujiwara08,Escher11,Demko12,Degen,Zhou18,Pezze,Meyer25}, quantum imaging \cite{Tsang16,Def2024}, gravitational waves detection \cite{DK12,GW} and more. In general terms, the  estimation problem is to learn an unknown parameter  encoded in the state of a quantum system, by performing a measurement on the system and using the outcome to compute an estimator \cite{Ho82,He76}. As quantum measurements have probabilistic outcomes and typically disturb the system's state, one usually requires multiple measurements on independently prepared systems, a setup similar to  estimation with i.i.d. models in statistics. A key early result in quantum estimation is the quantum Cram\'{e}r-Rao bound (QCRB) for parametric models \cite{Ho73,Be76,He76,Ho82,QCR1}. Given a system prepared in a state $\rho_\theta$, the covariance of any unbiased estimator obtained by measuring the system is lower bounded by the inverse of the quantum Fisher information $F_\theta$. The latter is an intrinsic property of the quantum statistical model, which suggests that concepts and tools from ``classical statistics'' can be extended to the domain of quantum statistical inference. 

Since these early results, notable developments include among others the theory of quantum sufficiency \cite{Petz86,JencovaPetz06}, the full classification of quantum Fisher information metrics \cite{Petz96}, the asymptotic hypothesis testing theory including quantum Stein Lemma \cite{HP91,ON00,OH04} and its generalisations \cite{BP10,Lami25,HY25}, quantum Sanov theorem \cite{Bjela05,Not14} and quantum Chernoff bound \cite{Auden07,Auden08}, quantum decision theory \cite{Chefles04,Buscemi12,Jencova12}, quantum compressed sensing \cite{Gross10,Flammia12,Cramer10,Kueng17}, low rank tomography \cite{Kolt112,Xia16,Cai16,Alq13,BGK15,GKKT20,SJKG22,AG17}, quantum local asymptotic normality \cite{LAN1,GJ07,GJK08,GK09,GG13,YFG13,BGM18,FY20,FY23,LN24}, optimal state estimation \cite{LAN1,GG13,ODonW16,ODonW17,Haah17,YCH19,FY23,Yuen23} and quantum semiparametrics \cite{TAD20}.

\smallskip

Outside the i.i.d. setting, an important practical problem is the estimation of dynamical parameters of open quantum systems undergoing Markovian evolution. An experimental setup that is particularly prevalent in quantum optics is that of open 
systems subject to continuous monitoring through indirect probing of their environment \cite{GardinerZoller,WisemanMilburn}. In the continuous time setting we deal with the mathematical framework of quantum input-output dynamics driven by quantum Wiener processes, which has strong connections with classical filtering and control theory \cite{Belavkin94,GardinerZoller,WM93,GJ09,CKS17}. In this paper we investigate the discrete time version \cite{AttalPautrat,VerstraeteCirac,Gough04}, which is connected to several areas such as collision models theory \cite{Ciccarello22}, matrix product states \cite{PerezGarciaWolfCirac,SchonSolanoVerstraeteCiracWolf}, finitely correlated states \cite{FannesNachtergaeleWerner,FannesNachtergaeleWerner2}, and Haroche's one-atom maser photon-box \cite{Haroche2}. Throughout this work we use the terminology of \emph{quantum Markov chains} (QMCs) which is closer to the statistical spirit of this work.  

\smallskip

The question of how to estimate dynamical parameters of quantum Markov processes goes back to \cite{Mab96,GW01,Berry02,Wiseman04} with other approaches including Bayesian estimation \citep{GM13,Negretti13,KM16,Ralph17,Zhang19}, filtering methods \cite{Ralph11,CG09,Six15}, quantum smoothing \citep{Tsang09,T09,Tsang10,Guevara15}; more recently the focus has been in finding the ultimate precision limits in terms of the quantum Fisher information \cite{Guta2011,Molmer14,GutaCB15,Guta_2015,Guta_2017,Genoni17}, designing realistic optimal measurement strategies \cite{Godley2023,GGG25,DayouCounting}, and exploring the connection to enhanced metrology and dynamical phase transitions \cite{GutaGL16,ARPG17,Genoni18,Ilias22}. On a more fundamental statistical level, the identifiability theory, information geometry and quantum local asymptotic normality have been studied in \cite{GK17} for \emph{irreducible} continuous-time dynamics and in \cite{GK15} for \emph{primitive} discrete time QMCs. A first result of these studies is that the space of identifiable parameters of the stationary output process is a manifold obtained by quotienting the space of dynamics by a group of system transformations which do not change the output. This is reminiscent of Petrie's result \cite{Pe69} on identifiability of certain classes of (classical) hidden Markov chains. The second result is that the output state model can be approximated in the sense of quantum Le Cam distance by a quantum Gaussian shift model, locally on the tangent space to the manifold of identifiable parameters. This is a quantum version of local asymptotic normality results for irreducible Markov chains (\cite{Ho88,HRL90} and references therein) and hidden Markov chains \cite{BR96,BRR98} (with the stronger requirement of primitivity). 

\smallskip 
In this work we extend the existing theory from primitive QMCs to \emph{irreducible} ones, thus allowing for dynamics that exhibit  periodic behaviour \cite{Wo12}. We find that this seemingly minor change leads to significantly different identification and asymptotic statistical theory: the space of identifiable parameters is not a smooth manifold anymore but rather an \emph{orbifold} \cite{ALR07} and QLAN is replaced by convergence to a \emph{mixture} of quantum Gaussian shift models. While orbifolds (or stratified spaces) do appear in classical statistics, in particular in the context of mixture models \cite{HoNguyen19}, we are not aware of a similar analysis in the hidden Markov chains setting. 
With this in mind, for reader's convenience we have included a background section on quantum mechanics, estimation and QMCs and a short introduction to orbifold theory. Below we provide a high level summary of the concepts and results followed by an outline of the paper structure.

\smallskip 

\subsection{Summary of results}
We now review the results contained in this paper, and start by briefly depicting the physical setup investigated here. We consider a 
QMC consisting of a system with Hilbert space ${\cal H} = \cc^d$ and initial state $|\varphi\rangle \in {\cal H}$, interacting successively with independent and identically prepared ‘environment units’ with Hilbert space ${\cal K}= \cc^k$, each of them being prepared in a known pure state $\ket{\chi} \in {\cal K}$. The interaction is described by a
unitary operator 
$U$ on ${\cal H}\otimes{\cal K}$ which is applied sequentially to the system and each of the environment units. Since the state of the latter is fixed, the dynamics is completely determined by the isometry 
\begin{eqnarray*}
V : {\cal H}&\rightarrow& {\cal H}\otimes {\cal K}  \\
\ket{\varphi} &\mapsto& U(\ket{\varphi} \otimes \ket{\chi})
\end{eqnarray*} 
and we will refer to this as the ``dynamics''. Physically, this type of discrete-time process underlies landmark experimental setups such as Haroche's photon-box \cite{Haroche2} but also models general continuous-time quantum open system dynamics in the Markov approximation, including gravitational wave detectors \cite{GW3} and atomic clouds magnetometers \cite{Budker2007}. 
From a classical perspective, the dynamics bears similarities to that of a Markov chain driven by a Bernoulli sequence \cite{DiaconisFreedman99}, where the incoming environmental units play the role of ``quantum coins'', while the unitary corresponds to the deterministic dynamical law governing the system's evolution. 
An important difference to the classical case is that the state of the ``quantum coin'' changes due to the interaction with the system, and the system cannot be observed directly without perturbing its dynamics. 

After $n$ time steps, joint system and output state is given by 
$$
|\Psi^{\rm s+o}(n)\rangle = 
V(n)|\varphi\rangle = V^{(n)}\cdots V^{(1)}|\varphi\rangle 
$$
where $V^{(j)}$ is the isometric embedding of the system state into that of system and the $j$-th environment unit, arising from their interaction. The right side expression implies that the system-output state is a form of \emph{matrix product  state} \cite{PerezGarciaWolfCirac} where the system plays the role of memory building up correlations between the output units, which is broadly similar to properties of probability distributions of classical (hidden) Markov processes. 
By tracing out the environment, one finds that the reduced system state (represented as a density matrix) after $n$ steps is given by
 $$
 \rho^{\rm sys}(n) = {\rm Tr}_{\rm out} \left(| \Psi^{\rm s+o}(n)\rangle\langle \Psi^{\rm s+o}(n)| \right) = \mathcal{T}_*^n (\rho_{\rm in}), \qquad 
 \rho_{\rm in} := |\varphi\rangle\langle \varphi| 
 $$
 where $\mathcal{T}_*$ is the system's one-step quantum transition operator (channel) defined as $\mathcal{T}_*(\rho) = 
 {\rm Tr}_{\rm out}(V \rho V^*)$. 
On the other hand, by tracing out the system we obtain the output state
$$
\rho^{\rm out}(n, \rho_{\rm in}) = {\rm Tr}_{\rm sys}\left(| \Psi^{\rm s+o}(n)\rangle\langle \Psi^{\rm s+o}(n)| \right). 
$$

An important feature of the input-output setup of open quantum dynamics is that, instead of directly measuring the system, one learns about the dynamics by measuring the output units after the interaction with the system. In this sense quantum Markov chains bear similarities to classical \emph{hidden} Markov chains, but in order to avoid confusion we will stick with the standard terminology and not use the term ``hidden'' in the quantum setup. 

\smallskip

The key assumption made throughout the paper is that the QMC is \emph{irreducible}, which is equivalent to the existence of a unique strictly positive state 
$\rho^{\rm ss}$ such that $\mathcal{T}_*(\rho^{\rm ss}) = \rho^{\rm ss}$, called the \emph{stationary state}. This has  a two-fold motivation: on the one hand, irreducible QMCs can be considered as building blocks of more general dynamics and cover a large number of practical examples; on the other hand, QMCs with multiple stationary states exhibit radically different statistical properties such as quadratic (Heisenberg) scaling of the quantum Fisher information \cite{GutaGL16}, and therefore need to be treated separately.
We stress that irreducibility is a strictly weaker assumption than \emph{primitivity} as it allows the transition operator $\TT_*$ to have a non-trivial \emph{period}. As we will show, this gives rise to significantly different identifiability and asymptotic estimation theory compared to that of primitive QMCs \cite{GK15}. A second assumption is that the output is observed in the stationary regime, for which we provide three theoretical motivations in section \ref{sec:identifiability}. 


\smallskip 

The quantum statistical problem under investigation is the following: given a QMC with unknown unitary dynamics, we would like to estimate the dynamics (or a parameter thereof) by performing measurements on the (stationary) output units and using the outcomes to compute an estimator. 
More specifically, we would like to 
know which dynamical parameters are identifiable, and how well they can be estimated in the large time limit regime.

To answer the first question we define an equivalence relation between irreducible dynamics with identical stationary output states and characterise the corresponding equivalence classes. Let $\mathscr{V}^{\rm irr}$ denote the set of isometries of irreducible QMCs. Two isometries $V_1, V_2$ are called \emph{output equivalent} if they have the same stationary output states, i.e. $\rho^{\rm out}_{V_1}(n) =\rho^{\rm out}_{V_2}(n)$ for all $n\geq 1$. The equivalence relation is completely characterised in Theorem \ref{thm:identification}, which we reproduce here informally for convenience. 

\smallskip 

{\bf Result 1. (output equivalence)} 
\emph{Two isometries $V_1, V_2$ are output equivalent if and only if there exists a unitary $W$ on $\mathcal{H}$ and a phase $e^{i\phi}\in \mathbb{C}$ such that 
$V_2 = e^{i\phi} (W\otimes \mathbf{1}_{\mathcal{K}}) V_1 W^*$. }

\smallskip

The space of identifiable parameters $\mathscr{P}^{\rm irr}$ is therefore the set of equivalence classes of isometries $[V]$ given by the orbit of $V$ under the action of the group $\mathscr{G}= U(1)\times PU(d) $ acting on $\mathscr{V}^{\rm irr}$, cf. equation \eqref{eq:group.action}. Passing from $\mathscr{V}^{\rm irr}$ to $\orbifirr$ 
comes at the price of giving up the manifold structure of $\mathscr{V}^{\rm irr}$ and having to deal with a quotient space, which, a priori 
is defined in a rather implicit way. This issue is resolved in Theorem  \ref{th:orbifold.structure} which shows that $\mathscr{P}^{\rm irr}$ has the structure of an orbifold, which loosely speaking means a topological space that can be locally identified with the quotient of a subset of the Euclidean space under the action of a finite group of smooth transformations. The  key ingredient is Theorem \ref{almostfree} which identifies the the stabiliser group of 
$\mathscr{G}_V$ of a dynamics $V$ with the set of pairs 
$\{(\gamma_V^k, Z_v^k): k=0,\dots, p_V-1\}$ of peripheral eigenvalues and unitary eigenvectors of the transition operator $\mathcal{T}_V$.

\smallskip 

{\bf Result 2. (orbifold structure)} \emph{The space of identifiable parameters $\mathscr{P}^{\rm irr}$ has a natural orbifold structure inherited from the action of the compact Lie group $\mathscr{G}$ over the manifold of dynamical parameters $\mathscr{V}^{\rm irr}$.}


\smallskip 


The theory of orbifolds \cite{Ca22} allows us to derive several geometric properties of $\orbifirr$: compute its dimension, express it as a disjoint union of manifolds (called \textit{singular manifolds}), describe the connected components of such submanifolds and analyse how such submanifolds sit together inside the orbifold. Using these tools we construct an atlas of orbifold charts that turns out to be a valuable tool in the study of the asymptotic statistical theory of QMCs. In brief, the tangent space $T_V(\manifirr)$ to 
$\mathscr{V}^{\rm irr}$ at a point $V$ decomposes as the direct sum 
$$
T_V(\manifirr) = T^{\rm nonid}_V\oplus T^{\rm id}_V,
$$
where $T^{\rm nonid}_V$ is the space of non-identifiable directions  consisting of tangent vectors to the orbits of $\mathscr{G}$, and $T^{\rm id}_V$ is the space of  identifiable directions consisting of bounded linear operators 
$A:{\cal H} \rightarrow {\cal H}\otimes {\cal K}$ satisfying 
$VV^* A = 0$, cf Proposition \ref{horprop}. The space $T^{\rm id}_V$ is naturally  endowed with a complex structure and carries the inner product 
$$
(A,B)_V = {\rm Tr}(\rho^{\rm ss}_V A^*B),
$$ 
which plays an important role in defining the different limit statistical models. The space $T^{\rm id}_V$ also carries a unitary representation $\pi_V$ of the discrete 
stabiliser group $\mathscr{G}_V$ which is intimately related to the unitary eigenvectors of the transition operator 
$\mathcal{T}_V$ corresponding to peripheral eigenvalues. 
The quotient of $T^{\rm id}_V$ by this action is the tangent cone to the orbifold $\mathscr{P}^{\rm irr}$  and describes an infinitesimal neighbourhood of $[V]$. 

\smallskip 

Let us now consider the problem of estimating the dynamics $V$ from (system and) output measurements. According to the quantum Cram\'{e}r-Rao bound, the covariance of any (locally) unbiased estimator is larger than the inverse of the quantum Fisher information (QFI) of the model. For the system and output state 
$|\Psi^{\rm s+o}_V(n)\rangle$ we show that the QFI $F^{\rm s+o}_V(n)$ scales linearly with respect to time for directions in the identifiable space $T_V^{\rm id}$ and is bounded for the non-identifiable 
directions in $T_V^{\rm nonid}$. Proposition \ref{prop:QFI} shows that the scaling factor is given by the real part of the inner product on $T_V^{\rm id}$ defined earlier.

\smallskip

{\bf Result 3. (QFI rate for identifiable parameters)}
Let $A =A^{\rm id}+ A^{\rm nonid}, B =B^{\rm id}+ B^{\rm nonid}$ be the decompositions into identifiable and non-identifiable components of $A,B\in T_V$. Then the QFI rate at $V$ is given by
$$
\lim_{n\to\infty} \frac{1}{n} F^{\rm s+o}_{V,n}(A,B) = 4{\rm Re} \,(A^{\rm id}, B^{\rm id})_V.
$$



Going beyond QFI, we show that the sequence of stationary system and output quantum models $|\Psi^{\rm s+o}_V(n)\rangle$ converges in the Le Cam sense (see Definition \ref{def:LECAM}) to a quantum Gaussian shift model, i.e. it satisfies quantum local asymptotic normality (QLAN). To formulate this, let us consider a fixed parameter $V_0\in \mathscr{V}^{\rm irr}$ and the system and output state $|\Psi^{\rm s+o}_{V}(n)\rangle$ corresponding to a parameter $V= V_0 + n^{-1/2}A$ where $A\in T^{\rm id}_{V_0}$ is a ``local parameter'' whose magnitude is allowed to grow slowly as $n^{\delta}$ for $0<\delta<1/2$. Let us use the above geometric data to define the limit quantum model which  consists of Gaussian (coherent) states $ |{\rm Coh}(A)\rangle := W(A)|\Omega\rangle$ of a continuous variables system, where $W(A)$ are Weyl operators of the canonical commutations relations algebra $CCR(T^{\rm id}_{V_0}, (\cdot, \cdot)_{V_0})$ defined by 
$$
W(A)W(B) = e^{-i {\rm Im} (A,B)_{V_0}} W(A+B), \qquad A,B\in T^{\rm id}_{V_0}
$$ 
and $|\Omega\rangle$  is the ``vacuum'' state with characteristic function  
$$
\langle \Omega | W(A) |\Omega \rangle = e^{-\frac{(A,A)_{V_0}}{2}}.
$$ 
The following result shows the Le Cam convergence of the system and output models to the Gaussian shift model over growing local neighbourhoods (cf. Theorem \ref{thm:SOlan}).

\smallskip 

{\bf Result 4. (QLAN for system and output state)} For any given $V_0\in \mathscr{V}^{\rm irr}$, the system and output state satisfies QLAN. Concretely, for $\delta$ small enough, there exist quantum channels $\mathcal{F}^{\rm s+o}_n, \mathcal{B}^{\rm s+o}_n$ such that 
\begin{eqnarray*}
&&
\sup_{\|A\|\leq n^\delta} \| \mathcal{F}^{\rm s+o}_n (\rho^{\rm s+o}_{A,n}) - G(A)\|_1 = o(n^{-2\delta}),\\  
&&
\sup_{\|A\|\leq n^\delta} \| \rho^{\rm s+o}_{A,n} -\mathcal{B}^{\rm s+o}_n( G(A))\|_1 = o(n^{-2\delta}),
\end{eqnarray*}
where $\rho^{\rm s+o}_{A,n} = 
|\Psi^{\rm s+o}_{V_0+ n^{-1/2}A}(n)\rangle \langle \Psi^{\rm s+o}_{V_0 + n^{-1/2}A}(n)| $ and $G(A) = |{\rm Coh}(A)\rangle \langle {\rm Coh}(A)|$.

\smallskip 

Note that the QLAN result is operational in the sense that it prescribes physical transformations connecting the models, and strong in the sense that it provides convergence rates for the deficiency between the models, while the latter are allowed to grow locally. Note also that the same result holds for ``full'' neighbourgoods of $V_0$ including non-indentifiable directions, 
but the limit remains dependent only on the identifiable component.


\smallskip 

Let us consider now the stationary output state $\rho^{\rm out}_V(n)$. Result 2. showed that the state depends only on the equivalence class $[V]$ so the model should be seen as being parametrised by points in the space $\mathscr{P}^{\rm irr}$. In Theorem \ref{thm:nLAN} 
we show that the sequence of stationary output quantum models $\rho^{\rm out}_{V}(n)$ converges in the Le Cam sense to a limit quantum model, locally in the space of identifiable parameters. Remarkably, unlike the case of primitive (aperiodic) dynamics, the limit model is \emph{not} a quantum Gaussian shift model but rather a mixture of Gaussian models which are related to each other by the unitary action of a discrete group. In particular, quantum local asymptotic normality does not hold in the vecinity of periodic QMCs. To define the limit model consider the algebra $CCR(T^{\rm id}_{V_0}, (\cdot, \cdot)_{V_0})$ introduced earlier and define the mixture of coherent states
$$
\rho(A):=\frac{1}{p_{[V_0]}}\sum_{g\in \mathscr{G}_{V_0}}
|{\rm Coh}(U_{V_0}(g) A)\rangle \langle {\rm Coh}(U_{V_0}(g) A)|,
$$
where $U_{V_0}(g)$ is the unitary representation of $\mathscr{G}_{V_0}$ on $\tsident_{V_0}$. Physically, this amounts to averaging the coherent state $|{\rm Coh}(A)\rangle\langle {\rm Coh}(A)|$ by the action of the second quantisation of the unitary group $\{U_{V_0}(g)^k\}_{k=0}^{p_{[V_0]}}$ on $T^{\rm id}_{V_0}$. Note that this model is different from those obtained in quantum versions of local asymptotic mixed normality \cite{Benoist18} where one has access to the label of the individual components of the mixture. Theorem \ref{thm:nLAN} is informally stated as follows.

\smallskip 


{\bf Result 5. (Limit mixture of quantum Gaussian shift models)} Let $V_0\in \manifirr$ be fixed and consider the sequence of local output models 
$\rho^{\rm out}_{A,n}:= \rho^{\rm out}_{[V_0 + n^{-1/2} A]}(n)$ with local parameter $A\in T^{\rm id}_{V_0}$. 
Then, for $\delta$ small enough, there exist quantum channels $\mathcal{F}^{\rm out}_n, \mathcal{B}^{\rm out}_n$ such that 
\begin{eqnarray*}
&&
\sup_{\|A\|\leq n^\delta} \| \mathcal{F}^{\rm out}_n (\rho^{\rm out}_{A,n}) - \rho(A)\|_1 = o(n^{-2\delta}),\\  
&&
\sup_{\|A\|\leq n^\delta} \| \rho^{\rm out}_{A,n} -\mathcal{B}^{\rm out}_n( \rho(A))\|_1 = o(n^{-2\delta}).
\end{eqnarray*}

\smallskip 


Results 4 and 5 put the basis of an asymptotic estimation theory for QMCs. In classical asymptotics, it is well known that such strong convergence results can be used to derive asymptotic minimax rates of convergence \cite{LeCam}, and 
similar techniques have been demonstrated in the quantum setup in \cite{LAN1,BGM18,YCH19} for i.i.d. models. A detailed analysis of the estimation theory is left for a future work. 

\vspace{2mm}

{\bf Result 6. (Worked out example)}
In the final part of the paper we apply the general theory developed here to the class of QMCs with two dimensional system and environmental units, i.e. $\mathcal{H} = \mathcal{K} = \mathbb{C}^2$. The space of isometries $\mathscr{V}^{\rm irr}$ has dimension 12 and its quotient by the group $\mathscr{G}$ is an orbifold $\mathscr{P}^{\rm irr}$ made up of an 8 dimensional manifold $\mathscr{P}^{\rm prim}$ and a 4 dimensional submanifold $\mathscr{P}_{2,4}$ of dynamics with period $2$. For the purpose of illustration, a cartoon version of this geometry is represented in panels a) and b) of Figure \ref{fig:qubit}. In the neighbourhood of primitive dynamics the output model satisfies QLAN with a 4-modes limit quantum Gaussian shift model. In contrast, the asymptotic behaviour near periodic points has a richer, more interesting structure. The identifiable tangent space $T^{\rm id}_V$ at a periodic point $V$ decomposes into 2 subspaces 
$\mathcal{V}_0,\mathcal{V}_1$ of (complex) dimension 2 which are eigenspaces of the stabiliser with eigenvalues $1$ and respectively $-1$. This means that the limit model is a product between a 2-mode Gaussian shift model corresponding to movement within the periodic submanifold $\mathscr{P}_{2,4}$, and a 2-mode Gaussian mixture model for direction away from the singular manifold. We analyse 3 one-dimensional models passing through a periodic point, whose (identifiable) tangent vectors lie in $\mathcal{V}_1$, $\mathcal{V}_0$ or neither (labelled as (I), (II), (III) in Figure \ref{fig:qubit} b)). In the latter two cases we provide output measurement and estimation procedures which ``localise'' the parameter at standard rate; in the third case we show that the signal to noise ratio (SNR) of average statistics for simple repeated measurements vanishes near the periodic point, while the SNR for certain \emph{joint} measurements on pairs of environmental units remains non-zero in this limit. This analysis suggest that parameters can be estimated at standard rate in a preliminary localisation stage, and thus provides the basis for a future investigation into the asymptotic minimax theory of QMCs, combining preliminary estimation with the asymptotic approximation results derived here, in the spirit of the works \cite{YCH19,LAN6} in the i.i.d. setting. Figure \ref{fig:qubit} c) illustrates the three corresponding limit Gaussian and mixed Gaussian statistical experiments.

\begin{figure*}[!ht]
    \centering
   \includegraphics[width=\linewidth]{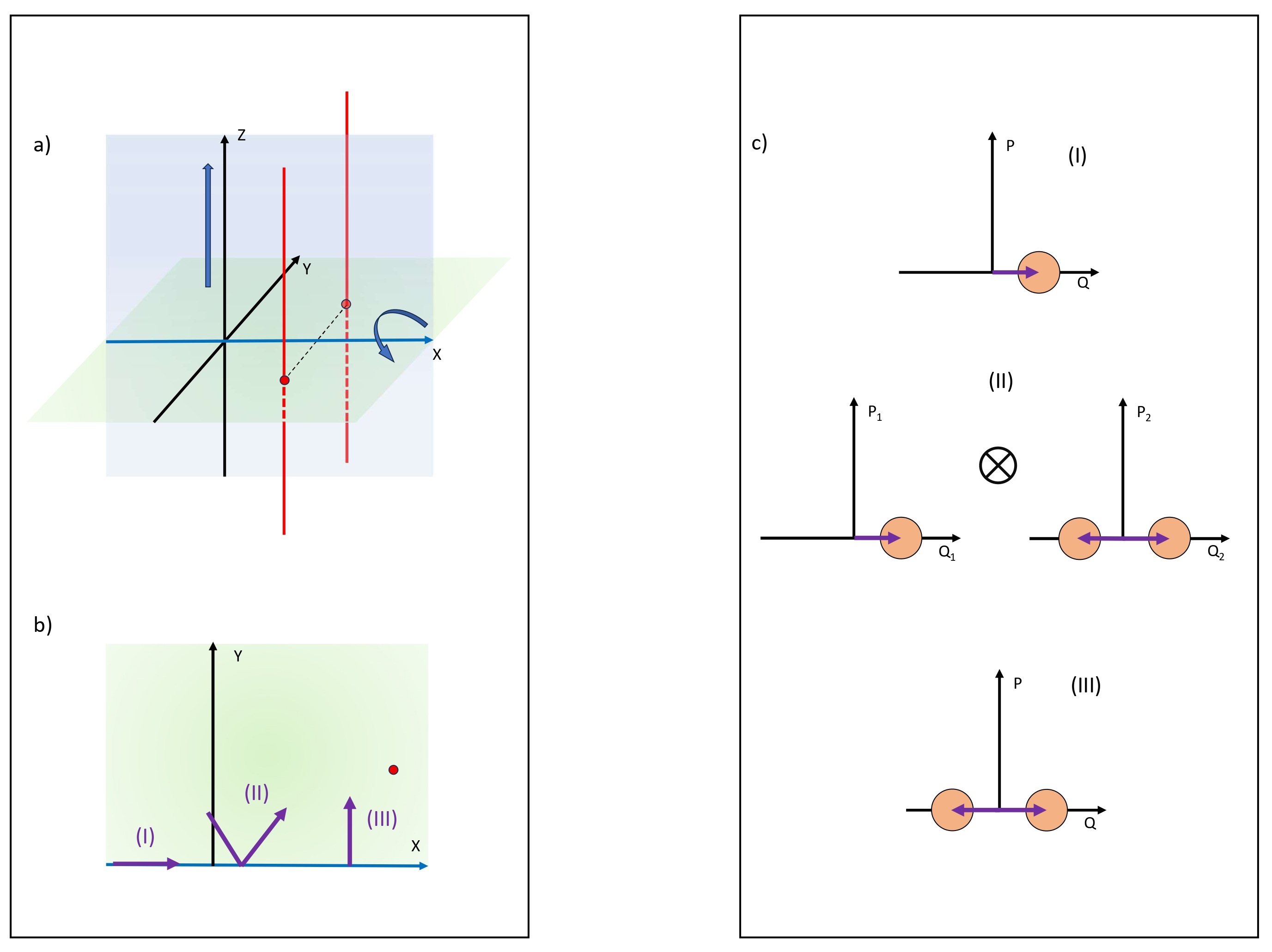}
    \caption{Conceptual representation of the local geometry of stationary QMCs with two dimensional system and ancillas, and limit quantum statistical models for 3 one-parameter models. 
    Panel a) depicts the (12 dimensional) manifold of isometries $\mathscr{V}^{\rm irr}$ and the action of the symmetry group 
    $\mathscr{G}$. Locally, the action consists of translations along the $Z$ axis together with the $\mathbb{Z}_2$ action of reflections around the $XZ$ plane that represents periodic QMCs. The blue arrows correspond to the transformations generating the group, while the two vertical red lines constitute a single arbitrarily chosen orbit. At every point, the non-identifiable tangent space 
    $\mathcal{T}^{\rm nonid}$  consists of vectors in the $Z$ direction while $\mathcal{T}^{\rm id}$ is parallel to the $XY$ plane. The latter carries the action of the stabiliser group 
    $\mathbb{Z}_2$ and decomposes into the direct sum of eigenspaces ${\cal V}_0$ and ${\cal V}_1$  consisting of vectors along the $X$ and respectively the $Y$ axis.  %
    Panel b) represents the orbifold obtained as the quotient of $\mathcal{V}^{\rm irr}$ under the action of $\mathscr{G}$, which can be identified with the $XY$ plane folded along the $X$ axis, represented as a half plane including the `edge' axis $X$;  the axis corresponds to the (4 dimensional) manifold $\mathscr{P}_{2,4}$ of equivalence classes of periodic isometries,  while the plane represents the (8 dimensional) manifold $\mathscr{P}^{\rm prim}$ of equivalence classes of primitive QMCs. 
    The red dot corresponds to the orbit represented by red lines in $a)$. We analyse the asymptotics of three \emph{one-dimensional models} with distinct geometries and limit models. In the parameter space these are the following lines in the $XY$ plane: a line (I) along the $X$ axis, a line (II) which is non-orthogonal to either $X$ or $Y$ axis, and a line (III) parallel to the $Y$ axis. For clarity, the lines are not represented in panel a) but their orbifold projections  are represented in panel b) by purple arrows: (I) lies inside the singular manifold, (II) is the `broken' line that crosses the singular manifold in one point, and (III) is a half line in the $Y$ direction, exhibiting non-identifiability.
    Panel $c)$ reports a representation in the phase space of the limit quantum statistical models corresponding to the three one-parameter models, locally around a point at the intersection with the $X$ axis (periodic QMC). Orange circles represent quantum coherent states with standard vacuum covariance.
    (I) is a pure quantum Gaussian shift model where the mean moves along the $Q$ axis. (II) is the tensor product of a pure Gaussian shift model as the one considered in (I) and a mixture of two coherent states with opposite means: the means move together when the local parameter changes. Finally, the third model is described by a mixture of coherent states lying on the $Q$ axis with opposite means. We remark that the local parameter is identifiable in (I) and (II), but not in (III).}
    \label{fig:qubit}
\end{figure*}

\subsection{Structure of the paper}

In section \ref{sec:back} we provide a brief  introduction to concepts and results used in this work, such as fundamental  notions of quantum mechanics, quantum estimation theory and quantum Markov chains. 
Section \ref{sec:identifiability} deals with the identifiability   problem for irreducible QMCs. Three statistical tasks are presented and the corresponding notions of identifiability are defined in Definitions \ref{def:equiv}, \ref{def:macro.equivalence} and \ref{def:waek.equivalence}; these are showed to be equivalent using Lemma \ref{lem:ltflu}, which is of independent interest. The equivalence class of QMCs with identical stationary outputs is completely characterized in Theorem \ref{thm:identification}. In Section \ref{sec:global} we further elaborate on the identifiability theory and show that the equivalence classes are orbits of
a group acting smoothly on the manifold of irreducible QMCs  (Theorem \ref{almostfree}); moreover, in the same Theorem we show that the action satisfies the requirements for the quotient space to be endowed with the structure of an orbifold. Section \ref{subsec:orbi} contains a reminder of all the notions and results about orbifolds used in the paper. 
In Section \ref{subsec:idatl} we construct an orbifold atlas for the parameter space which turns out to be a convenient parametrization for the results in Section \ref{sec:LAN}. Section \ref{subsec:orbi} treats the decomposition of $\orbifirr$ into submanifolds and discusses some of their properties (Proposition \ref{prop:sd}).

Theorems \ref{thm:SOlan} and \ref{thm:nLAN} in Section \ref{sec:limit_results} provide local approximations of the statistical models corresponding to the system and output state and the stationary output state, respectively. Corollary \ref{coro:plan} determines exactly those cases in which LAN holds. Section \ref{sec:example} applies the general theory to the class of QMCs with two-dimensional system and environment unit.

\section{Background and preliminary results} \label{sec:back}

Quantum theory employs the mathematical framework of linear operators on complex Hilbert spaces to describe physical phenomena at the microscopic scale of atoms and photons. One of the key features of the theory is that quantum measurements are \emph{intrinsically probabilistic}, so that the problem of learning the state of a system is fundamentally of a statistical nature. Such problems have become ubiquitous in Quantum Technology where experimenters use statistical inference to reconstruct parameters of quantum states and devices for validation purposes, but also to estimate unknown physical quantities with high precision. 
In this section we give a brief introduction to the basic concepts and techniques of the field, and highlight a new technical tool on convergence of quantum models which has a wider applicability.


\subsection{Quantum mechanics primer}\label{sec:QM}
In this section we give a condensed introduction to the basic notions of quantum mechanics leading to the theory of quantum channels and Gaussian states which play a key role later on.
\subsubsection{Postulates of quantum mechanics}
In quantum mechanics, each system (e.g. particle, electromagnetic field) is described mathematically in terms of a separable, complex \emph{Hilbert space}
$\mathcal{H}$, and certain linear operators acting on it. 
In this paper we adopt the physicists (Dirac) notation, whereby a vector 
$\psi$ in a Hilbert space $\mathcal{H}$ is denoted by using the ``ket'' symbol $|\psi\rangle$, while the ``bra'' vector $\langle \phi|$ denotes the adjoint $\phi^*$, such that the inner product between $|\phi\rangle$ and $|\psi\rangle$ is denoted $\langle \phi|\psi\rangle$.


\vspace{2mm}

We denote by $L^\infty(\mathcal{H},{\cal K})$ the space of bounded operators defined on $\mathcal{H}$ and taking values in ${\cal K}$; if ${\cal H}={\cal K}$, we will use the simpler notation $L^\infty({\cal H})$. We consider $L^\infty({\cal H})$ equipped with the operator norm $\|X\|_\infty = \sup_{\|\psi\|_{\cal H}=1}\|X\psi\|_{\cal K}$ and we let $L^1(\mathcal{H})$ be the space of trace-class operators, i.e. bounded operators $\tau$ satisfying $\|\tau\|_1 := \tr(|\tau|)<\infty$ where 
$|\tau| = \sqrt{\tau^*\tau}$ is the absolute value (operator)  and $\tau^*$ is the adjoint of $\tau$. Any bounded linear functional on $L^1(\mathcal{H})$ takes the form 
$$
\tau \in L^1(\mathcal{H})\mapsto
\tr(\rho X) \in \mathbb{C}
$$
for some bounded operator $X$, so that $L^{\infty}(\mathcal{H})$ is the dual of  $L^1(\mathcal{H})$. 
This fundamental duality is the quantum analogue of the duality encountered in probability theory between the space of absolutely integrable functions $L^1(\Omega, \Sigma, \mu) $ and that of bounded measurable functions $L^\infty(\Omega, \Sigma, \mu)$, where $(\Omega, \Sigma,\mu)$ is a probability space. 

\vspace{2mm}

\emph{Quantum states.} The \emph{state} of a quantum system with Hilbert space $\mathcal{H} $ incorporates information about its preparation and is represented mathematically by a \emph{density matrix}, i.e. an operator in $ L^1(\mathcal{H})$ which is   positive and has trace one. The space of states $S(\mathcal{H}) \subset L^1(\mathcal{H})$ is convex and its extremal elements are the \emph{pure states} represented by one-dimensional projections $P_\psi = |\psi\rangle\langle \psi|$ where $|\psi\rangle $ is a unit vector. When dealing with pure states we will often work with the vectors themselves rather than the projections and refer to $|\psi\rangle$ as being the ``state'', or  the ``wave function'' of the system. In general, a state $\rho\in S(\mathcal{H})$ which is not pure is called \emph{mixed} and can be written as a convex combinations of pure states. One such representation is given by the spectral decomposition
$$
\rho = \sum_{i} \lambda_i |e_i\rangle\langle e_i|
$$
where $\lambda_i$ are the eigenvalues of $\rho$ and $|e_i\rangle$ are the corresponding eigenvectors. The positivity and trace properties of $\rho$ imply that $\{\lambda_i\}_i$ is a discrete probability distribution. 

\vspace{2mm}

\noindent
\emph{Observables.} 
In quantum mechanics, the system's observables are represented by \emph{selfadjoint operators} acting on its Hilbert space $\mathcal{H}$. According to the Spectral Theorem, any observable has a spectral decomposition of the form
\begin{equation}
\label{eq:spectral.theorem}
A= \int_{\sigma(A)} \lambda P(d\lambda)
\end{equation}
where $\sigma(A)\subset \mathbb{R}$ is the spectrum of $A$ and $P(d\lambda)$ is the associated \emph{projection valued measure} 
whose elements are orthogonal projections. More precisely, 
$P:\Sigma \to L^\infty(\mathcal{H})$ is a map from the Borel $\sigma$-algebra of $\sigma(A)$ to bounded operators such that
\begin{enumerate}
\item 
$P(E) =\int_E P(d\lambda)$ is an orthogonal projection for each $E\in \Sigma$;
\item 
$P$ is $\sigma$-additive, i.e. 
$P\left(\bigcup_i E_i\right) =\sum_i P(E_i)$ for disjoint sets $\{E_i\}_i$ in $\Sigma$;
\item 
$P(\emptyset) =0$ and $P(\sigma(A))  =\mathbf{1}$.
\end{enumerate}
For finite dimensional spaces, observables have discrete spectrum given by the set of eigenvalues, and the spectral projections are the corresponding eigenprojectors.

\vspace{2mm}

The duality between $L^1(\mathcal{H})$ and $L^\infty(\mathcal{H})$ provides us with a quantum notion of expectation. 
For any state $\rho$ and any bounded observable $A$, we define the expected value of the latter with respect to the former by $\tr(\rho A)$. Its full probabilistic interpretation is described below in the context of measurements.


\vspace{2mm}

\noindent 
\emph{Measurements.}
Measurements provide the link between the quantum world of ``wave functions'' and the classical one of stochastic measurement outcomes. In its most general form, a measurement with outcomes in a measure space $(\Omega, \Sigma,\mu)$ is described by a bounded linear transformation 
$$
\mathcal{M}_*:L^1(\mathcal{H})
\to L^1(\Omega,\Sigma, \mu)
$$
which maps a state $\rho$ into the probability density $p^{\mathcal{M}}_\rho:=  \mathcal{M}_*(\rho)$ with respect to the given reference measure $\mu$. When the measurement is performed on a system in state 
$\rho$, one obtains a sample $X$ from the distribution $\mathbb{P}^{\mathcal{M}}_\rho$ with density 
$p^{\mathcal{M}}_\rho$. For each $E\in \Sigma$, the dual map 
$\mathcal{M}:L^\infty(\Omega,\Sigma, \mu)\to L^{\infty}(\mathcal{H})$ defines the positive operator $M(E) := \mathcal{M}(\chi_E)$ where $\chi_E$ is the characteristic function of $E$. The collection 
$\{M(E): E\in \Sigma\}$ is called a \emph{positive operator valued measure} (POVM) and determines the probability distribution of the measurement outcomes 
$$
\mathbb{P}^{\mathcal{M}}_\rho(E)  =\tr(\rho M(E)).
$$

Note that, as a map from states to probability densities, the measurement can be regarded as a quantum-to-classical randomisation, in analogy to  classical-to-classical randomisations 
$R:L^{1}(\Omega_1, \Sigma_1, \mu_1)\to L^{1}(\Omega_2, \Sigma_2, \mu_2)$ which map probability densities to probability densities.

The most common type of measurement is that associated to an observable. 
When measuring the observable $A$ with spectral decomposition \eqref{eq:spectral.theorem} on a system in state $\rho$, we obtain a random outcome $a\in \sigma(A)$ with probability distribution
$
\mathbb{P}_{\rho}( a \in E)  =
    \tr( \rho P(E)),
$ $E\in \Sigma$.
In this case the POVM is given by the projection valued measure associated to the spectral decomposition of $A$. In particular, the expected value of the outcome $X$ is 
$$
\mathbb{E}_\rho (X)= 
\int_{\sigma(A)} a \tr(\rho P(da))  =  \tr(\rho A).
$$




\vspace{2mm}

\noindent 
\emph{Unitary transformations.}
The time evolution of a closed quantum system $\mathcal{H}$ is determined by a unitary operator 
$U\in L^{\infty}(\mathcal{H})$, so that an initial vector state $|\psi\rangle$ is transformed into the final state $U|\psi\rangle$. On the level of density matrices, this action becomes that of unitary conjugation
\begin{eqnarray*}
\mathcal{U}_* : L^1(\mathcal{H}) &\to& 
L^1(\mathcal{H})\\
\rho &\mapsto& U\rho U^*.
\end{eqnarray*}
This description of the evolution is often called the \emph{Schr\"{o}dinger picture}, where states change and observables are fixed in time. 
An alternative view is the \emph{Heisenberg picture} where the system's state is always fixed but observables evolve according to the dual
map  
\begin{eqnarray*}
\mathcal{U}
:L^{\infty}(\mathcal{H})&\to& L^{\infty}(\mathcal{H})\\
 X &\mapsto& U^* X U
\end{eqnarray*}
The two pictures are equivalent in the sense that they both described the same expected values of observables and measurement probabilities
$$
\tr\left[ \mathcal{U}_*(\rho) A \right]  = 
\tr\left[\rho~ \mathcal{U}(A) \right]  , \qquad A\in L^{\infty}(\mathcal{H}), \rho \in L^1(\mathcal{H}).
$$
This is similar to the treatment of dynamics in stochastic processes where the evolution can be described in terms of time changing random variables or probability distributions. 

\vspace{2mm}

\noindent
\emph{Composite systems.}
The Hilbert space of a composite system  consisting of a finite number of sub-systems with Hilbert spaces $\mathcal{H}_1, \dots ,\mathcal{H}_k$, is given by the \emph{tensor product} $\mathcal{H} = \bigotimes_{i=1}^k\mathcal{H}_i$.
States of the form 
$\rho=\bigotimes_{i=1}^k\rho_i$ are called \emph{product states} and correspond to individual systems being prepared independently in states 
$\rho_1, \dots , \rho_k$. Convex combinations of such states are called \emph{separable} while all other states are called \emph{entangled}. 
In particular, pure (vector) states are either of the product form $\bigotimes_{i=1}^k |\psi_i\rangle $ or they are entangled.

States over multipartite systems are quantum analogues of joint distributions of several random variables. If $\rho_{12}$ is a density matrix on the bipartite space $\mathcal{H}_1 \otimes \mathcal{H}_2$ then its \emph{marginals} represent the states of the subsystems $\mathcal{H}_1$ and  $\mathcal{H}_2$ and are given by the partial traces $\rho_1:= \tr_2 (\rho_{12})$ and $\rho_2:=\tr_1 (\rho_{21})$. 
As in the classical case, the marginals determine the expectations of observables corresponding to individual sub-subsystems: indeed if $A\in L^\infty(\mathcal{H}_1)$ and $B\in L^\infty(\mathcal{H}_2)$ then their expectations can be computed as 
$$
\tr_1(\rho_1 A) = 
\tr_{12}(\rho_{12} (A\otimes \mathbf{1}_2)),\qquad
\tr_2(\rho_2 B) = 
\tr_{12}(\rho_{12} (\mathbf{1}_1\otimes B)).
$$
where $A\otimes \mathbf{1}_2$ denotes the ampliation of $A$ by the identity operator $\mathbf{1}_2$ on $\mathcal{H}_2$. 



\begin{center}
\begin{figure}
    \centering
    \includegraphics[width=0.7\linewidth]{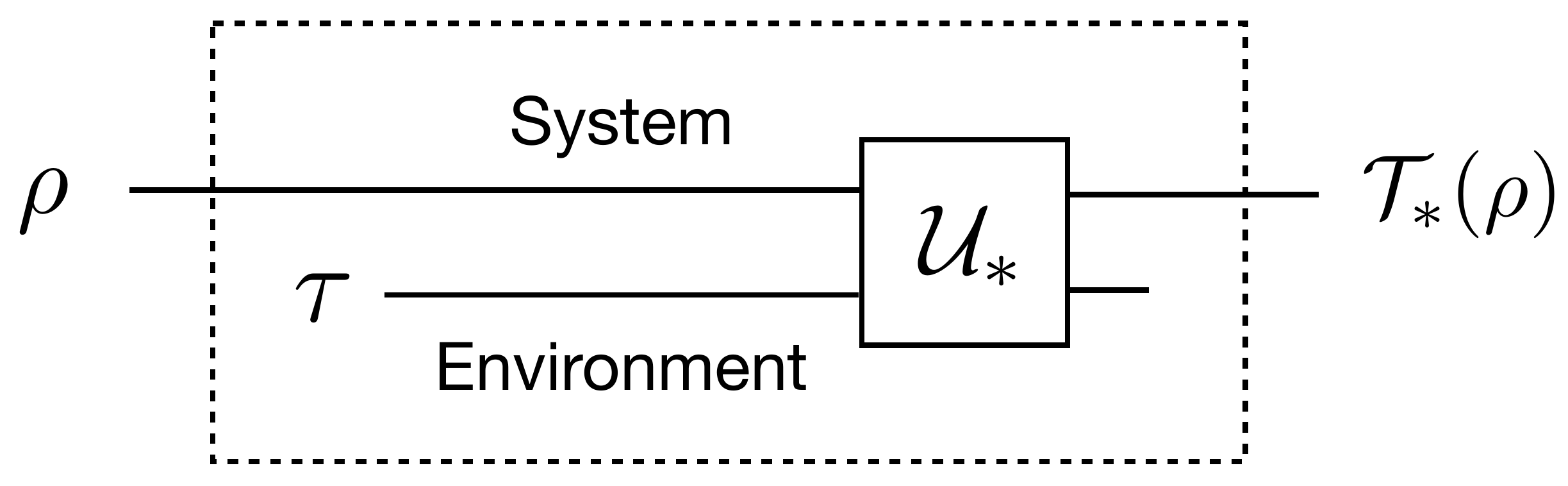}\hspace{3mm}
       \caption{Quantum channel as a black box transformation mapping the initial system state $\rho$ to final state $\mathcal{T}_*(\rho)$ through the unitary interaction with an environment system in a fixed initial state $\tau$. }
    \label{fig:channel}
\end{figure}   
\end{center}

\subsubsection{Quantum channels}
In practice, quantum systems are never perfectly isolated but interact with their environment, so even though the full system-environment evolution may be unitary, the \emph{marginal} state of the system does not transform unitarily. Instead, such transformations are described by a class of maps called \emph{quantum channels} on which we focus our attention at the end of this section. These can be seen as quantum-to-quantum randomisations but also as quantum analogue of transition matrices for Markov chains, as we will see later. 

\vspace{2mm}

We introduce this notion through the  
example of a system $\mathcal{H}$ interacting with the 
environment $\mathcal{K}$ via a unitary $U$ acting on $\mathcal{H}\otimes \mathcal{K}$, cf. Figure \ref{fig:channel}. If the two systems are independent and have initial states 
$|\varphi\rangle\in \mathcal{H}$ and respectively $|\chi\rangle\in \mathcal{K}$  
then the final joint state is
$
|\Psi\rangle:= U|\varphi \otimes \chi\rangle. 
$
By expanding with respect to 
an orthonormal basis $\{|1\rangle, \dots , |k\rangle\}$ in $\mathcal{K}$ we write 
$$ 
|\Psi\rangle
= \sum_{i=1}^k K_i|\varphi \rangle \otimes |i\rangle 
$$
where $K_i = \langle i|U|\chi\rangle \in L^{\infty}(\mathcal{H})$ are operators satisfying $\sum_{i=1}^k K_i^*K_i = \mathbf{1}_\mathcal{H}$, and the inner product is with respect to $\mathcal{K}$. 
The marginal system state (density matrix) after the evolution is 
$$
\rho^{{\rm sys}\prime}= \tr_{\mathcal{K}} (U (\rho^{\rm sys}\otimes\tau) U^*)=
\sum_{i=1}^k K_i |\varphi\rangle\langle\varphi| K_i^*
$$
where $\rho^{\rm sys}:= |\varphi\rangle \langle\varphi|$ and $\tau :=|\chi\rangle \langle \chi|$ are the initial states of system and environment. 
We now extend this argument to arbitrary system states and express the state transformation as the map which in the Schr\"{o}dinger picture is given by (cf. Figure \ref{fig:channel})
\begin{eqnarray*}
\mathcal{T}_*&:& L^1(\mathcal{H}) \to L^1(\mathcal{H}) \\
\mathcal{T}_* &:& \rho \mapsto 
\tr_\mathcal{K}(U(\rho\otimes \tau)U^*) = 
\sum_{i=1}^k K_i \rho K_i^*.
\end{eqnarray*}
Remarkably, it turns out that \emph{any} physical transformation is of the above form. Indeed, if we consider a system undergoing a generic black-box transformation 
$\mathcal{T}_*: L^1(\mathcal{H})\to L^1(\mathcal{H})$, then on physical grounds it is natural to require $\mathcal{T}_*$ to map states into states and to preserve convex combinations; this means that $\mathcal{T}_*$ should be linear, positive ($\mathcal{T}_*(A)\geq 0$ for $A\geq 0$) and trace preserving. Moreover, by applying a similar argument to the composite system $\mathbb{C}^n\otimes \mathcal{H}$ where the $n$-dimensional auxiliary system undergoes the identity transformation $\mathcal{I}_{n*}$, we conclude that 
$\mathcal{I}_{n*}\otimes \mathcal{T}_*$ must be positive for all $n \in \mathbb{N}$. A map $\mathcal{T}_*$ with this property is called \emph{completely positive} (CP). This requirement is a strictly stronger than positivity as it can be shown that certain maps (e.g. transposition) are positive but not completely positive. 

The following theorem \cite{Kraus71,Choi75} shows that all physical state transformations can be realised by coupling the system with an ``environment'' via a unitary transformation and tracing out the environment, as in our motivating example, cf. Figure \ref{fig:channel}.
\begin{theorem}[{\bf Kraus Theorem}]
\label{th:Kraus}
A linear map $\mathcal{T}_*: L^1(\mathcal{H})\to L^1(\mathcal{H})$ is completely positive and trace preserving if and only if it is of the form
$$
\mathcal{T}_*(\rho) = 
\sum_{i=1}^{\infty} K_i \rho K_i
$$
for a set of operators $ K_i \in L^\infty(\mathcal{H}) $ satifying $\sum_{i=1}^{\infty} K_i^* K_i =\mathbf{1}_{\cal H}$ (where the convergence holds in the strong operator topology). Such a map (and its dual $\mathcal{T}: L^\infty(\mathcal{H})\to L^\infty(\mathcal{H})$) is called a quantum channel.
\end{theorem}

\subsubsection{Continuous variables and Gaussian states} \label{cvsystems}

The term continuous variables (CV) system refers to a quantum system whose fundamental observables are quantum mechanical versions of canonical coordinates and momenta associated to degrees of freedom of a  classical mechanical system. 
In a physical context, these observables are often position and momentum operators of free  particles, or electric and magnetic components of a number of monochromatic modes of the electromagnetic field. In a quantum statistical  context, such systems emerge from the asymptotic analysis of estimation problems where the limit model is described by a CV system whose state is a quantum analogue of a  Gaussian distribution. We refer to \cite{Ho82,Pa92} for the material presented below.

\vspace{2mm}

We start with an arbitrary complex Hilbert space $(X, \langle\cdot|\cdot\rangle)$ of finite dimension $k$, which we refer to as the one-particle space.
If we regard $X$ as a real linear space of dimension $2k$ (called the \emph{phase space}) we can define the real inner product $\beta$, and the bilinear \emph{symplectic form} $\sigma$ as
\begin{align*}
\beta(x,y) &= {\rm Re}\langle x|y\rangle, &\sigma(x,y) &= {\rm Im}\langle x| y\rangle.
\end{align*}
Now we define the \emph{symmetric Fock space}
$$
\mathfrak H(X)=
\mathbb{C}|\Omega\rangle \oplus\bigoplus_{n=1}^\infty  X^{\otimes_s n },
$$
where $X^{\otimes_s n}$ denotes the symmetric subspace of $X^{\otimes n}$, often called the $n$-particle space, and 
$\mathbb{C}|\Omega\rangle$ is a a one-dimensional space generated  by the vector $|\Omega\rangle$ called the  \emph{vacuum}. For each $|x\rangle \in X$ we define the \emph{exponential vectors}
$$
|e^{x}\rangle = \bigoplus_{n=0}^\infty \frac{1}{\sqrt{n!}} |x\rangle^{\otimes n}\in \mathfrak H(X),
$$
where we denote $|x\rangle^{\otimes 0}=|\Omega\rangle$ (so that $|e^{\bf 0}\rangle = |\Omega\rangle$). The exponential vectors span the Fock space $\mathfrak H(X)$, and satisfy
$\langle e^{x}|e^{y}\rangle=e^{\langle x|y\rangle}$ for all $x,y\in X$. We further define the unitary  \emph {Weyl operators} by their action on exponential vectors as
\begin{equation}
\label{eq:Weyl.operator.action}
W(x)| e^{y}\rangle = e^{-\frac 12\|x\|^2-\langle x|y\rangle} |e^{x+y}\rangle, \qquad x,y\in X.
\end{equation}
From this it follows that 
\begin{equation}\label{CCR2}
W(x)^*=W(-x), \quad W(x+y)=e^{i \sigma(x, y)}W(x)W(y), \qquad x,y\in X,
\end{equation}
which implies that $x\mapsto W(x)$ is a projective representation of the group of phase space translations on $\mathfrak H(X)$. 

\vspace{2mm}

While the Weyl operators implement phase space translations, we can also represent phase space rotations by means of \emph{second quantisation} operators. If $U$ is a unitary on $X$, its second quantisation
$\Gamma[U]$ is defined as a unitary operator on $\mathfrak{H}(X)$ with the following action on exponential vectors:
$$
\Gamma[U]| e^{x}\rangle = |e^{Ux}\rangle.
$$
In particular, the action of $\Gamma[U]$ restricted to the $n$-particle space is given by $U^{\otimes n}$.
This induces an automorphism $\alpha_U$ on $L^{\infty}(\mathfrak H(X))$ defined by
\begin{equation}
\label{eq:Gauge.automorphism}
A \in L^{\infty}(\mathfrak H(X))\mapsto \alpha_U(A) = \Gamma[U]A\Gamma[U]^*
\end{equation}
In particular,  
$\alpha_U(W(x)) = W(U x)$ for all $x\in X$, and  since $\Gamma[U]|\Omega\rangle=|\Omega\rangle$, we obtain
$$
\Gamma[U]W(x)|\Omega\rangle= W(Ux)|\Omega\rangle.
$$

\vspace{2mm}

For each $x\in X$, consider the 
unitary representation 
$\mathbb{R}\ni t\mapsto W(t x)$ of the abelian group $\mathbb{R}$; by Stone-von Neumann Theorem, the group has a selfadjoint  generator $Z(x)$, 
so that 
$W(t x) = \exp(-it Z(x))$. The operators $Z(x)$ are called \emph{quadratures}, and are the quantum analogues of the classical phase space position and momentum variables. However, the quadratures do not commute in general and instead satisfy the following \emph{canonical commutation relations} (CCR) which can be derived from \eqref{CCR2}
\begin{equation}
\label{eq:canonical.commutation.relations}
Z(x) Z(y) - Z(y)Z(x) = 2i \sigma(x,y)\mathbf 1.
\end{equation}

\vspace{2mm}

Let $X=X_1\oplus X_2$ be a decomposition of $X$ into two orthogonal closed subspaces and note that the map
\begin{eqnarray*}
|e^{x_1}\rangle \otimes |e^{x_2}\rangle &\mapsto&
   |e^{x_1\oplus x_2}\rangle  
\end{eqnarray*}
can be extended linearly to a unitary $V_{X_1,X_2}: \mathfrak H(X_1)\otimes \mathfrak H(X_2)\to\mathfrak H(X_1\oplus X_2)$. Bearing this in mind, it is convenient to make the identifications
\begin{align}
\label{eq:Fock.tensor.decomposition}
\mathfrak H(X_1\oplus X_2) &= \mathfrak H(X_1)\otimes \mathfrak H(X_2), & |e^{ x_1\oplus x_2}\rangle =|e^{x_1}\rangle \otimes |e^{ x_2}\rangle,
\end{align}
where for simplicity we omitted the unitary $V_{X_1,X_2}$. From this we derive the following relation for the Weyl operators 
$$
W(x_1\oplus x_2) = W(x_1)\otimes W(x_2), \quad x_1\in X_1,x_2\in X_2.
$$
In particular, if $\{e_j: j=1, \dots, k\}$ is an orthonormal basis in $X$ then the Fock space decomposes into a tensor product of ``one-mode" Fock spaces
$$
\mathfrak{H}(X) = \otimes_{j=1}^k
\mathfrak{H}(\mathbb{C}e_j)
$$
and for each $x = \sum_{j=1}^k x_k e_j$ we have  $W(x) = \otimes_{j=1}^kW (x_j e_j)$. The associated generators pairs $ Q_j:= Z(-ie_j),P_j:= Z(e_j)$ act on the $j$th terms of the tensor product and as identity on the others, and form a basis for quadratures, playing the role of ``position'' and ``momentum'' of the $k$ modes specified by the basis. In this representation the canonical variables satisfy the commutation relations $[Q_i,P_j] =2i\delta_{ij}\mathbf{1}$, as a consequence of \eqref{eq:canonical.commutation.relations}. The corresponding classical variables are then chosen to form the symplectic basis $\{e_{1}, ie_1, e_2,ie_2, \ldots, e_k, ie_k\}$ of the (real linear) phase space $X$, which is then identified with $\mathbb R^{2k}$ so that the natural basis $\{\mathbf e_1,\ldots, \mathbf e_{2k}\}$ corresponds to the above symplectic basis. Therefore, each $x=\sum_{i=1}^k x_k e_k\in X$ can be identified with the column vector ${\bf x}=\begin{pmatrix}{\rm Re}\, x_1,{\rm Im}\, x_1, \ldots, {\rm Re}\, x_k,{\rm Im}\, x_k\end{pmatrix}^\intercal$ and it follows that
\begin{align*}
\beta(x,y) &= {\bf x}^T{\bf y}, & \sigma(x,y) &= {\bf x}^T{\bf \Omega}{\bf y}, \quad{\rm where} \quad {\bf \Omega}=\bigoplus_{j=1}^k \left(\begin{smallmatrix}0&1\\-1& 0\end{smallmatrix}\right).
\end{align*}

\vspace{2mm}

We now proceed to introduce an important class of states of a CV system, called Gaussian states.
First, we define the \emph{Wigner function} of any CV state $\rho$ as the Fourier transform
\begin{equation}
\label{eq:wigner.fct}
w_\rho({\bf y}) =\frac{1}{(2\pi)^{2k}}\int_{\mathbb R^{2k}}
e^{-i{\bf y}^T{\bf \Omega x}} \,\tr[\rho W({\bf x})] \,d{\bf x}, \qquad {\bf y} \in X,
\end{equation}
where $d\bf x$ is the $2k$-dimensional Lebesgue measure. (Note that we have used the coordinates $x\mapsto {\bf x}$ given above). The Wigner function plays the role of a ``joint quasi-probability distribution'' of the canonical operators $Z(x)$. Indeed from 
\eqref{eq:wigner.fct} together with 
$W(x) = \exp(-i Z(x))$, one finds that for each $j=1,\ldots, k$, the one-dimensional marginal of $w_\rho(\cdot)$ along the direction $e_j={\bf e}_{2j-1}$ is equal to the probability density of $Q_j$, and the marginal along $ie_j={\bf e}_{2j}$ is $P_j$. However, the caveat is that even though $w_{\rho}(\cdot)$ always formally integrates to $1$, it may not be a positive function, and therefore does not generally provide a joint probability distribution for its marginals. This reflects the fact that non-commutative quantities typically do not have joint distributions. 


Now, a state with density matrix $\rho$ is called \emph{Gaussian state of mean ${\bf x}_0 
$ and 
covariance matrix ${\bf \Sigma}$} if its Wigner function $w_\rho ({\bf x})$ is the multivariate Gaussian probability density
\begin{equation}
\label{eq:Gaussian.state}
w_\rho ({\bf x}) =\frac{1}{\sqrt{(2\pi)^{2k} {\rm Det}({\bf \Sigma})}} 
    \exp\left(-\frac{1}{2}({\bf x}-{\bf x}_0)^T{\bf \Sigma}^{-1}({\bf x} - {\bf x}_0)\right)
\end{equation}
when written in the above coordinates. We denote such states $G({\bf x}_0, {\bf \Sigma})$. While the mean ${\bf x}_0$ can be arbitrary, the covariance matrix 
satisfies the constraint
${\rm Det}({\bf \Sigma})\geq 1$ for any state $\rho$, due to the CCR \eqref{eq:canonical.commutation.relations}. This is the well-known \emph{Heisenberg uncertainty principle} for the canonical coordinates; in this context it implies that there are classical Gaussian distributions which do not correspond to any quantum Gaussian state.

A special example of Gaussian states which  play an important role in our LAN result are the \emph{coherent states} 
$$
|{\rm Coh}(x_0)\rangle :=  W(x_0)|\Omega\rangle = e^{-\frac{1}{2}\| x_0\|^2} |e^{ x_0}\rangle.
$$
One can check that this state has Wigner function
$$
w_{\rho}({\bf y}) = \frac{1}{(2\pi)^{2k}}\int_{\mathbb R^{2k}} e^{-i{\bf y}^T{\bf \Omega x}-i{\bf x}^T{\bf \Omega}{\bf x}_0} e^{-\frac 12{\bf x}^T{\bf x}}d{\bf x} = \frac{1}{(2\pi)^k}e^{-\frac 12({\bf y}-{\bf x}_0)^T({\bf y}-{\bf x}_0)},
$$
and hence 
$|{\rm Coh}(x_0)\rangle\langle {\rm Coh}(x_0)| = G({\bf x}_0, I_{2k})$.

\subsection{Introduction to quantum estimation and local asymptotic normality} \label{sec:qe}

In this section we give a brief introduction to quantum statistical inference, focused on the estimation of finite dimensional parameters and the theory of quantum local asymptotic normality.

\subsubsection{Quantum Cram\'{e}r-Rao bound} A \emph{quantum statistical model} consists of a family of states $\rho_\theta$ on a Hilbert space $\mathcal{H}$, which is indexed by a parameter $\theta$ that belongs to a set $\Theta$. In this paper we focus on parameter estimation and $\Theta$ is taken to be an open set of $\mathbb{R}^p$ while $\rho_\theta$ is assumed to depend smoothly on $\theta$. The statistical task is to estimate the unknown parameter $\theta$ by performing a measurement on the system and using the measurement outcome to compute an estimator.


Let $\mathcal{M}_*:L^1(\mathcal{H})\to L^1(\Omega, \Sigma ,\mu)$ be a measurement; its outcome $X\in \Omega$ is a sample from the distribution $\mathbb{P}^{\mathcal{M}}_\theta$ whose density with respect to $\mu$ is given by 
$p^{\mathcal{M}}_\theta: = \mathcal{M}_*(\rho_\theta)$. Therefore, according to the ``classical'' Cram\'{e}r-Rao bound, the covariance of any \emph{unbiased} estimator 
$\hat{\theta}  = \hat{\theta} (X)$ satisfies
$$
{\rm Cov}_{{\theta}}(\hat{{\theta}}):=\mathbb{E}_{ \theta}[(\hat{{\theta}}-{\theta})(\hat{{\theta}}-{\theta})^T]\geq I_{\cal M}({\theta})^{-1}
$$
where $I_{\cal M}({\theta})$ is the Fisher information matrix of the classical statistical model $\{ \mathbb{P}^{\mathcal{M}}_\theta~:~ \theta\in \Theta\subseteq \mathbb{R}^p\}$. The \emph{quantum Cram\'{e}r-Rao bound} (QCRB) \cite{Ho73,Be76,He76, Ho82,QCR1} is a fundamental result in quantum estimation which gives a lower bound for the precision of \emph{any} such measurement. 
\begin{theorem}[{\bf Quantum Cram\'{e}r-Rao bound}]
Let $\{ \rho_\theta ~:~ \theta\in \Theta\subseteq \mathbb{R}^p\}$ be a quantum statistical model on $\mathcal{H}$. Let $F(\theta)$ be the $p\times p$ quantum Fisher information (QFI) matrix whose elements are given by
$$
F(\theta)_{ij} = 
\tr (\rho_{\theta} 
\mathcal{L}^i_{\theta} \circ \mathcal{L}^j_{\theta}), \qquad i,j=1,\dots ,p
$$ 
where $\mathcal{L}^j_{\theta}$ are operators called symmetric logarithmic derivatives (SLD) satisfying 
$$
\partial_i\rho_{\theta}= \mathcal{L}^j_{\theta}\circ \rho_{\theta}
$$ 
and $\circ$ denotes the symmetric product $A\circ B = (AB+BA)/2$. Then for any measurement $\mathcal{M}$ the matrix inequality holds
\begin{equation}
\label{eq:cF<qF}
I_{\mathcal{M}}(\theta) \leq F(\theta).
\end{equation}
In particular, for any unbiased estimator $\hat{\theta} = \hat{\theta}(X)$ of $\theta$ the quantum Cram\'{e}r-Rao bound holds:
\begin{equation}\label{eq:multidimQCRB}
{\rm Cov}_{{\theta}}(\hat{{\theta}})
\geq F(\theta)^{-1}.
\end{equation}

\end{theorem}
For the sake of simplicity, we only presented the quantum Cram\'er-Rao bound in the case when the symmetric logarithmic derivatives are bounded operators. However, using the notion of square integrable operators with respect to a state and making the suitable assumptions, the QCRB theory carries to the unbounded case as well \cite{Ho82}.

Since measurements affect the state of the system, a single "quantum sample" can provide at most $F(\theta)$ amount of Fisher information. In practical situations this is often insufficient, and the experimenter needs to perform repeated measurements on an ensemble of identically prepared and independent systems in state $\rho_\theta$. The total QFI of the corresponding i.i.d. model $\{\rho_\theta^{\otimes n} ~:~ \theta\in \Theta\subset \mathbb{R}^p\}$ is $nF(\theta)$ and the QCRB \eqref{eq:multidimQCRB} shows the usual $n^{-1}$ rate for i.i.d. models. For one-dimensional parameters, the bound \eqref{eq:cF<qF} is formally saturated by measuring the SLD operator 
$\mathcal{L}_\theta$. In this case, the QCRB \eqref{eq:multidimQCRB} can be achieved \emph{asymptotically} with the  sample size $n$ by an adaptive, two stage measurement procedure \cite{GillMassar}. A preliminary non-optimal estimator 
$\widetilde{\theta}_n$ is computed from results of 
identical measurements on a small but growing subsample $\widetilde{n}\ll n$; on the remaining samples the SLD operator 
$\mathcal{L}_{\widetilde{\theta}_n}$ is measured and the final estimator $\hat{\theta}_n$ is computed e.g. by maximum likelihood. Under appropriate conditions on  
$\widetilde{\theta}_n$ one obtains that $\hat{\theta}_{n}$ is asymptotically normal and
$$
\lim_{n\to \infty } n \mathbb{E}(\theta -\hat{\theta}_n)^2 
 = F_\theta^{-1}.
$$ 

In contrast, for multidimensional parameter models, the QCRB is not achievable in general, even in the asymptotic limit; Intuitively, this has to do with the fact that the SLDs $\mathcal{L}^j_\theta$ may not commute with each other, so the different components of $\theta$ cannot be estimated optimally, simultaneously. A more rigorous analysis \cite{YCH19,RafalReview} based on quantum local asymptotic normality theory shows that the QCRB is achievable asymptotically if and only if the QFI matrix is real, i.e. 
$$
\tr(\rho_\theta [\mathcal{L}^i_\theta, \mathcal{L}^j_\theta]) =0, \quad i,j=1, \dots , p, \quad {\rm where}\quad [X,Y] = XY-YX. 
$$ 
When the QCRB is not achievable, one can nevertheless devise  optimal estimators which 
minimise the quadratic risk $R^G(\hat{\theta}_n):= \tr(G {\rm Cov}(\hat{\theta}_n))$ for a given real positive weight matrix $G$. In this case the Holevo bound \cite{Holevo2011} can be achieved by using the machinery of quantum local asymptotic normality, as we will show in section \ref{sec:QLAN} for the case of a simple two parameter example. The  general idea of this method is to map i.i.d. models into quantum Gaussian shifts models for which the optimal estimation problem is easier to solve.

\subsubsection{Quantum Gaussian shift models} \label{sec:QGSM}

In this section we introduce the notion of quantum Gaussian shift model and discuss the associated parameter estimation problem. In addition, we introduce a model consisting of mixtures of Gaussian shifts which will play an important role in the paper. For this we use the CV formalism introduced in section \ref{cvsystems}.  

\vspace{2mm}

Consider a CV system with Fock space $\mathfrak{H}(X)$, where $X$ is a finite dimensional Hilbert space of modes, as detailed in section \ref{cvsystems}. The CV system is characterised by canonical coordinates operators $Z( x)$ satisfying the commutation relations \eqref{eq:canonical.commutation.relations}, or equivalently in terms of the projective unitary representation of the translations group on $X$, given by the Weyl operators 
$W(x) = \exp(-iZ( x))$. 
In general, a \emph{quantum Gaussian shift model} consists of a family $\{ G(A\theta, V) ~:~ \theta \in \Theta \equiv \mathbb{R}^p\}$ where $G(A\theta, V)$ is a Gaussian state on $\mathfrak{H}(X)$ with covariance matrix $V$ and mean $A\theta$ with $A:\mathbb{R}^p\to X$ a real linear map and $V$ a fixed a quantum covariance matrix, cf equation \eqref{eq:Gaussian.state}.

For the purpose of this paper we restrict our attention to a specific Gaussian shift model
\begin{equation}
\label{eq:gaussian.shift.general}
{\bf G}=\{ |{\rm Coh}(x)\rangle :=W(x)|\Omega\rangle ~:~  x
\in X\}.
\end{equation}
where $|{\rm Coh}(x)\rangle$ is the coherent state whose unknown displacement parameter $x
\in X$ needs to be estimated. Since the  Wigner function of $|{\rm Coh}(x)\rangle$ is the probability density of the normal with mean 
$x$ and fixed covariance 
$I_{2k}$, the model is very similar to a classical Gaussian shift but cannot be reduced to the latter since the different components  are non-commuting selfadjoint operators and therefore cannot be measured simultaneously. 

\vspace{2mm}

To better understand the model consider an ONB $\{|e_{j}\rangle, j=1,\dots ,k\}$ in $X$ and write 
$|x \rangle = \sum_j x_j |e_j\rangle$ with Fourier coefficients $x_j\in \mathbb{C}$; for simplicity we identify $x$ with $(x_1,\dots , x_k)\in \mathbb{C}^k$. The associated basis of canonical variables
$ Q_j:= Z(-ie_j),P_j:= Z(e_j)$ have normal distributions with (real) means given by 
${\rm Re}(x_j)$ and respectively ${\rm Im}(x_j)$  and variance $1$, in the state $|{\rm Coh}(x)\rangle$. Each individual parameter can be estimated optimally by measuring the corresponding canonical variable. However, since these do not commute with each other, they cannot be measured simultaneously and one needs to make a trade-off between estimating means of incompatible observables. This is reflected in the fact that the QCRB \eqref{eq:multidimQCRB} for this model is not achievable 
(even in an asymptotic sense). The alternative is to focus on a \emph{specific} estimation problem and find measurements which minimise the risk for a given loss function, usually taken to be quadratic in the parameters. The caveat is that in this case the solution depends on the loss function, and no measurement is optimal for all such decision problems. One natural choice is the mean square error
$$
\mathbb{E}(\| \hat{x} - x \|^2) = 
\sum_{j=1}^k \mathbb{E}(| \hat{x}_j -x_j|^2),
$$
where $\hat{x}$ is an estimator constructed from the outcome of a specific measurement. Since all $Q_j$ commute with each other and similarly for all $P_j$, the simplest estimation strategy is 
to simultaneously measure these groups of observables separately. Since this requires two quantum systems,  we first map the coherent state into two identical copies with reduced amplitude through the isometry
$$
J: |{\rm Coh}(x)\rangle \mapsto 
\left|{\rm Coh}\left(\frac{x}{\sqrt{2}}\right)\right\rangle\otimes 
\left|{\rm Coh}\left(\frac{x}{\sqrt{2}}\right)\right\rangle
$$
and then simultaneously measure all coordinates $Q_1,\dots , Q_k$ one one of the copies, and all $P_1,\dots P_k$ on the other. The independent outcomes have distributions $N({\rm Re}( x)/\sqrt{2},I_{k})$ and $N({\rm Im}(x)/\sqrt{2},I_{k})$ and by rescaling we obtain an unbiased estimator 
$\hat{x}$ with distribution  
$N(x, 2 I_{2k})$ and whose risk is 
$\mathbb{E}(\|\hat{ x} - x\|^2) = 4k$. Note that this is twice the risk of the corresponding classical Gaussian shift model $N(x, I_{2k})$, a consequence of the intrinsic quantum nature of the multi-parameter coherent state model.  

\vspace{2mm}

For the general Gaussian shift model \eqref{eq:gaussian.shift.general}, a similar optimisation problem can be solved for arbitrary quadratic loss functions \cite{GG13}. The solution is to measure certain commuting canonical variables of a doubled up CV system consisting of the original one and an identical CV system prepared in a specific mean zero, pure Gaussian state.

\vspace{2mm}

We now introduce a second model which will feature in our main results later on. With the same notations as above, let $U$ be a unitary on $X$ such that $U^p=\id$ for some integer $p$ and $U^m \neq \id$ for every $m=1,\dots, p-1$, and consider the \emph{mixture model} 
\begin{equation}
\label{eq:mixture.model}
{\bf GM}=\left\{ \rho( x):= \frac{1}{p}\sum_{k=0}^{p-1} \alpha_{U*}^k(W( x)|\Omega\rangle\langle \Omega|W(x)^*)~:~ x\in X\right\},
\end{equation}
where $\alpha_{U*}(\cdot) = \Gamma(U^*)(\cdot)\Gamma(U) $ is the second quantised automorphism, cf. equation \eqref{eq:Gauge.automorphism}. 
 Note that each component of the mixture is a coherent state 
 $|{\rm Coh}(U^{-k}x)\rangle$:
 \begin{eqnarray*}
 \alpha_{U*}^k(W(x)|\Omega\rangle\langle \Omega|W(x)^*) &=&
 \Gamma(U^{-k}) W(x)|\Omega\rangle\langle \Omega|W(x)^* \Gamma(U^k) 
 = W(U^{-k}x)|\Omega\rangle\langle\Omega| W(U^{-k}x)^*
\\
&=&|{\rm Coh}(U^{-k}x)\rangle\langle {\rm Coh}(U^{-k}x)|\end{eqnarray*}
Therefore, the model \eqref{eq:mixture.model} can be seen as a mixture of Gaussian shift models which are related to each other by a group of unitary transformations $\{\Gamma(U^k)\}_{k=0,\dots , p-1}$.
Let
\begin{equation}
U=\sum_{m=0}^{p-1} \gamma^{m}R_m, \quad X=\bigoplus_{m=0}^{p-1} {\cal V}_m \text{ with } {\cal V}_m=R_mX,
\end{equation}
be the spectral decomposition of $U$, 
where $\gamma=e^{i2\pi/p}$. Then, using \eqref{eq:Fock.tensor.decomposition} the Fock space $\mathfrak{H}(X)$ can be decomposed as tensor product 
$\bigotimes_{m=0}^{p-1} \mathfrak{H}(\mathcal{V}_m)$ with 
$$
W(x) = \bigotimes_{m=0}^{p-1} W(R_m x), \qquad
\Gamma(U) = \bigotimes_{m=0}^{p-1} \Gamma (\gamma^{m} \mathbf{1}_m), 
\qquad 
|\Omega\rangle = \bigotimes_{m=0}^{p-1} |\Omega_m\rangle
$$ 
where $|\Omega_m\rangle$ is the vacuum state in 
$\mathfrak{H}(\mathcal{V}_m)$.
The mixture model can then be written as 
\begin{equation} \label{eq:2pexample}
\rho(x)=
\frac{1}{p}\sum_{k=0}^{p-1}\bigotimes_{m=0}^{p-1}W(\gamma^{mk}R_m  x)\ket{\Omega_m}\bra{\Omega_m}W(\gamma^{mk}R_m x)^*.
\end{equation}
The state \eqref{eq:mixture.model} is invariant under the action of the group $\{\Gamma(U^k)\}_{k=0,\dots p-1}$ and the parameter $x$ is not fully identifiable, i.e. some information is lost due to mixing. The identifiable parameters are characterised in the following lemma (for the proof, we refer to Appendix \ref{app:equivmixedGauss}).  
\begin{lemma} \label{lem:idgm}
The states $\rho( x)$ and $\rho(y)$ are identical if and only if there exists a $k=0,\dots, p-1$ such that $y=U^k  x$. Therefore, the space of identifiable parameters is the quotient of $X$ with respect to the action of the group $\{U^k\}_{k=0,\dots p-1}$.
\end{lemma} 
Another consequence of the invariance under the group action is that the spectral decomposition of $\rho(x)$ is block diagonal with respect to the eigenprojectors of $\Gamma(U)$. 
This structure is described in Proposition \ref{prop.mixture.model}  below whose proof can be found in Appendix \ref{app:equivmixedGauss}.
Before formulating the Proposition we introduce the following notation. Let $\mathcal{V}_0^{\perp} := \bigoplus_{m=1}^{p-1} \mathcal{V}_m$, such that $X={\cal V}_0 \oplus {\cal V}_0^\perp$ and the Fock space and the vacuum state factorise as
$$
{\frak H}(X)={\frak H}({\cal V}_0) \otimes {\frak H}({\cal V}_0^\perp), \qquad \ket{\Omega} =\ket{\Omega_0} \otimes \bigotimes_{m=1}^{p-1} |\Omega_m\rangle =\ket{\Omega_0} \otimes \ket{\Omega_0^\perp}.
$$


\begin{proposition}
\label{prop.mixture.model}
The operator $\Gamma(U)$ has spectral decomposition
$$
\Gamma(U)= \sum_{m=0}^{p-1} \gamma^m Q_m
$$
where $Q_m$ is the projection onto the space $\mathfrak{H}(\mathcal{V}_0)\otimes{\frak V}_m $, where  
$$
{\frak V}_m:= 
\bigoplus_{n\geq 0}\, {\rm Symm} \left[\bigoplus_{\substack{m_1,\dots, m_n =1,\dots, p-1 \\m_1 \oplus \cdots \oplus m_n=m}}{\cal V}_{m_1} \otimes \cdots \otimes {\cal V}_{m_n}\right] \subset \mathfrak{H}(\mathcal{V}_0^\perp),  \qquad m=0,\dots p-1,
$$
where ${\rm Symm}$ denotes the symmetric subspace, and $\oplus$ denotes addition modulo $p$.

\vspace{2mm}

The state $\rho(x)$ has the following decomposition
\begin{equation} \label{eq:alternative}
\rho( x)=
W(x_0)|\Omega_0\rangle\langle \Omega_0|W( x_0)^*\otimes \sum_{m=0}^{p-1}\ket{\zeta_m( x_0^\perp)}\bra{\zeta_m(x_0^\perp)}
\end{equation}
where $\ket{\zeta_m(x_0^\perp)}$ are orthogonal unnormalised vectors
\[
\begin{split}&\ket{\zeta_m(x_0^\perp)}=Q_m W( x_0^\perp)|\Omega_0^\perp
\rangle\\
    &=e^{-\|x_0^\perp\|^2/2} \left ( \delta_0(m) \ket{\Omega_0^\perp} + \sum_{l \geq 1}\frac{1}{\sqrt{l!}}\sum_{m_1\oplus\dots\oplus m_l=m} R_{m_1}|x_0^\perp \rangle \otimes \cdots \otimes R_{m_l}|x_0^\perp \rangle\right).\end{split}\]
\end{proposition}




The complete analysis of the estimation theory for the mixture model $\mathbf{GM} $ 
is non-trivial and goes beyond the scope of this work. However, we will briefly comment on the case of period $p=2$, which already points out some interesting features of such statistical models and their difference compared to Gaussian shift models. 
In this case $\rho( x)$ is the tensor product 
\begin{align}&
\label{eq:gaussian.mixed.p=2}
\rho( x) = 
W(R_0 x)|\Omega\rangle\langle \Omega|W(R_0 x)^*\otimes \frac{1}{2}\sum_{k=0}^{1}W((-1)^kR_1 x)|\Omega\rangle\langle \Omega|W((-1)^kR_1 x)^* 
\end{align}
and the identifiable parameters are given by the pair $(x_0,[x_1])$, where $x_0 \in {\cal V}_0$, ${\cal V}_1$ and $[x_1]=\{x_1,-x_1\}$. The left-hand side tensor in \eqref{eq:gaussian.mixed.p=2} describes a standard quantum Gaussian shift model with parameter $x_0$ which can be estimated by performing the measurement described at the beginning of this section. The right-hand side is a mixture of two coherent states with opposite amplitudes. Simplifying further, let us assume that $\mathcal{V}_1$ is one-dimensional, i.e. $x_1 = e^{i\phi} |x_1|\in \mathbb{C}$. Applying Proposition \ref{prop.mixture.model} we find that the state can be written as the mixture of two orthogonal pure states supported by the subspaces of the Fock space with even and respective odd number of excitations.
$$
\rho(x_1)= \frac{1}{2}\left[ W(x_1)|\Omega_0^\perp\rangle\langle \Omega_0^\perp|W(x_1)^* + W(-x_1)|\Omega_0^\perp\rangle\langle \Omega_0^\perp|W(-x_1)^* \right] =
|\zeta_{0}\rangle\langle \zeta_{0}| + |\zeta_{1}\rangle\langle \zeta_{1}| 
$$
where 
$$
|\zeta_{0}\rangle = e^{-|x_1|^2/2}
\sum_{k=0}^{\infty} \frac{x_1^{2k}}{\sqrt{(2k)!}} |2k\rangle,
\qquad
|\zeta_{1}\rangle = 
e^{-|x_1|^2/2}
\sum_{k=0}^{\infty} \frac{x_1^{2k+1}}{\sqrt{(2k+1)!}} |2k+1\rangle
$$
Since the even and odd subspaces do not depend on the parameter value, we can perform a projective measurement whose elements are the projections on these subspaces, and conditional on the measurement outcome, remain with a sample of the corresponding pure state, i.e. the normalised versions of the pure states $|\zeta_0\rangle, |\zeta_1\rangle$. This procedure is a "sufficient statistic" in the sense that by ignoring the measurement outcome we recover the original mixed state without any loss of statistical information. Therefore, one can construct "optimal" estimation procedures by first projecting on the odd/even subspaces and then applying the optimal measurement for the corresponding pure state, thus reducing the problem to one involving pure rather than mixed states. 

Solving the corresponding optimal estimation problem is left for a future investigation. However, we do want to make some qualitative remarks and point to an interesting connection with another quantum estimation problem appearing in quantum imaging \cite{Tsang16}. The simplest measurement strategy is to apply the joint measurement of $Q$ and $P$ described at the beginning of the section, which turned out to be optimal for the quantum Gaussian shift with coherent states. In this case we obtain an outcome in $\mathbb{R}^2$ whose distribution is the equal mixture of the normal distributions 
$N\left(\pm ({\rm Re} (x_1), {\rm Im} (x_1)), 2 I_2\right)$, i.e. a classical analogue of our mixture model. Estimating parameters of Gaussian mixtures is an important problem in (classical) statistics \cite{McLachlan} which can exhibit features such as non-identifiability and non-standard estimation rates. For further insight let us simplify the problem furher and assume that $x_1$ is real, so we deal with a one-dimensional estimation problem. Since the sign of $x_1$ is not identifiable, we can only estimate the absolute value $|x_1|$, but the classical Fisher information decreases to zero in the limit of small $|x_1|$. This singularity is reflected in non-standard estimation rate of $n^{-1/4}$ in the i.i.d. setting, cf. \cite{Chen95}. On the other hand, let us consider what happens if instead of a joint measurement of $Q$ and $P$ we measure the number operator $N$. The integer valued outcome has Poisson distribution with intensity $|x_1|^2$ and the classical Fisher information is constant and equal to the quantum Fisher information, so the QCRB is achieved for all $|x_1|\neq 0$. This surprising fact is at the basis of current investigations in quantum imaging \cite{Tsang19}, which aim at improving accuracy of optical imaging by choosing favorable measurement settings which avoid the limitations of "classical" imaging, and bypass the famous Rayleigh imaging limit. Extending this technique to the two-dimensional model we started with, and further to the general case of quantum Gaussian mixture models with arbitrary $p$, remains an interesting open question.

\subsubsection{Quantum decision problems and comparison of quantum  models}

So far we focused our attention on the specific  problem of parameter estimation. More generally, a (quantum-to-classical) decision problem for a model 
${\bf Q}:= \{\rho_\theta: \theta\in \Theta\}$ on a Hilbert space $\mathcal{H}$ is specified by a decision space $(\Omega, \Sigma)$ and a loss function $\ell:\Omega\times \Theta\to \mathbb{R}_+$, 
and the aim is to design a measurement 
$\mathcal{M}_*: L^1(\mathcal{H}) \to 
L^{1}(\Omega, \Sigma, \mu)$ such that the risk 
${R}(\mathcal{M}, \theta) = 
\mathbb{E}^{\mathcal{M}}_\theta (\ell(X ,\theta))$ is as small as possible. Here the expectation is taken with respect to the distribution of the outcome $X$ of $\mathcal{M}$, whose density is  
$p^{\mathcal{M}}_\theta := \mathcal{M}_*(\rho_\theta)$. In the frequentist approach one considers the \emph{maximum risk} $R_{\rm max} (\mathcal{M}):= \sup_{\theta} R(\mathcal{M}, \theta)$ and aims to find a measurement $\mathcal{M}$ which has the smallest maximum risk. This leads to the notion of  \emph{minimax risk} of the model ${\bf Q}$ defined as 
$$
R_{\rm minmax}({\bf Q}) = \inf_{\mathcal{M}} \sup_{\theta} R(\mathcal{M}, \theta).
$$ 

\begin{definition}[{\bf Equivalent models}]
Two quantum models 
${\bf Q}:= \{\rho_\theta\in S(\mathcal{H}): \theta\in \Theta\}$ and ${\bf R}:= \{\sigma_\theta\in S(\mathcal{K}): \theta\in \Theta\}$ are \emph{statistically equivalent} 
if there exist channels 
$\mathcal{T}_*: L^1(\mathcal{H}) \to L^1(\mathcal{K})$ and ${\cal S}_*: L^1(\mathcal{K}) \to L^1(\mathcal{H})$ such that
$$
\mathcal{T}_*(\rho_\theta) = \sigma_\theta,\qquad {\rm and }\qquad 
\mathcal{S}_*(\sigma_\theta)
 =\rho_\theta, \qquad
 {\rm for~all~}\theta\in \Theta.
$$
\end{definition}
The definition is justified by the fact that equivalent models share the same set of possible risk functions for any decision problem. Indeed, since equivalent models can be physically mapped into each other, this means that a measurement for one model can be pulled back into one for the other model 
: given any measurement $\mathcal{N}_*:L^1(\mathcal{K}) \to L^1(\Omega, \Sigma, \mu)$ for the model ${\bf R}$ we define the measurement $\mathcal{M}:= \mathcal{T}\circ \mathcal{N}$ for the model ${\bf Q}$ such that 
$p^{\mathcal{N}}_\theta = \mathcal{N}_*(\sigma_\theta) =  \mathcal{N}_*\circ \TT_*(\rho_\theta) = p^{\mathcal{M}}_\theta$ and the risks of these decisions coincide. In particular, the models have equal minmax risks.


\emph{Quantum sufficiency} theory \cite{Petz86,JencovaPetz06} 
provides a complete algebraic characterisation of statistical equivalence in terms of an isomorphism between the minimal sufficient algebras associated to the two models. However, the notion of equivalence is rather rigid as models may be ``close'' to each other but not statistically equivalent. While the theory of quantum sufficiency goes beyond the scope of this paper, the idea of ``closeness'' of statistical models plays a crucial part in our work. More specifically, we use the following definition, which is the natural quantum extension of the Le Cam distance between classical models \cite{LeCam}.
\begin{definition}[{\bf Statistical deficiency and Le Cam distance}]
\label{def:LECAM}
Let 
${\bf Q}:= \{\rho_\theta\in S(\mathcal{H}): \theta\in \Theta\}$ and 
${\bf R} := \{\sigma_\theta\in S(\mathcal{K}): \theta\in \Theta\}$ 
be two quantum statistical models. The \emph{statistical deficiency} of ${\bf Q}$ with respect to ${\bf R}$ is 
\begin{eqnarray*}
\delta({\bf Q}, {\bf R}) 
& = &
\inf_{\mathcal{T}_*} \sup_{\theta\in \Theta}
\| \mathcal{T}_*(\rho_\theta) - \sigma_\theta \|_1,
\end{eqnarray*}
where the infimum is taken over all channels $\mathcal{T}_*: L^1(\mathcal{H}) \to L^1(\mathcal{K})$; in particular if 
${\bf Q}$ is more informative than ${\bf R}$, then $\delta({\bf Q}, {\bf R})=0$. The \emph{Le Cam distance} between ${\bf Q}$ and ${\bf R}$ is 
$$\Delta({\bf Q},{\bf R}) = {\rm max}
(\delta({\bf Q}, {\bf R}),\delta({\bf R}, {\bf Q})).$$

\end{definition}
For any decision problem with decision space $(\Omega, \Sigma) $ and bounded loss function $\ell(\cdot, \cdot)$, the Le Cam distance can be used to upper bound the risks of one model in terms of the other. Indeed, let $\mathcal{N}$ and $\mathcal{M}$ be measurements defined as above. Then, since measurement maps are contractive with respect to the 1-norm, we have
\begin{equation}
\label{eq:contractivity.norm.1}
\|p^{\mathcal{M}}_\theta - p^{\mathcal{N}}_\theta\|_1 = \| \mathcal{N}_*\circ \mathcal{T}_*(\rho_\theta) - \mathcal{N}_* (\sigma_\theta)
\|_1\leq \|  \mathcal{T}_*(\rho_\theta) - \sigma_\theta\|_1
\end{equation}
which implies
$$
R_{\rm minmax}({\bf Q}) \leq 
R_{\rm minmax}({\bf R}) + 
\|\ell\|_\infty \cdot \delta({\bf Q}, {\bf R}).
$$
Using a similar argument in the opposite direction we conclude that models which are close in the Le Cam distance have similar statistical behaviour. This is particularly helpful in asymptotic theory where a sequence of models converges to a more tractable model in a certain ``large sample size'' limit, and therefore studying the limit model can provide an asymptotically optimal strategy for the original problem.  
An important instance of this is local asymptotic normality which we discuss below.


\subsubsection{Quantum local asymptotic normality}
\label{sec:QLAN}

In a nutshell, quantum local asymptotic normality (QLAN) means that certain sequences of quantum models can be approximated by a quantum Gaussian shift model, in a local neighbourhood of a parameter value, asymptotically with respect to the ``sample size''. This phenomenon has been demonstrated in several settings: i.i.d. models with finite dimensional systems and mixed states \cite{LAN1,GJ07,GJK08,GK09,GG13,YFG13,BGM18,FY20,FY23,LN24}, i.i.d. models with infinite dimensional pure states \cite{BGM18}, and  models consisting of correlated output states of quantum Markov processes \cite{Guta11,GutaCatana14,GK15,GK17}. In order to provide background intuition for our results on QLAN, we describe here the simplest instance of QLAN arising in the estimation of pure qubit states and we discuss how this can be used to devise asymptotically optimal state estimation procedures. 

\vspace{2mm}

Let $\{|0\rangle , |1\rangle\}$ be the standard orthonormal basis in $\mathbb{C}^2$, and let  
$$
\sigma_x = 
\left(
\begin{array}{cc}
  0   &  1\\
 1    & 0
\end{array}
\right), 
\quad
\sigma_y = 
\left(
\begin{array}{cc}
  0   &  -i\\
 i    & 0
\end{array}
\right)
\quad
\sigma_z = 
\left(
\begin{array}{cc}
  1   &  0\\
 0    & -1
\end{array}
\right)
$$
be the Pauli observables which together with the identity form a basis of qubit observables, but can also be seen as the generators of the two dimensional special unitary group ${\rm SU}(2)$. Geometrically, the qubit state space $S(\mathbb{C}^2)$ is isomorphic to the three-dimensional unit sphere $B_1(\mathbb{R}^3)$ through what is known as the \emph{Bloch representation} defined by the linear map
$$
S(\mathbb{C}^2) \ni \rho  
~
\longmapsto 
~
r^\rho  =(r^\rho_x, r^\rho_y, r^\rho_z)
\in B_1(\mathbb{R}^3)
$$
where $r^\rho$ is the \emph{Bloch vector} of $\rho$ and its components are given by the expectations of the Pauli observables $r^\rho_i = {\rm Tr}(\rho \sigma_i)$ for $i=x,y,z$. In this representation, the pure states correspond to vectors on the surface of the unit sphere, with the standard basis vectors at the poles, cf. Figure \ref{fig:Bloch.sphere} a). 
\begin{center}
\begin{figure}
    \centering
    \includegraphics[width=1\linewidth]
    {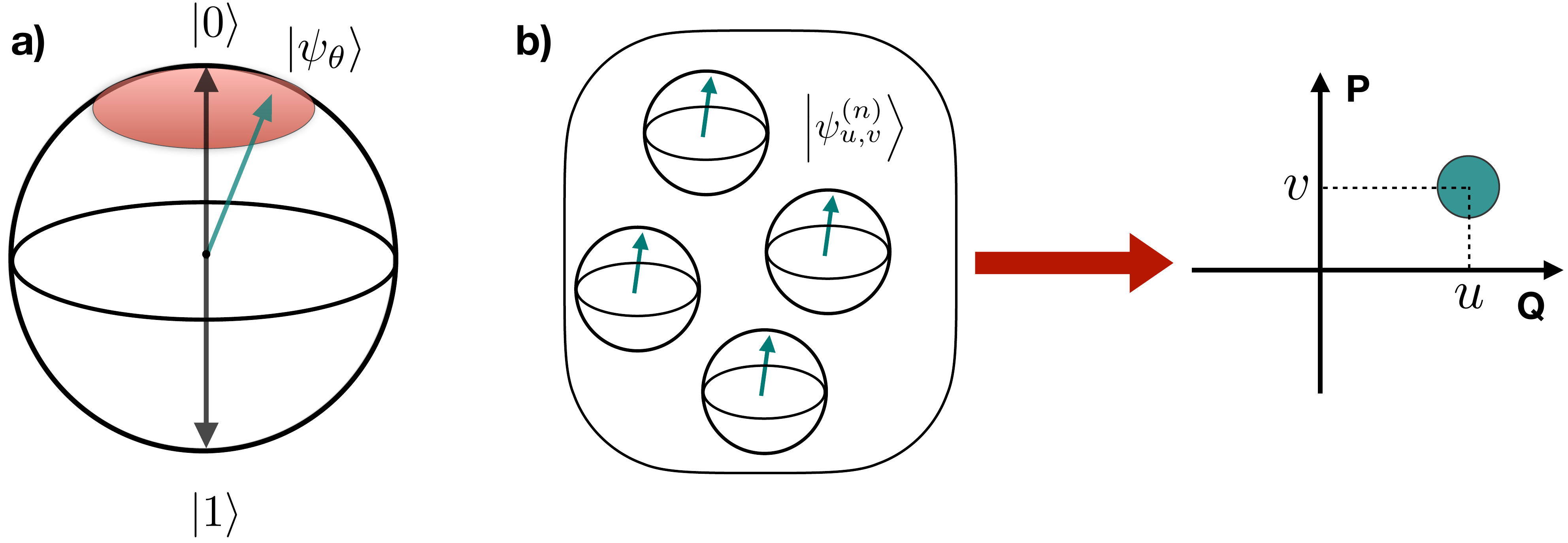}
    \caption{{\bf Panel a)} Bloch ball representation of qubit states. Pure states are represented as vectors on the unit sphere with basis vectors $0\rangle$ and $|1\rangle$ as north and south poles respectively. The pure state statistical model $|\psi_\theta\rangle $ covers a two dimensional neighbourhood of the north pole. Each state is a rotation of $|0\rangle$ with rotation parameter $\theta = (\theta_1, \theta_2) \in \Theta$. {\bf Panel b)} Illustration of quantum local asymptotic normality for pure qubit states. The i.i.d. model consisting of an ensemble of $n$ identically prepared qubits in state $|\psi^{(n)}_{u,v}\rangle:= |\psi_{u/\sqrt{n},v/\sqrt{n}}\rangle^{\otimes n}$ where 
    $u,v$ are local rotation parameters. For large $n$ the local i.i.d. model converges to quantum Gaussian model consisting of a single sample from the coherent state $|u,v\rangle$ whose mean is given by the local paramters. }
    \label{fig:Bloch.sphere}
\end{figure}
\end{center}
We consider the ``global'' qubit model consisting of pure states  
$\{ |\psi_{\theta}\rangle ~:~ \theta\in \Theta\subset \mathbb{R}^2\}$ with
$$
|\psi_{\theta}\rangle := 
\exp\left[-i (\theta_1\sigma_y - \theta_2\sigma_x) \right] |0\rangle,
$$
and assume that $\theta:=(\theta_1,\theta_2)$ belongs to a bounded, open and sufficiently small neighbourhood of the origin $\Theta\subset \mathbb{R}^2$  such that the parameter is identifiable, i.e. the map $\theta\to |\psi_\theta\rangle$ is one-to-one, cf. Figure \ref{fig:Bloch.sphere} a). 
More generally, one could drop the purity assumption and consider a three dimensional model consisting of mixed 
states, but this will not discussed here and we refer to \cite{LAN1} for further details.

As argued before, a single sample is insufficient for practical estimation, and instead one considers the corresponding i.i.d. estimation problem: given $n$ identically prepared qubits in state 
$
|\psi_\theta\rangle
$, 
estimate $\theta$ with the ``highest precision'' by measuring the qubit ensemble whose joint state is $|\psi_\theta\rangle^{\otimes n}$. To qualify the term ``highest precision'' we choose the loss function  
$\ell (\hat{\theta},\theta) =\|\hat{\theta}-\theta\|^2$ but note that more general locally quadratic functions can be considered. 
In the asymptotic scenario of large sample size $n$, this can be done by using a two-stage procedure \cite{GillMassar,GiGoGu}. In the first stage we use a small proportion $\widetilde{n} =n^{1-\epsilon}\ll n$ of the samples to compute a preliminary (non-optimal) estimator $\widetilde{\theta}_n$, where $\epsilon>0$ is a small constant. For simplicity, let us assume that $\widetilde{\theta}_n =(0,0)$, which can always be achieved by a suitable re-parametrisation using the rotation symmetry of the model. By applying standard concentration bounds, one can show that for reasonable preliminary estimators, 
$\theta$ belongs to a ball of size $Cn^{-1/2+\delta}$ centered at $\widetilde{\theta}_n$, for some fixed $\epsilon<\delta<1/2$ and $C>0$. 
We can therefore write 
$\theta= (\theta_1,\theta_2) = (u/\sqrt{n},v/\sqrt{n})$ with local parameters satisfying $\|(u,v)\|\leq n^\delta$. In the second stage we aim to measure the remaining $n-\tilde{n}$ samples in such a way as to obtain an optimal estimator $(\hat{u}_n, \hat{v}_n)$ of the local parameters $(u,v)$, and combine this with the preliminary estimator to obtain the final estimator 
$\hat{\theta}_n = \tilde{\theta}_n+ (\hat{u}_n, \hat{v}_n)/\sqrt{n}$. The performance of this final estimator depends crucially on that of the local estimator, which in turn relies on understanding the asymptotic behaviour of the local model. This insight is provided by the quantum local asymptotic normality theory on which we focus our attention now.

\vspace{2mm}

With the notation established above, we consider the sequence of local i.i.d. models 
$$
|\psi^{(n)}_{z}\rangle = 
|\psi_{u/\sqrt{n},v/\sqrt{n}}\rangle^{\otimes n}  :=
\left( 
\exp\left[-i (u\sigma_y - v\sigma_x)/\sqrt{n} \right] |0\rangle
\right)^{\otimes n},\quad  z:= u+iv.
$$
consisting of $n$ identically prepared qubits, each individual state being obtained by rotating $|0\rangle$ by a small amount which scales as $n^{-1/2}$.
For notational convenience we use the complex local parameter 
$z= u+iv$ in the following.
We present two formulations of QLAN showing that the i.i.d. sequence converges to the one-mode quantum Gaussian shift model 
$$
{\bf G}: =\{|{\rm Coh}(z)\rangle:=  W(z)|\Omega\rangle : z\in \mathbb{C}\},
$$ 
where $|{\rm Coh}(z)\rangle$ is a coherent state of a one-mode CV system whose canonical coordinates $Q,P$ have mean $(u,v) = ({\rm Re}(z), {\rm Im}(z))$, cf. section \ref{cvsystems}. The general idea of QLAN is illustrated in Figure \ref{fig:Bloch.sphere} b). The first formulation  takes advantage of the fact that all models consist of pure states. Therefore, their properties are encoded in the Hilbert space geometry which is  determined by the inner products of pairs of states with different parameters. 
\begin{theorem}[{\bf weak QLAN}]
\label{th:weak.QLAN}
For any pair of parameters 
$z_i =u_i+iv_i$ with $i=1,2$, the following limit holds 
\begin{equation}
\label{eq:weak.lan}
\lim_{n\to\infty}
\langle 
\psi^{(n)}_{z_1} | \psi^{(n)}_{z_2}\rangle =
\langle {\rm Coh}(z_1)| {\rm Coh}(z_2)\rangle
.
\end{equation}
\end{theorem}
Theorem \ref{th:weak.QLAN} can be proved by expanding the inner product on the left side in Taylor series in $n^{-1/2}$ and using the connvergence of $(1-a/n)^n$ to $\exp(-a)$ together with the explicit expression of the overlap on the right.
Although it has an appealing geometric interpretation, the weak convergence result \eqref{eq:weak.lan} is restricted to pure states and holds only point-wise with respect to local parameters. This makes it less amenable to statistical considerations which require uniform bounds over local parameters. These drawbacks are corrected by the second QLAN result which is formulated in terms of the Le Cam distance between suitable restrictions of the two models to local parameters in a growing ball of size $ C n^{\delta}$ where $0<\delta<1/2$ and C>$0$ are fixed constants. 
Let us define the sequence of local i.i.d. models 
$$
{\bf Q}(n) := 
\{|\psi^{(n)}_{z}\rangle ~:~ \|z\|\leq C n^{\delta} \}
$$
and the corresponding sequence of restricted quantum Gaussian shift models
$$
{\bf G}(n) = \{|{\rm Coh}(z) \rangle ~:~ 
\|z\|\leq C n^{\delta}\}
$$

\begin{theorem}[{\bf strong QLAN}]\label{th:strong.QLAN}
The sequences of local i.i.d. and restricted Gaussian models ${\bf Q}(n) $ and respectively 
${\bf G}(n)$ approach each other in the Le Cam sense
$$
\lim_{n\to\infty}\Delta ({\bf Q}(n), {\bf G}(n)) =0.
$$
In particular, there exist quantum channels $\mathcal{T}_{n*}$ and $\mathcal{S}_{n*}$ such that 
\begin{eqnarray*}
&&
\lim_{n\to\infty}
\sup_{\|z\|\leq Cn^{\delta}}
\| \mathcal{T}_{n*}(\rho^{(n)}_{z}) - G(z,I_2) \|_1 =0
\\
&& 
\lim_{n\to\infty}\sup_{\|z\|\leq Cn^{\delta}}
\| \mathcal{S}_{n*}(G(z,I_2)) - \rho^{(n)}_{z}\|_1 =0.
\end{eqnarray*}
where $\rho^{(n)}_z = |\psi^{(n)}_{z}\rangle\langle \psi^{(n)}_{z}|$ and 
$G(z,I_2):= |{\rm Coh}(z)\rangle\langle{\rm Coh}(z)|$.
\end{theorem}
Theorem \ref{th:strong.QLAN} can be proved by explicitly constructing the quantum channels $\mathcal{T}_{n*}$ and 
$\mathcal{S}_{n*}$ based on an isometric embedding of the symmetric subspace of 
$\left(\mathbb{C}^2\right)^{\otimes n}$ in the one-mode Fock space $\mathfrak{H}(\mathbb{C})$ of the limit 
Gaussian model. This provides an operational procedure to physically map the i.i.d. model into the Guassian one (up to small errors) and vice-versa, in full analogy to the classical LAN theory developed by Le Cam \cite{LeCam}. In particular, strong QLAN allows us to translate the original i.i.d. estimation problem into a simpler one concerning the estimation of a Gaussian shift model, which was discussed in section \ref{sec:QGSM}.

\vspace{2mm}
To conclude, we return to the original i.i.d. model and the qubit estimation problem and sketch the second stage of the procedure to which we alluded earlier in this section. After the first stage, the ``global'' parameter $\theta$ has been ``localised'' in a  neighbourhood of 
size $n^{-1/2+\delta}$ around the preliminary estimator 
$\tilde{\theta}_n$ which after re-parametrisation and unitary rotation can be assumed to be $\tilde{\theta}_n=(0,0)$; the state of the remaining qubits has local parameters satisfying $\| (u,v)\|\leq n^{\delta}$ with $\epsilon<\delta<1/2$. By applying the channel $\mathcal{T}_{n*}$ in Theorem \ref{th:strong.QLAN}, the i.i.d. state is mapped into a one-mode CV state $\mathcal{T}_{n*}(\rho^{n}_{z})$ which is close to the pure Gaussian state $G(z,I_2)$ with means $(u,v)= ({\rm Re}(z),{\rm Im}(z))$. To estimate the mean, we apply the measurement described in section \ref{sec:QGSM}, which is optimal for the quadratic loss function 
considered here. 
With the outcomes $(\hat{u}_n, \hat{v}_n)$ of the measurement we can compute the final estimator 
$\hat{\theta}_n = \tilde{\theta}_n + (\hat{u}_n, \hat{v}_n)/\sqrt{n}$. This procedure can be shown to be asymptotically optimal in the sense that $\hat{\theta}_n$ is a asymptotic minimax estimator, but also in the sense that it achieves the asymptotic Holevo bound \cite{Holevo2011}
$$
\lim_{n\to\infty} n \mathbb{E} \left[ \|\hat{\theta}_n - \theta\|^2\right] =4.
$$
The key ingredients are the concentration bound used to localise the parameter in the first stage, the strong QLAN convergence used in the states in the second stage, the optimality of the CV measurement for Gaussian shift models and the of norm-one contractivity property 
\eqref{eq:contractivity.norm.1} which allows to approximate risks in terms of the Le Cam distance, cf. \cite{YCH19} for details.

\subsubsection{Useful tools in proving convergence of statistical models}

Despite being the relevant notion in statistical applications, convergence results for the Le Cam distance (strong convergence) are often hard to prove. On the other hand, the notion of weak convergence introduced earlier is easier to verify, but has the drawback that it only makes sense for pure states models, and in itself, is not sufficiently strong to establish 
optimal estimation results. In this section we present new results that allow to upgrade weak convergence of pure statistical models to strong convergence and to prove strong convergence for certain types of mixed statistical models. The general line follows ideas form section 3 in \cite{GK15}, but the new results allow for growing local parameter sets and provide explicit error bounds.



\begin{lemma}[{\bf Weak to strong convergence for pure state models}]\label{lem:wts2} Let $\{r_n\}_{n \in \mathbb{N}}$ be a sequence of positive numbers such that $r_n \uparrow +\infty$, $C$ a positive constant, $X$ be a normed space of finite dimension $d$, and $A$ a finite set. Moreover, let us consider the sequence of pure statistical models
\begin{align*}
    &\mathbf{Q}(n):=\{\rho_n(x):=\ket{v_n(x,a)}\bra{v_n(x,a)}\}_{\{(x,a) \in {\cal U}(n)\times A\}}, \\
    &\mathbf{Q}^{\infty}:=\{\rho(x):=\ket{v(x,a)}\bra{v(x,a)}\}_{\{(x,a) \in X\times A\}},
    \end{align*}
where ${\cal U}(n):=B_{C r_n}(0)\subseteq X$ is the ball centered in $0$ with radius $C r_n$. We will use $\mathbf{Q}^{\infty}(n)$ to denote the model $\mathbf{Q}^{\infty}$ restricted to ${\cal U}(n)\times A$.

Let us assume that the following hypotheses hold true:
\begin{enumerate}
\item \label{it:Lp2} \textbf{(H\"older parametrization)} there exist two positive constants $C,\alpha>0$ such that for every $x,y \in X$ and $a \in A$
$$1-|\langle v(x,a),v(y,a)\rangle|^2 \leq C \|x-y\|^\alpha;$$
\item  \label{it:fc2} \textbf{(Fast enough uniform weak convergence)} let
$$f(n):=\sup_{(x,a),(y,b) \in {\cal U}(n)\times A}|\langle v_n(x,a),v_n(y,b) \rangle -\langle v(x,a),v(y,b) \rangle|;$$
assume that
$$\lim_{n \rightarrow +\infty} r_n^df(n)=0.$$
Note that, in this case, we can find a sequence of positive numbers $\{\delta_n\}_{n \in \mathbb{N}}$ such that $\delta_n \downarrow 0$ and and $r_n^df(n)=o(\delta^d_n).$
Then
$$\Delta(\mathbf{Q}(n),\mathbf{Q}^\infty(n))=O\left (\delta_n^{\alpha}+f(n)^{\frac{1}{2}}+\left(\frac{r_n}{\delta_n}\right )^{\frac{d}{4}}f(n)^{\frac{1}{4}}\right )$$
and, therefore,
$$\lim_{n \rightarrow +\infty}\Delta(\mathbf{Q}(n),\mathbf{Q}^\infty(n))=0.$$
\end{enumerate}

\end{lemma}

The proof of Lemma \ref{lem:wts2} can be found in Appendix \ref{app:lemmas.convergence.models} as a corollary of the more general Lemma \ref{lem:wts1} which is also used in proving the next result concerning mixed statistical models.
This shows that if one is able to express the states as a sum of rank one operators, where the sum index does not depend on the unknown parameter, nor on $n$, then the strong convergence of the statistical models boils down to showing the strong convergence of the collection of rank one operators appearing in the decomposition and one can apply Lemma \ref{lem:wts1}.

\begin{lemma}[{\bf Strong convergence of mixed statistical models}] \label{lemm:mixed}
Let $\{r_n\}_{n \in \mathbb{N}}$ be a sequence of positive numbers such that $r_n \uparrow +\infty$, $C$ a positive constant, $X$ be a normed space of finite dimension $d$. Moreover, let us consider a sequence of statistical models of the form
\begin{align*}
    &\mathbf{Q}(n):=\{\rho_n(x):\, x \in {\cal U}(n)\}, \\
    &\mathbf{Q}^{\infty}:=\{\rho(x):\,x \in X\},
    \end{align*}
where ${\cal U}(n):=B_{C r_n}(0)\subseteq H$ is the ball centered in $0$ with radius $C r_n$. We will use $\mathbf{Q}^{\infty}(n)$ to denote the model $\mathbf{Q}^{\infty}$ restricted to ${\cal U}(n)$. Assume that the following statements hold true:
\begin{enumerate}
\item there exists a finite set $A$ such than one can write $\rho_n(x)$ and $\rho(x)$ as a sum of rank one operators in the following way:
$$\rho_n(x)=\sum_{a \in A}\ket{v_n(x,a)}\bra{v_n(x,a)}, \quad \rho(x)=\sum_{a \in A}\ket{v(x,a)}\bra{v(x,a)};$$
\item there exists a sequence of positive numbers $(a_n)_{n \geq 0}$ such that
$$\lim_{n \rightarrow +\infty}a_n=+\infty$$
and
$$\lim_{n \rightarrow +\infty}\Delta(\mathbf{P}(n),\mathbf{P}^{\infty}(n))a_n=0,$$
where
\begin{align*}
    &\mathbf{P}(n):=\{\ket{v_n(x,a)}\bra{v_n(x,a)}:(x,a) \in {\cal U}(n)\times A\}, \\
    &\mathbf{P}^{\infty}(n):=\{\ket{v(x,a)}\bra{v(x,a)}:(x,a) \in {\cal U}(n)\times A\}.
    \end{align*}
\end{enumerate}
In this case, one has
$$\lim_{n \rightarrow +\infty}\Delta(\mathbf{Q}(n),\mathbf{Q}^{\infty}(n))a_n=0$$
\end{lemma}
The proof of lemma \ref{lemm:mixed} can be found in Appendix \ref{app:lemmas.convergence.models}. 
\subsection{Irreducible quantum Markov chains} \label{sec:irrQMC}

In this section we review the basic elements of the ergodic theory of quantum channels that are needed in the paper. There is a remarkable similarity to the classical theory of transition matrices and at the end of the section we show the latter can be seen as special case of the quantum theory.

\subsubsection{Irreducible quantum channels}
\label{sec:irreducible.channels}
Consider a system with finite dimensional space $\mathcal H = \mathbb C^d$ and let $\mathcal{T}_*: L^1(\mathcal H)\to L^1(\mathcal H)$ be 
a quantum channel (in the Schr\"{o}dinger picture) cf. Theorem \ref{th:Kraus}. In this section, the channel is seen as a quantum analogue of a Markov transition operator, and the goal is to review fundamental ergodic properties which will play an inportant role in the statistical analysis of quantum Markov chains. For the notions and results presented in this section, we refer to \cite{EH78,FP09,Wo12}.

\vspace{2mm}

We fix a unitary dilation specified by an auxiliary Hilbert space $\mathcal K=\mathbb C^k$, a unit vector $\ket{\chi} \in {\cal K}$ and a unitary $U:\mathcal H\otimes {\cal K} \to \mathcal H\otimes \mathcal K$ such for any state $\rho\in S(\mathcal{H})$ 
$$
\mathcal{T}_* (\rho) = \tr_\mathcal{K}(U(\rho\otimes |\chi\rangle\langle \chi|)U^*).
$$ 
Let $V:{\cal H} \rightarrow {\cal H} \otimes {\cal K}$ be the isometry
\begin{eqnarray*}
V: |\psi\rangle  &\mapsto& U(|\psi\rangle \otimes |\chi\rangle) = 
\sum_{i=1}^k 
K_i|\psi\rangle \otimes |i\rangle
\label{eq:dilation.V}
\end{eqnarray*}
where $\{ |1\rangle , \dots , |k\rangle\}$ is an ONB in 
$\mathcal{K}$ and $K_i:= \langle i| U|\chi \rangle\in L^\infty(\mathcal{H})$ are Kraus operators satisfying $\sum_{i=1}^k K_i^*K_i = \mathbf{1}$. Then $\mathcal{T}_*$ depends on $U$ through the isometry $V$ and can be written as
$$
\mathcal{T}_*(\rho) = 
\tr_{\mathcal{H}}(V\rho V^*)= \sum_{i=1}^k K_i \rho K_i^*.
$$
The dual map $\mathcal{T}: L^\infty(\mathcal H)\to  L^\infty(\mathcal H)$ is then given by
$$
\mathcal{T}(X)=\bra{\chi} U^* (X \otimes \id_{\cal K}) U \ket{\chi}
=
V^*(X\otimes \id_{\cal K})V 
= \sum_{i=1}^k K_i^* X K_i .
$$
%

By the quantum Perron-Frobenius Theorem \cite{Wo12},  
for any channel $\mathcal{T}_*$, the set of  eigenvalues is contained in the unit complex disk, and includes the eigenvalue $\lambda =1$. Moreover the  channel has a positive eigenvector with eigenvalue $1$, which means that there exists a state $\rho^{\rm ss}$ which is invariant (stationary) for the channel i.e.   
$\mathcal{T}_*(\rho^{\rm ss}) = \rho^{\rm ss}$. In this paper  will be particularly interested the class of irreducible channels.

\begin{definition}[{\bf irreducible channel}]
    The channel $\mathcal{T}_*$ is called \emph{irreducible} if it has a unique faithful invariant state, i.e. there exists a unique state $\rho^{\rm ss}>0$ such that $\mathcal{T}_*(\rho^{\rm ss}) = \rho^{\rm ss}$.
\end{definition}
For irreducible channels, time averages converge to the stationary mean 
\begin{equation}
\label{average.Schrodinger}
\lim_{n\rightarrow \infty} 
\frac1n \sum_{k=1}^n 
\mathcal{T}_*^k(\rho) = 
\rho^{\rm ss}, 
\quad \text{ for all }\rho\in L^1(\mathcal{H}),
\end{equation}
In the Heisenberg picture, this is equivalently expressed as 
\begin{equation}\label{average.Heisenberg}
\lim_{n\rightarrow \infty} \frac1n \sum_{k=1}^n \mathcal{T}^k(X) = {\rm tr}[\rho^{\rm ss} X]\id_{\cal H}, \quad \text{ for all } X\in L^\infty ({\cal H}).
\end{equation}
In general however, the sequence $\mathcal{T}_*^k(\rho)$ may not converge. 
This depends on whether 
$\mathcal{T}_*$ exhibits periodic behaviour, which is determined by the nature of its \emph{peripheral spectrum}, the set of eigenvalues with absolute value equal to one.

\begin{theorem}[{\bf Periodic decomposition for irreducible channels}]
\label{th:periodic.structure}
Let $\mathcal{T}_*$ be an irreducible channel. The following statements hold.

\begin{enumerate}
\item[i)]
There exists an integer $p$, called the period,  such that the peripheral spectrum of $\mathcal{T}_*$  is 
$\{\gamma^{i}\}_{i=0}^{p-1}$ where 
$\gamma:=e^{2\pi i/p}$, and every eigenvalue is algebraically simple. 

\item[ii)] 
There exists a unique decomposition  
\begin{equation} \label{eq:perdec}
\mathcal{H}= \bigoplus_{a=0}^{p-1} \mathcal{H}_a
\end{equation}
such that the eigenvectors of 
$\mathcal{T}$ and $\mathcal{T_*}$  corresponding to $\gamma^i$ are given respectively by
\begin{equation} \label{eq:pereigen}
Z^i=\sum_{a=0}^{p-1} \gamma^{ai} P_a, \qquad J_i= \sum_{a=0}^{p-1} \overline{\gamma}^{ai} \rho_a^{\rm ss}, \qquad \text{for all } i=0,\dots p-1,
\end{equation}
where $P_a$ denotes the projection onto the subspace $\mathcal{H}_a$, and $\rho^{\rm ss }_a=P_a \rho^{\rm ss} P_a$.
\item[iii)] 
The action of $\mathcal{T}$ is such that 
$\mathcal{T}(P_{a\oplus 1}) = P_a$ for $a=0,\dots,p-1$. 

\end{enumerate}
\end{theorem}
We refer to \cite{Wo12} for  details about the proof of Theorem \ref{th:periodic.structure}, but note that the Theorem implies that  $\mathcal{T}_*$ jumps between the subspaces $\mathcal{H}_a$
$$\TT_*(P_a \rho P_a) = P_{a\oplus1} \TT_*(\rho) P_{a\oplus1} , \qquad a=0,\dots,p-1.
$$

Indeed, from $\TT(P_{a\oplus1}) = P_a$ one gets $(\id_{\cal H} - P_{a\oplus1} ) K_i P_a = 0 $ for each $i$, which means that each $K_i$ maps $\mathcal{H}_a$ into $\mathcal{H}_{a\oplus1}$. Another consequence is that the stationary state commutes with all $P_a$, so 
\begin{equation}
\label{eq:spectral.decomp.rho.ss}
\rho^{\rm ss} =\sum_{a=0}^{p-1} \rho^{\rm ss}_a, \qquad \rho^{\rm ss }_a=  \sum_{i\in I_a} \pi^{a}_i |\phi^a_i\rangle \langle \phi^a_i |,
\end{equation}
where $\pi_i^a>0$ are the eigenvalues of $\rho^{\rm ss}$ (possibly repeated) and $\ket{\phi_i^a} \in {\cal H}_a$ the corresponding eigenvector. Indeed, 
$$
 \TT_*\left ( \sum_{a=0}^{p-1} P_a\rho^{\rm ss}P_a \right) =  \sum_{a=0}^{p-1} P_{a\oplus1} \TT_*(\rho^{\rm ss}) P_{a\oplus1} =  \sum_{a=0}^{p-1} P_{a\oplus1} \rho^{\rm ss} P_{a\oplus1}
$$
so $\sum_{a=0}^{p-1}\rho^{\rm ss}_a$ is equal to the unique stationary state $\rho^{\rm ss}$. Note that $\TT_*(\rho_a^{\rm ss})=\rho_{a\oplus1}^{\rm ss}$ and $\tr(\rho_a^{\rm ss})=1/p$.


The commutative *-algebra generated by $Z$ (or equivalently by the $P_a$'s) is called the \emph{decoherence-free algebra}  of 
$\TT$, and is also the algebra of the fixed points of the quantum channel $\TT^p$. While $\mathcal{T}_*^n$ may not converge, the following two limit statement holds true for any initial state $\rho$
    
\begin{equation}
\lim_{n \rightarrow \infty}~ 
    \left\|\TT^n_*(\rho)-\sum_{a=0}^{p-1} \tr(\rho P_{a \ominus n}) \rho^{\rm ss}_a\right\|_1=0,
   \label{eq:dfs}
   \end{equation}
and in particular for $n=pl$ we get 
    \begin{equation}
    \lim_{l \rightarrow \infty} \TT_*^{pl}(\rho)=\mathcal{E}_*(\rho):=  p\sum_{a=0}^{p-1} \tr(\rho P_a)\rho_a^{\rm ss}. 
    \label{eq:erg}
    \end{equation}
Note that equation \eqref{eq:dfs} implies that irrespective its initial state after a sufficiently long time, the system will be approximately in a convex combination of the $\rho_a^{\rm ss}$s. If $p=1$, $\TT$ is called \emph{aperiodic or primitive} and one has that for every initial state $\rho$
$$\lim_{n \rightarrow +\infty} \TT^n_*(\rho)=\rho^{\rm ss}.$$

The reader will recognise that these properties generalise well known facts about transition matrices of irreducible classical Markov chains. In fact one can consider the classical case as a special case of the quantum setting described above. Indeed, given a transition matrix 
$T= [T_{ij}] \in M_d$ we can define the quantum channel 
$\mathcal{T}:L^{\infty}(\mathcal{H})\to L^\infty(\mathcal{H})$ with Kraus operators 
$K_{ij} = \sqrt{T_{ij}} |j\rangle \langle i|$ for $i,j=1,\dots d$ which leaves the commutative diagonal algebra 
$\mathcal{D}:= \{ D= \sum_{i} d_i |i\rangle \langle i| ~:~ d_i\in \mathbb{C}\} \subset L^\infty(\mathcal{H})$ invariant, and on which it has the same action as $T$ i.e. 
$\left[\mathcal{T}(D)\right]_{ii} = \sum_j T_{ij} d_j$. It can be shown that if $T$ is irreducible then $\mathcal{T}$ is as well, and both have the same period. In addition the 
stationary components $\rho^{\rm ss}_a$ of $\mathcal{T}$ are diagonal matrices corresponding to the stationary components of the stationary state of $T$.


\subsubsection{Quantum Markov chains}
\label{sec:QMC}
Using the same setup as in section \ref{sec:irreducible.channels} we introduce the notion of a quantum Markov chain (QMC) generated as a result of repeated interactions between the system and input ``environment units'', as illustrated in Figure \ref{Fig:i-o}. 

Consider the system $\mathcal{H}=\mathbb{C}^d$ interacting successively with environment units $\mathcal{K}=\mathbb{C}^k$ prepared independently in the same state $|\chi\rangle$. We assume that each interaction is described by the same unitary $U$ on $\mathcal{H}\otimes \mathcal{K}$. 
%
\begin{figure}[h]
\begin{center}
\includegraphics[width=10cm]{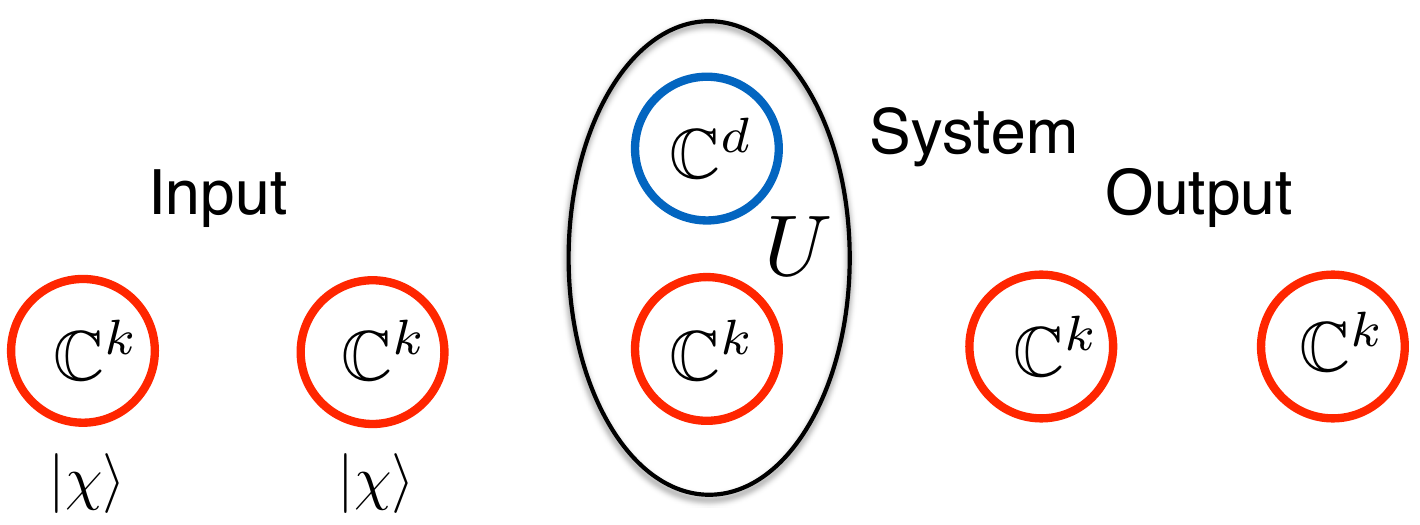}
\end{center}
\caption{A discrete-time quantum Markov chain. A sequence of identically prepared input units interact successively with a system via the unitary $U$. After the interaction, the output units are correlated and carry information about the unitary $U$.
}\label{Fig:i-o}
\end{figure}
If the initial system state is $|\varphi\rangle\in \mathcal{H} $ then after one time step the joint system-output state transforms as
$$
U: |\chi\otimes \varphi\rangle \mapsto 
U|\chi\otimes \varphi\rangle = V|\varphi\rangle =
\sum_{i=1}^k K_i |\varphi\rangle \otimes |i\rangle, 
$$
where $\{|1\rangle, \dots ,  |k\rangle\}$ is a fixed orthonormal basis in $\mathcal{K}$ and $K_i = \langle i|U|\chi\rangle $ are the Kraus operators of the channel $\TT$ introduced in section \ref{sec:irreducible.channels}.


After $n$ times steps the system and output state is
\begin{eqnarray} \label{eq:sostate}
|\Psi(n)\rangle &=& U(n) |\varphi\otimes \chi^{\otimes n}\rangle = 
U^{(n)}\cdot\dots \cdot U^{(2)}\cdot U^{(1)}|\varphi\otimes \chi^{\otimes n}\rangle\nonumber \\
&=&\sum_{i_1,\dots , i_n =1}^k 
K_{i_n}\dots K_{i_1} |\varphi\rangle \otimes |i_{1}\rangle\otimes \dots\otimes |i_{n}\rangle\nonumber\\
&=& \sum_{{\bf i}\in I_n} K^{(n)}_{\bf i} |\varphi\rangle \otimes |{\bf i}\rangle =:
 V(n)|\varphi\rangle\in \mathcal{H}\otimes \mathcal{K}^{\otimes n}\label{eq:joint.state}
\end{eqnarray}
where $U^{(i)}$ is the unitary acting on the system and the $i$-th input unit, we used the compact notation ${\bf i} := (i_1,\dots ,i_n)\in I_n:= \{1,\dots ,k\}^n$ and $ K^{(n)}_{\bf i} := K_{i_n}\dots K_{i_1}$, and the last equality defines the iterated version of $V(n)$ of the isometry $V$. By sequentially tracing out the environment units in equation \eqref{eq:joint.state} we find that the reduced state of the system at time $n$ is  
$$
\rho^{\rm sys}(n) := 
\tr_{\mathcal{K}^{\otimes n }}
(|\Psi(n)\rangle \langle \Psi(n)|) = 
\TT^n_*(\rho_{\rm in}), \qquad \rho_{\rm in} = |\varphi\rangle\langle \varphi|,
$$
where the partial trace is taken over the output units. Note that this is similar to classical Markov chains where the $n$-steps evolution is described by the $n$-th power of the transition operator. On the other hand, by tracing out the system in equation \eqref{eq:joint.state} we obtain the \emph{output} state, i.e. state of the 
$n$ ``environmental units'', after the interaction with the system
\begin{eqnarray}
    \label{eq:output.state}
\rho^{\rm out} (n) 
&= &
\tr_{\mathcal{H}}\left[U(n) (\rho_{\rm in} \otimes \tau^{\otimes n}) U(n)^*\right]= \tr_{\mathcal{H}}\left[V(n) \rho_{\rm in} V(n)^*\right]\\
&=&
    \sum_{{\bf i}, {\bf j}} {\rm Tr}(K_{\bf i} \rho_{\rm in} K^*_{\bf j})
|{\bf i}\rangle\langle {\bf j}|.\end{eqnarray}
In practical settings, the system is often not directly accessible so the QMC should be seen as a black-box input-output system where the black-box dynamics is probed by performing measurements on the output state. For this reason, our investigation will focus on understanding the structure of the output state from an asymptotic statistical viewpoint. This structure depends strongly on the ergodic properties of the dynamics, and throughout the paper we assume that the Markov operator $\TT$ is irreducible, which is satisfied in many physical settings. 

\section{Identifiability of stationary irreducible QMCs}
\label{sec:identifiability}

In this section we develop the identifiability theory for \emph{irreducible} QMCs as a preliminary step in the estimation problem. Along the way, we compare our results with those for \emph{primitive} QMCs \cite{GK15}. The main results are the equivalence between the different notions of equivalent parameters (Proposition \ref{prop:equivalence}) and Theorem \ref{thm:identification}, which gives a complete characterisation of QMCs with the same stationary output. We conclude the section recalling that a partial counterpart of Theorem \ref{thm:identification} was proved for the classical counterpart of the models studied in this work, i.e. hidden Markov chains.

\vspace{2mm}

Consider an irreducible QMC specified by the isometry $V$ as outlined in sections \ref{sec:irreducible.channels} 
and \ref{sec:QMC}, and let $\rho_V^{\rm ss}$ be its unique stationary state and denote by 
$\rho^{\rm out}_V(n)$ the output state \eqref{eq:output.state} when the initial system state is $\rho_{\rm in} = \rho_V^{\rm ss}$. In this case the output state is stationary in time. To motivate this choice of initial state for the system, note that if the QMC is primitive, any system initial state converges to $\rho_V^{\rm ss}$ exponentially fast, which means that the stationary output is the relevant model in an asymptotic analysis. We will therefore start by introducing an equivalence relation between QMCs based on the distinguishability of  their stationary output states (Definition \ref{def:equiv}). However, this notion of identifiability requires further justification when one considers periodic QMCs, for which convergence to stationarity does not generally hold. We show that the same equivalence notion emerges from two other statistical settings: the estimation of dynamical parameters without any information about the initial state of the system (Definition \ref{def:waek.equivalence}), or having access only to macroscopic fluctuations of the output (Definition \ref{def:macro.equivalence}).
\begin{definition}[\textbf{Stationary output equivalence}] 
\label{def:equiv}
Consider two irreducible quantum Markov chains with isometries $V_1:\mathcal{H}_1\to \mathcal{H}_1\otimes \mathcal{K}$ and 
$V_2:\mathcal{H}_2\to \mathcal{H}_2\otimes \mathcal{K}$ having the same output space $\mathcal{K}$ but possibly different system spaces $\mathcal{H}_1$ and $\mathcal{H}_2$. We say that the chains/isometries are \emph{output equivalent} if their corresponding stationary output states are identical
$$
\rho^{\rm out}_{V_1}(n) = \rho^{\rm out}_{V_2}(n), {\rm ~for~all~} n\in \mathbb{N}.
$$

\end{definition}

As pointed out earlier, the stationary output is an appropriate asymptotic model in the case of primitive QMC, but needs some justification in the case of irreducible QMCs which exhibit periodicity. In this case the output retains a dependence on the initial system state, even after long times. However, if the  initial system state is unknown (as it is often the case), multiple dynamics with different initial states may produce the same output state. This is the basis of our second equivalence definition.

\begin{definition}[\textbf{Weak output equivalence}]
\label{def:waek.equivalence}Consider two irreducible quantum Markov chains with isometries $V_1:\mathcal{H}_1\to \mathcal{H}_1\otimes \mathcal{K}$ and 
$V_2:\mathcal{H}_2\to \mathcal{H}_1\otimes \mathcal{K}$ having the same output space $\mathcal{K}$ but possibly different system spaces $\mathcal{H}_1$ and $\mathcal{H}_2$.  We say that the chains/isometries are \emph{weakly output equivalent} if there exists a pair of system states $(\rho_1,\rho_2)$ such that the corresponding outputs are identical
\[
\tr_{{\cal H}_1}(V_1(n)\rho_1 V_1(n)^*)=\tr_{{\cal H}_2}(V_2(n)\rho_2 V_2(n)^*), {\rm ~for~all~} n\in \mathbb{N}.
\]
\end{definition}
From Definitions \ref{def:equiv} and \ref{def:waek.equivalence} it follows that if two chains are stationary output equivalent, they are also weakly output equivalent, by choosing their stationary states as initial states. As we will show below, the converse is also true, so the two notions are equivalent.  

\vspace{2mm}

In order to formulate a third equivalent notion of indistinguishability, we need to introduce a family of physically meaningful observables which provide a rich class of statistics, namely time averages of local output observables. 

Let $Q\in L^\infty({\cal K}^{\otimes k})$ be a selfadjoint operator interpreted as an observable of a chain of $k$ subsequent output units, for some fixed $k\in \mathbb{N}$. Let $n$ be the output size such that $n\gg k$ and consider the observable $Q^{(i)}$ which acts as  
$Q$ on the output units 
$\{i,i+1,\dots i+k-1\}$ and identity on the rest of the units. We denote the stationary mean value of $Q$ and, respectively, the time correlations by  
\begin{align*}
m_V(Q) &= {\rm Tr}(\rho^{\rm out}_V (n)Q^{(i)}),\\
c_{a,V} (Q) &= {\rm Tr}\left[\rho^{\rm out}_V (n)(Q^{(i)}-m_V(Q)\mathbf{1})\circ ( Q^{(i+a)} -m_V(Q)\mathbf{1})\right],
\end{align*}
where $A\circ B$ denotes the symmetric product,  
$1\leq i,i+a\leq  n $. Note that both are independent of $n$ and $i$ (as far as $n \gg i$) due to stationarity. We define the \emph{time average} and, respectively, the \emph{fluctuation operator} associated to $Q$  by
$$
\overline{Q}_{n}:=\frac{1}{n-k+1}\sum_{l=1}^{n-k+1} Q^{(l)}, \qquad F_n(Q):=
\sqrt{n-k+1} (\overline{Q}_{n}- m_V (Q)).
$$ The following Lemma shows that these observables satisfy a Law of Large Numbers and a Central Limit Theorem, respectively.

\begin{lemma} \label{lem:ltflu}
Consider an irreducible quantum Markov chain with isometry $V:\mathcal{H}\to \mathcal{H}\otimes \mathcal{K}$. Let $\rho$ be a system state and let $\rho^{\rm out}_V(n,\rho)$ be the output state for the initial state $\rho$. Let 
$\mathbb{P}(F_n(Q), \rho^{\rm out}_V(n,\rho))$ and 
$\mathbb{P}(\overline{Q}_n, \rho^{\rm out}_V(n,\rho))$ denote the probability distributions of $F_n(Q)$ and respectively  $\overline{Q}_n$ with respect to $\rho^{\rm out}_V(n,\rho)$.

Then the following limits hold in the sense of convergence in distribution:
\begin{equation}
\lim_{n \rightarrow +\infty}
\mathbb{P}(F_n(Q), \rho^{\rm out}_V(n,\rho))
= {\cal N}(0,\sigma^2_{V}(Q)), \quad \lim_{n \rightarrow +\infty}
\mathbb{P}(\overline{Q}_n, \rho^{\rm out}_V(n,\rho))
    = \delta_{m_V(Q)},
\end{equation}
where the asymptotic variance of the Gaussian random variable is given by the following expression:
\begin{equation} \label{eq:asvar}
\sigma^2_{V}(Q):= c_{0,V}(Q)+2\sum_{l=1}^{+\infty}
c_{l,V}(Q).
\end{equation}

\end{lemma}
The proof of Lemma \ref{lem:ltflu} can be found in Appendix \ref{ap:fo}. This result extend similar ones for primitive dynamics \cite{Ma03} by allowing for periodic dynamics and arbitrary initial states.
Lemma \ref{lem:ltflu} shows that the asymptotic mean of 
$\overline{Q}_n$ forgets the initial state $\rho$ and is uniquely determined by the isometry $V$ through the output stationary state $\rho_V^{\rm out}$. In turn, the set of means $m_V(Q)$ completely characterises the stationary output, since the latter is determined by expectations of local observables. Therefore, a consequence of Lemma \ref{lem:ltflu} is that weak output equivalence (cf. Definition \ref{def:waek.equivalence}) implies implies strong output equivalence (cf. Definition \ref{def:equiv}) and the two definitions are equivalent. Another consequence is that the following definition is also equivalent to the previous ones.


\begin{definition}[{\bf Macroscopic output equivalence}] 
\label{def:macro.equivalence}
Consider two irreducible quantum Markov chains with isometries $V_1:\mathcal{H}_1\to \mathcal{H}_1\otimes \mathcal{K}$ and 
$V_2:\mathcal{H}_2\to \mathcal{H}_1\otimes \mathcal{K}$ having the same output space $\mathcal{K}$, but possibly different system spaces $\mathcal{H}_1$ and $\mathcal{H}_2$. We say that the chains/isometries are \textit{macroscopic output equivalent} if for every local observable $Q$ and pair of initial states $(\rho_1,\rho_2)$ the limit laws of $F_n(Q)$ and $\overline{Q}_n$ coincide, i.e. 
\begin{equation} \label{eq:asympeq}
m_{V_1}(Q)=m_{V_2}(Q), \quad \sigma_{Q,V_1}^2=\sigma_{Q,V_2}^2.\end{equation}
\end{definition}

Note that the condition stated in Eq. \eqref{eq:asympeq} is redundant, since one can easily see from equation \eqref{eq:asvar} that the equality of all asymptotic means implies that of asymptotic variances as well. Let us summarise in the following Proposition what we just discussed.
\begin{proposition} \label{prop:equivalence}
  Let $V_1:\mathcal{H}_1\to \mathcal{H}_1\otimes \mathcal{K}$ and 
$V_2:\mathcal{H}_1\to \mathcal{H}_1\otimes \mathcal{K}$ be two  isometries corresponding to irreducible QMCs. The following are equivalent:
\begin{enumerate}
\item[(i)]
$V_1$ and $V_2$ are stationary output equivalent;
\item[(ii)] $V_1$ and $V_2$ are weak output equivalent;
\item[(iii)] $V_1$ and $V_2$ are macroscopic output equivalent.
\end{enumerate}
If any of the previous equivalent conditions hold, we will simply say that the two isometries $V_1$ and $V_2$ are 
\emph{output equivalent} or that they are \emph{indistinguishable parameters}.
\end{proposition}

After motivating the notion of equivalent dynamics, we now proceed to characterise the corresponding equivalence classes. The following Theorem provides an explicit description of QMCs which give rise to the same stationary output. It generalizes the corresponding result in \cite{GK15} (Theorem 2), showing that the assumption of primitivity can be relaxed to irreducibility; this has also been observed in \cite{BCJPta}. The proof can be found in Appendix \ref{ap:oe}.
\begin{theorem}[{\bf Equivalent irreducible QMCs}]\label{thm:identification}
Let $V_1:\mathcal{H}_1\to \mathcal{H}_1\otimes \mathcal{K}$ and 
$V_2:\mathcal{H}_1\to \mathcal{H}_1\otimes \mathcal{K}$ be two  isometries corresponding to irreducible QMCs with stationarys states 
$\rho^{\rm ss}_1$ and respectively $\rho^{\rm ss}_2$. The following are equivalent:

\begin{itemize}
\item[(i)]
$V_1$ and $V_2$ are output equivalent;
\item[(ii)] There exists a unitary
$W: \mathcal{H}_2\to \mathcal{H}_1$ and a complex phase $e^{i\omega}$ such that 
$$
K_{2,i} = e^{i\omega} W^* K_{1,i} W \quad  \text{and} 
\quad 
\rho^{\rm ss}_2=W^*\rho^{\rm ss}_1W
$$ 

\item[(iii)]
There exists a unitary
$W: \mathcal{H}_2\to \mathcal{H}_1$ and a complex phase $e^{i\omega}$ such that 
$$
(W\otimes \mathbf{1}_{\cal K}) V_2 = e^{i\omega}V_1 W.
$$
\end{itemize}
\end{theorem}

We note in particular that equivalent irreducible QMCs are necessarily of the same system dimension. We point out that the set of stationary output states of irreducible QMCs studied here coincides with that of ergodic purely generated ${\rm C}^*$-finitely correlated states, cf. \cite{Ac81,AF83} and \cite{FNW92}, and therefore, all our results hold for this class of states. The latter are closely connected to ground states of local  Hamiltonians, which exhibit translational symmetry breaking when the corresponding QMC is periodic (see for instance Examples 4 and 6 in \cite{FNW92}).

\subsection{Identifiability for certain classical hidden Markov chains} 

Before ending this section, we recall that a result in the spirit of Theorem \ref{thm:identification} has also been obtained for the classical counterpart of the models studied in this work, i.e. hidden Markov chains.

Indeed, let us consider a Markov chain on a finite state space $E$ with transition matrix $P$ and suppose that one can only observe a function $f:E \rightarrow A$ of the state of the stochastic system, where $A$ is a finite set of labels. For every label $a \in A$, let us define the matrix $P^{(a)}$ whose entries are given by
$$P^{(a)}(x,y)=P(x,y) \delta_{f(x),a}.$$
Given a probability density $\nu$ on $E$, we choose to use the notation in which its evolution via the transition matrix is given by $P\nu$, i.e. $P(x,y)$ represents the probability that the process passes from state $y$ to state $x$. Therefore $P^{(a)}$ is the matrix containing probabilities of all possible transitions to a state which is labelled with $a$ by $f$.

Note that $P^{(a)}$ is substochastic and that $P=\sum_{a \in A}P^{(a)}$, which is the equivalent for transition matrices of the Kraus representation of a quantum channel; such a way of writing $P$ contains the key for finding the ``correct'' translation to the classical setting of irreducibility of the quantum channel: it is well known that irreducibility of a quantum channel $T$ with Kraus operators $\{K_i\}_{i \in I}$ is equivalent to the fact that for every non-zero vector $v \in {\cal H}$
\begin{equation} \label{eq:access}
{\rm span}\{K_{i_n} \cdots K_{i_1}v: \, n \in \mathbb{N}, \, i_1,\ \dots, i_n \in I\}={\cal H}.
\end{equation}

Therefore, we consider hidden Markov processes that satisfy the following condition: for every vector $v\in \mathbb{R}^{|E|}$
\begin{equation} \label{eq:irrHM}
{\rm span}\{P^{(a_n)} \cdots P^{(a_1)}v: \, n \in \mathbb{N}, \, a_1,\ \dots, a_n \in A\}=\RR^{|E|}.
\end{equation}
If $v$ is a probability density, in control theory literature the left hand side of previous equation is sometimes called  the reachable space starting from $v$. One can easily see that condition in Eq. \eqref{eq:irrHM} in general is strictly stronger than irreducibility of the underlying Markov process, therefore any transition matrix $P$ satisfying such condition has a unique invariant measure $\nu^{\rm ss}_P$ with full support. The marginal law of the first $n$ observed labels at stationarity is given by
$$\nu^{\rm out}_{P,f}(a_1,\dots, a_n)=\mathbf{1}^T P^{(a_n)} \cdots P^{(a_1)}\nu^{\rm ss}, \quad a_1,\dots,a_n \in A.$$
The following Theorem from \cite{IAK92} is a partial counterpart of our Theorem \ref{thm:identification}.

\begin{theorem} \label{thm:clid}
    Let $(E_i,P_i,f_i)$ for $i=1,2$ be two triples of state spaces, transition matrices and labelling functions (with the same alphabet set $A$), which satisfy the condition \eqref{eq:irrHM}. Then the two hidden Markov chains have the same stationary output states, i.e.
    $$\nu^{\rm out}_{P_1,f_1}(n)=\nu^{\rm out}_{P_2,f_2}(n) \quad \text{ for every } n \geq 1$$
    if and only if there exists a linear bijection $\Pi:\mathbb{R}^{|E_2|} \rightarrow \mathbb{R}^{|E_1|}$ such that
    \begin{itemize}
    \item $P^{(a)}_{2}=\Pi^{-1}P^{(a)}_1\Pi$ for every $a \in A$ and
    \item $\nu_2^{\rm ss}=\nu_1^{\rm ss}\Pi$.
    \end{itemize} 
    Such $\Pi$ is unique.
\end{theorem}




\section{Global geometric structure of identifiable stationary irreducible QMCs} \label{sec:global}

The upshot of the previous section is that the space of output-identifiable parameters 
of an unknown QMC is the quotient of the set $\manifirr$ of isometries $V:{\cal H} \rightarrow {\cal H}\otimes {\cal K}$ with respect to the equivalence relation introduced in Definition \ref{def:equiv}. We denote this quotient by $\orbifirr$. The equivalence was characterised in Theorem \ref{thm:identification} in terms of unitary conjugation and phase multiplication.


\vspace{2mm}

The goal of this section is to describe the global geometric structure of $\orbifirr$. Firstly, we show that $\manifirr$ can be endowed with a manifold structure; we then prove that there exists a group $\mathscr{G}$ of smooth transformations of such manifold and that the orbit space is homeomorphic to $\orbifirr$. The advantage of seeing $\orbifirr$ as an orbit space is that it carries naturally a structure of \emph{orbifold} (this is true under some assumptions on the group action, which we prove to hold). After a brief intermezzo in which we quickly recall the main definitions and results concerning orbifolds, the rest of the section is devoted to constructing a convenient set of charts and studying some geometric properties of $\orbifirr$. 

\subsection{Manifold of irreducible isometries and gauge group}
\label{sec:isometries.manifold}
Let us consider the set of isometries 
\begin{equation}
\label{eq:tangent.space.isometries}
\mathscr{V}=\{V\in L^\infty(\mathcal H,\mathcal H\otimes\mathcal K)~:~ V^*V=\id\}.
\end{equation}
The constant rank theorem (Corollary 5.9, page 80 in \cite{BO86}) shows that $\mathscr{V}$ is a submanifold of $L^\infty(\mathcal H,\mathcal H\otimes\mathcal K)$ of dimension $(2k-1)d^2$ and that the tangent space at a point $V \in \mathscr{V}$ can be identified with
\begin{equation} \label{eq:tangentmanirr} T_V(\mathscr{V})=\{L\in L^\infty(\mathcal H,\mathcal H\otimes K)~:~ L^*V-V^*L=0\}.
\end{equation}
so that, infinitesimally, an isometry in the neighbourhood of $V$ can be represented as $V+i\delta L$ for small real $\delta$.
We will introduce a block decomposition of $ L^\infty({\cal H}, {\cal H}\otimes {\cal K})$ that will allow us to formulate a more convenient description of $T_V(\mathscr{V})$: let us consider the range of $V$, which we denote by $R(V)$ and the block decomposition of $L^\infty(\mathcal H,\mathcal H\otimes {\cal K})$ induced by the orthogonal decomposition ${\cal H} \otimes {\cal K}=R(V) \oplus R(V)^\perp$:
$$L^\infty(\mathcal H,\mathcal H\otimes {\cal K}) \ni L \Leftrightarrow \begin{pmatrix} L_1:=V V^* L \\ L_2:=(\mathbf{1}_{{\cal H} \otimes {\cal K}}-VV^*)L\end{pmatrix},$$
where $V^*V$ is the projection onto $R(V)$, and with an abuse of notation we identified an operator in $L^\infty(\mathcal H,\mathcal H\otimes {\cal K})$ with range included in $R(V)$ with an operator in $L^\infty(\mathcal H,R(V))$ and similarly for $R(V)^\perp$. The condition \eqref{eq:tangent.space.isometries} on tangent vectors does not impose any restrictions on $L_2$, while it translates in terms of $L_1$ into $V^*L_1$ (which is an operator acting on ${\cal H}$) to be Hermitian. From this point of view, it is also clear how to compute the dimension of the manifold and that the dimension of $T_V(\mathscr{V})$ does not change with $V$.

We then consider the subset corresponding to isometries inducing irreducible quantum channels:
$$
\manifirr = \{V\in \mathscr{V} ~:~ \TT_V \text{ irreducible }\}.
$$
Since the set of irreducible channels is open, and the mapping $V\mapsto \TT_V$ is continuous, it follows that $\manifirr$ is an open submanifold of the manifold of $\mathscr{V}$. This manifold is the starting point of the geometric structure we will define in this section. Theorem \ref{thm:identification} shows that the same output state can correspond to different isometries, hence we define the quotient topological space
$$\orbifirr:=\{[V] ~:~ V \in \manifirr\},
$$



\bigskip We  now introduce a group of transformations of $\manifirr$ such that the orbit space is homeomorphic to $\orbifirr$. Let us consider the compact Lie group given by
\begin{equation}\label{symmetries}
\mathscr{G}=U(1)\times PU(d)
\end{equation}
and its action on the parameter manifold defined as follows:
\begin{align}
& \mu: \mathscr{G}\times \manifirr \to \manifirr, & (g,V)\mapsto \mu(g,V)=g \cdot V &:= \overline{c} (W\otimes \id)VW^*
\label{eq:group.action}\\
& & g&=(c,W)\in \mathscr{G}, \quad V\in \manifirr
\nonumber
\end{align}
where $W$ is understood as the equivalence class of unitaries related by a complex phase.

Clearly, the map $g\mapsto g \cdot V$ satisfies the required composition rule with respect to the group multiplication. Given an element $V \in \manifirr$, its \emph{stabiliser}
is defined as the following Lie subgroup
\begin{equation}
\label{eq:stabiliser_group}
\mathscr{G}_V:=\{g \in \mathscr{G}: \, g \cdot V=V\}.
\end{equation}

The action $\mu$ is said to be \emph{almost free} if $\mathscr{G}_V$ is finite at any $V$ and \emph{effective} if $\bigcap_{V \in \manifirr}\mathscr{G}_V=\{e\}$, where $e=(1,\mathbf{1}_{\cal H})$ is the unit of $\mathscr{G}$.

By $\manifirr/\mathscr{G}$ we denote the topological quotient space of $\manifirr$ induced by the equivalence relation that identifies two isometries $V_1$ and $V_2$ if there exists $g \in \mathscr{G}$ such that $g \cdot V_1=V_2$.
The following Theorem, whose proof can be found in Appendix \ref{sec:proofsgeom}, says that $\mathscr{G}$ is exactly the gauge group of our identification problem and characterises the stabiliser at each point.
\begin{theorem}\label{almostfree} The following statements hold true:
\begin{enumerate}
    \item The action of $\mathscr{G}$ on $\manifirr$ is effective and almost free;
    \item if $p_V$ is the period of $\TT_V$, $\gamma_V:=e^{i2\pi/p_V}$, and $Z_V$ is the eigenvector of $\TT_V$ corresponding to $\gamma_V$ defined in Eq. \eqref{eq:pereigen}, then one has
    $$\mathscr{G}_V=\{(\gamma_V^k,Z_V^k): \, k=0,\dots, p_V-1\}\simeq \mathbb{Z}_{p_V};$$
    \item $\mathcal P^{\rm irr}=\manifirr / \mathscr{G}$ as topological spaces.
\end{enumerate}
\end{theorem}

An important consequence of Theorem \ref{almostfree} is the following. Given $V \in \manifirr$, let us denote by $(d_0(V), \dots, d_{p_V-1}(V))$ the dimensions of the subspaces appearing in the decomposition of the Hilbert space related to the period (Eq. \eqref{eq:perdec}); using that $[V]$ is the orbit of $V$ under the action of $\mathscr{G}$, it is immediate to see that both $p_V$ and $(d_0(V), \dots, d_{p-1}(V))$ do not depend on the particular representative of $[V]$, hence we can safely use the notation $p_{[V]}$ and $(d_0([V]), \dots, d_{p_{[V]}-1}([V]))$.

\subsection{Intermezzo on orbifolds} \label{subsec:orbi}

If a compact Lie group $\mathscr{G}$ acts smoothly, effectively, and freely on a manifold $\mathscr{M}$, then there is a natural manifold structure on $\mathscr{M}/\mathscr{G}$ such that $\mathscr{M}$ is a principal $\mathscr{G}$-bundle over $\mathscr{M}/\mathscr{G}$ \cite{BO86}; This result was used in \cite{GK17} to develop the information geometry theory for continuous time quantum Markov dynamics. However, Theorem \ref{almostfree} shows that in the model we study here, the action of $\mathscr{G}$ is only \emph{almost free}; this implies that $\mathscr{M}/\mathscr{G}$ has a more intricate geometric structure, which can be understood using the concept of \emph{orbifolds}. Loosely speaking, an orbifold is a topological space which is locally described by the quotient of an euclidean space under the smooth action of a finite group; hence, it is a very natural notion in dynamical systems and information geometry, when one considers the space of configurations/parameters and wants to reduce the degrees of freedom using the symmetries of the dynamics/states.

The goal of this brief intermezzo is to introduce the notion of orbifolds and to present the main feature and results we will need in the rest of this section. For the benefit of readers not familiar with these concepts, we also include a simple example. 


There are several approaches to orbifold theory: we decided to follow the one in \cite{Ca22,ALR07,Dr11} and references therein, which is closer to the usual way smooth manifolds are defined. Let $X$ be a topological space; an \emph{orbifold chart} of dimension $n$ for an open connected set $U \subseteq X$ is a triple $(\widetilde{U},H, \phi)$ where
\begin{itemize}
\item $\widetilde{U}$ is an open connected subset of $\mathbb{R}^n$;
\item $H$ is a finite group acting smoothly and effectively on $\widetilde{U}$;
\item $\phi:\widetilde{U} \rightarrow U $ is a continuous $H$-invariant surjection;
\item $\phi$ induces an homeomorphism between $\widetilde{U}/H$ and $U$.
\end{itemize}
We recall that the action of $H$ is said to be effective if there is no $h \in H$ other than the unit element such that $h \cdot x =x $ for every $x \in \widetilde{U}$ (by $h \cdot x$ we denote the action of $h$ on $x$).

\vspace{2mm}

An \emph{embedding between orbifold charts} $(\widetilde{U}_i,H_i, \phi_i)$ for $i=1,2$ is given by a smooth embedding $\lambda :\widetilde{U}_1 \rightarrow \widetilde{U}_2$ such that $\phi_1=\phi_2 \circ \lambda$; a nontrivial consequence of this definition is that every embedding between charts induces an injective group homomorphism between $H_1$ and $H_2$ or, in other words, there exists a subgroup of $H_2$ which is isomorphic to $H_1$. Two charts $(\widetilde{U}_i,H_i, \phi_i)$ for $i=1,2$ are said to be \emph{compatible} if for any $ x\in U_1 \cap U_2$, there exists a neighborhood $x \in U_3\subseteq U_1 \cap U_2$ and an orbifold chart $(\widetilde{U}_3,H_3, \phi_3)$ for $U_3$ which admits embeddings in $(\widetilde{U}_i,H_i, \phi_i)$ for $i=1,2$.

\vspace{2mm}

A $n$-dimensional \emph{orbifold atlas} for $X$ is a collection ${\cal A} = \{(\widetilde{U}_i,H_i, \phi_i)\}$ of compatible orbifold charts of dimension $n$ that covers $X$; we
say that the atlas ${\cal A}$ \emph{refines} an atlas ${\cal B}$ when every chart in ${\cal A}$ admits an embedding in
some chart in ${\cal B}$. Two atlases are \emph{equivalent} if they have a common refinement.
\begin{definition}[{\bf orbifold}]
    An \emph{n-dimensional orbifold}  consists of a Hausdorff paracompact topological space $X$ together with an equivalence class $[{\cal A}]$ of $n$-dimensional orbifold atlases for $X$.
\end{definition}
We remark that a smooth manifold is an orbifold in which, for every chart, the group $H$ is trivial. The most prominent example of orbifold is the case of quotients of manifolds with respect to the action of a Lie group.

\begin{theorem}[Proposition 1.5.1 in \cite{Ca22}, Definition 1.7 in \cite{ALR07}] \label{thm:orbi}
Let $\mathscr{G}$ be a compact Lie group  which acts smoothly, effectively and almost
freely on a smooth manifold $\mathscr{M}$. Then the quotient space $\mathscr{M}/\mathscr{G}$ has a natural orbifold structure.
\end{theorem}
The compactness assumption on $\mathscr{G}$ in the previous theorem can be weakened, however this result suffices and is more convenient for our purposes. Let us denote the action of $\mathscr{G}$ on $\mathscr{M}$ with $\mu:\mathscr{G}\times \mathscr{M} \rightarrow \mathscr{M}$. We recall that the action of $\mathscr{G}$ is almost free if the \emph{stabiliser} subgroup $\mathscr{G}_x$ at any point $x \in \mathscr{M}$ 
is finite, where 
$$
\mathscr{G}_x:=\{g \in \mathscr{G}: \mu(g,x)=x\}.
$$

The proof of Theorem \ref{thm:orbi} provides the explicit construction of an atlas with a rich structure. Since such a construction will play a crucial role in this paper, we briefly sketch it here. For every point $y \in \mathscr{M}$, one can consider the orbit of $y$, i.e. $$\mathscr{G}(y):=\{z \in {\cal M}: \, \exists g \in \mathscr{G}: \mu(g , y)=z\}.$$ 
Note that $\mathscr{G}(y)$ is an equivalence class, hence $z\in \mathscr{G}(y)$ if and only if $\mathscr{G}(z)=\mathscr{G}(y)$. Moreover, $\mathscr{G}(y)$ is a submanifold of ${\cal M}$ (Proposition 3.28 in \cite{AB10}) and for every $z \in \mathscr{G}(y)$, the tangent space of $\mathscr{G}(y)$ at $z$, which we denote by $T_z(\mathscr{G}(y))$, is given by the image of the Lie algebra ${\frak g}$ of the group via the differential of the action, that is
$$T_z(\mathscr{G}(y)):=\{d_e\mu(\cdot,z)A~:~
A \in {\frak g}\}.
$$
We use the notation $e$ for the identity of $G$ and $d_e $ to denote the differential at $e$, so that the explicit action of $d_e\mu(\cdot,z)A$ on a smooth function $f$ is
$$
\left(d_e\mu(\cdot,z)A \right)[f]= 
\left.\frac{d}{dt} \right|_{t=0} f(\mu(\exp(tA),z)).
$$

Note that every point in $\mathscr{G}(y)$ will be identified in $\mathscr{M}/\mathscr{G}$, therefore, in order to construct a set of orbifold charts for $\mathscr{M}/\mathscr{G}$, one needs to restrict to subsets of $\mathscr{M}$ that contain only finitely many points that belong to the same orbit, which is what we do in the following. Since $\mathscr{G}$ is compact, it is possible to endow $\mathscr{M}$ with a Riemannian metric such that $\mathscr{G}$ acts on $\mathscr{M}$ as a group of isometries (the choice of the metric is not unique), and the associated geodesics allow one to locally define the exponential map from a neighborhood of the origin of the tangent space to $\mathscr M$. Now, at any point $y \in \mathscr{M}$, the tangent space can be written as the orthogonal sum of $T_y(\mathscr{G}(y))$ and its orthogonal complement $T_y(\mathscr{G}(y))^\perp$. Furthermore, by means of the exponential map we can parametrise a small neighborhood of $y$ using a small centered open ball of radius $\epsilon$ $B_\epsilon(y) \subset T_y(\mathscr{M})$. Let us consider an open ball $B^\perp_\epsilon(y)=B_\epsilon(y) \cap T_y(\mathscr{G}(y))^\perp$ and define the following elements:
\begin{itemize}
\item $\widetilde{U}:=B^\perp_\epsilon(y)$;
\item $H=\{d_y\mu(g,\cdot )  : \, g \in \mathscr{G}_y\}$ is a finite group acting on $\widetilde{U}$;
\item $\phi:=\pi \circ \exp_y:B^\perp_\epsilon(y) \rightarrow \mathscr{M}/\mathscr{G}$, where $\pi:\mathscr{M}\rightarrow \mathscr{M}/\mathscr{G}$ is the quotient map;
\item $U:=\pi \circ \exp_y(B^\perp_\epsilon(y)).$
\end{itemize}
One can show that if $\epsilon$ is small enough, $(\widetilde{U}, H, \phi)$ is an orbifold chart for $U$ and that all such charts are compatible (this relies on the Slice and Tubular Neighborhood Theorems, see Section 3.2 in  \cite{AB10}); we remark that a fundamental role is played by the property
$$\mu(g,\exp_y(\cdot))=\exp_y \circ d_y\mu(g,\cdot),$$
which holds thank to the fact that $\mathscr{G}$ acts as a group of isometries (which preserve geodesics) and that $t \mapsto \exp_y(tA)$ is the unique geodesic with tangent vector $A$ in $y$. Such charts have the following nice properties: they are said to be 
\begin{itemize}
    \item \emph{linear}, since $H$ acts on $\widetilde{U}$ as a group of orthogonal linear transformations;
    \item \emph{fundamental} for $\phi(y)$, since $y$ is a fixed point under the action of $H$ or, equivalently, $H=H_y$, where $H_y$ denotes the stabiliser group of $y$.
\end{itemize}

Returning to the case of a general orbifold $X$, let us further comment on the notion of fundamental chart. Let $x \in X$ be fixed, and consider a chart $(\widetilde{U}, H, \phi)$ and a point $\widetilde{x}\in \widetilde{U}$ such that $\phi(\widetilde{x})=x$, and let $H_{\widetilde{x}}$ be the stabiliser of $\widetilde{x}$. 
It can be shown that the isomorphism class of the stabiliser $H_{\widetilde{x}}$ is independent of the chart or point $\widetilde{x}$ and we call this the \emph{local group} $\Gamma_x$.
This implies that for every chart $(\widetilde{U}, H, \phi)$ that parametrizes a neighborhood of $x$, $H$ must contain a subgroup in $\Gamma_x$; therefore, a fundamental chart for $x$ is a chart with the ``minimal'' group of transformations.

\vspace{2mm}

There are different ways of defining the equivalent of a tangent bundle for an orbifold, which is known as tangent orbibundle. Since in this work we will not need the global construction, we will just define the generalization of the tangent space at a point. Let us consider a point $x$ in the orbifold and a chart $(\widetilde{U}, H, \phi)$ such that $x =\phi(\widetilde{x})$; the action of the stabilizer $H_{\widetilde{x}}$ of $\widetilde{x}$, induces the following action on the tangent space of $\widetilde{U}$ in $\widetilde{x}$, which we will denote by $\TT_{\widetilde{x}}(\widetilde{U})$:
$$H_{\widetilde{x}} \times T_{\widetilde{x}}(\widetilde{U}) \ni (h,v) \mapsto d_{\widetilde{x}} \mu(h,\cdot) v.$$
The \emph{tangent cone} at $x$ is given by $T_{x}(X):=T_{\widetilde{x}}(\widetilde{U}) / H_{\widetilde{x}}$; the tangent cone is independent of the choice of point in the fiber of $x$ in the following sense: if one picks a different point $\hat{x}$ such that $x=\phi(\hat{x})$, then there exists a group isomorphism $\alpha:H_{\widetilde{x}} \rightarrow H_{\hat{x}}$ and a linear isomorphism $f:T_{\widetilde{x}}(\widetilde{U}) \rightarrow T_{\hat{x}}(\widetilde{U})$ such that
$$fd_{\widetilde{x}}\mu(h,\cdot) v=d_{\hat{x}}\mu(\alpha(h),\cdot) fv, \quad \forall h \in H_{\widetilde{x}}, \, v \in T_{\widetilde{x}}(\widetilde{U}).$$
Analogously one can check that the tangent cone does not depend on the choice of the chart.

In general, $T_{x}(X)$ is not a vector space; however there is an important linear subspace of $T_{\widetilde{x}}(\widetilde{U})$ which is given by those vectors which are fixed by $d_{\widetilde{x}}\mu(h,\cdot)$ for every $h \in H_{\widetilde{x}}$: due to what we observed above, up to linear isomorphisms, such a subspace is independent of the choice of the point in the fiber of $x$ and the chart and we denote its isomorphism class of linear subspaces with $T_x(X)^{\Gamma_x}$ and we call it the \emph{tangent space at $x$}. The \emph{singular dimension} of $x$ is the dimension of $T_x(X)^{\Gamma_x}$; we denote by $\mathscr{S}_k$ the set of points with singular dimension $k$ for $k=0, \dots, n$, where $n$ is the dimension of the orbifold. Therefore, $X$ is the disjoint union of the $\mathscr{S}_k$'s, which is called the \emph{canonical stratification}. The set $\mathscr{S}_n$ is known as the set of \emph{regular points} and can be shown to be a dense open set in $X$ (see Proposition 2.8 \cite{Dr11}), while its complementary is called set of \emph{singular points}. In addition, the $\mathscr{S}_k$'s have the following further structure.
\begin{proposition}[Proposition 3.4 \cite{Dr11}, Section 4.5 \cite{CR01}, \cite{Ka78}]
For $k=0,\dots, n$, $\mathscr{S}_k$ has naturally the structure of a $k$-dimensional smooth manifold. The tangent space $T_x(\mathscr{S}_k)$ at a point
$x \in \mathscr{S}_k$ is canonically identified with $T_x(X)^{\Gamma_x}$, the space of tangent vectors at $x$.
\end{proposition}
We remark that $\Gamma_x$ is constant along connected components of $\mathscr{S}_k$'s; however, there can be two points with same local group, but sitting on two different submanifolds.

\subsubsection*{A simple orbifold example} 
\begin{figure*}[!ht]
    \centering
   \includegraphics[width=\linewidth]{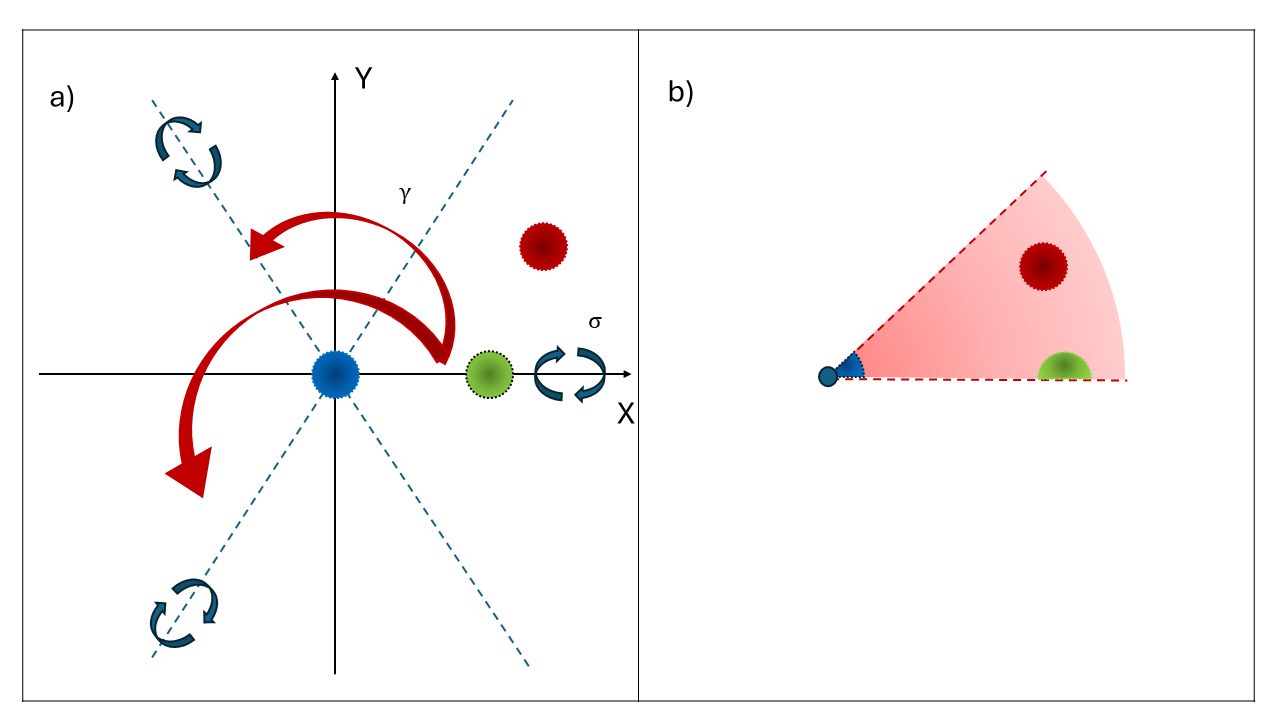}
    \caption{In Figure a) the axis represents a choice of coordinates for $\mathbb{R}^2$, while the red  sector in Figure b) is a graphic representation of $\mathbb{R}^2/\langle \gamma,\sigma \rangle$. Arrows in figure a) represent the five transformations in the dihedral group minus the identity: red arrows correspond to rotations and blue arrows to symmetries. Disks in Figure a) and b) represent charts and their image: the blue is a chart for $\pi((0,0))$, the green for $\pi(x)$ for some point $x$ laying on the $X$ axis and the red one for a point in $A_2$. The blue dot in Figure b) represents $\Sigma_0$, the red dotted half-lines (without the origin) represent $\Sigma_1$ and all the other points are regular points.}
    \label{fig:orbi}
\end{figure*}

To make all previously introduced orbifold notions  easier to digest, we provide a simple example. Let us consider the orbifold originating from the quotient of $\mathbb{R}^2$ under the action of the group of transformations $\mathscr{G}$ generated by $\gamma$, the rotation around the origin of an angle of $2\pi/3$, and by $\sigma$ the symmetry around the $X$ axis; we will use the notation $\mathscr{G}=\langle \gamma, \sigma \rangle$ and $\pi:\RR^2 \mapsto \RR^2/\mathscr{G}$ for the quotient map. Notice that $G$ is the dihedral group of order $6$, that is the group of symmetries of an equilateral triangle. Since $\mathscr{G}$ is a finite group that acts smoothly and effectively on $\RR^2$, we can apply Theorem \ref{thm:orbi}.

In order to understand the orbifold structure of $\RR^2/\mathscr{G}$, it is convenient to consider the following partition of $\RR^2$:
\begin{itemize}
    \item $A_0=\{(0,0)\}$, \item $A_1=\bigcup_{k=0}^{2} \gamma^k(\{x_1 \neq 0, x_2=0\})$, i.e. the $X$ axis without the origin and its rotations of angles which are multiples of $2\pi/3$, and
    \item $A_2=\RR^2\setminus (A_0\cup A_1)$.
\end{itemize}
The first difference between $A_0,A_1, A_2$ is that they contain points with different stabilizers: indeed, $\mathscr{G}_{(0,0)}=\mathscr{G}$, for every $x \in A_1$, $\mathscr{G}_x\simeq \mathbb{Z}_2$ and for every $x \in A_2$, $\mathscr{G}_x=\{{\rm Id}\}$.

\vspace{2mm}

We will now proceed to construct an atlas of fundamental and linear charts for $\RR^2/\mathscr{G}$, as we outlined above; in the case of this simple example we have several simplifications happening: since $\mathscr{G}$ is finite, the tangent space of every orbit $\mathscr{G}(x)$ is the zero vector; we can identify $\RR^2$ with the tangent space at any point; $\mathscr{G}$ is already a group of linear isometries with respect to the natural metric on $\RR^2$.

\vspace{2mm}

Let us first consider the origin; for every open ball $B_\epsilon((0,0))$ centered at the origin, $(B_\epsilon((0,0)) , \mathscr{G}, \pi)$ is a linear chart for $\pi(B_\epsilon((0,0)))$ and it is fundamental for $\pi((0,0))$. The tangent cone $T_{\pi((0,0))}(\RR^2/\mathscr{G})$ is given by $\RR^2/\mathscr{G}$ again and the tangent space is given by the zero vector alone. $\mathscr{S}_0=\{\pi((0,0))\}$ is the $0$-dimensional manifold appearing in the canonical stratification.

\vspace{2mm}

Given any $x \in \{x_1 \neq 0, x_2=0\} \subset A_1$ (the same conclusions hold for the other points in $A_1$ as well) and choosing the radius $\epsilon$ small enough, one can find an open ball $B_\epsilon(x)$ such that, defining
$$\mu (\mathscr{G},  B_\epsilon(x)):=\{\mu(g,x): \, g \in \mathscr{G}, \, x \in B_\epsilon(x)\},$$
one has $\mu (\mathscr{G},  B_\epsilon(x))\cap B_\epsilon(x)=B_\epsilon(x)$ and 
$$y \in B_\epsilon(x), \, g \in \mathscr{G}, \, \mu(g,y)  \in B_\epsilon(x) \text{ implies } g \in \langle \sigma \rangle .$$
Therefore $(B_\epsilon(x), \langle \sigma \rangle, \pi)$ is a linear orbifold chart for $\pi(B_\epsilon(x))$ and is fundamental for $\pi(x)$. The tangent cone $T_{\pi(x)}(\RR^2/\mathscr{G})$ is given by $\RR^2/\langle \sigma \rangle$ and the tangent space at $\pi(x)$ is one dimensional and is given by ${\rm span}\{(1,0)\}$. $\mathscr{S}_1=\pi(A_1)$ is the $1$-dimensional manifold in the canonical stratification and it is made of two connected components.

\vspace{2mm}

Finally, for every $x \in A_2$, one can choose a radius $\epsilon$ small enough such that $B_\epsilon(x)$ is such that $\mu (g,  B_\epsilon(x)) \cap B_\epsilon(x)=\emptyset$ unless $g={\rm Id}$. Therefore, $(B_\epsilon(x), \{{\rm Id}\}, \pi)$ is a linear orbifold chart for $\pi(B_\epsilon(x))$ and is fundamental for $\pi(x)$. The tangent cone and the tangent space in this case coincide and are given by the whole $\RR^2$. $\mathscr{S}_2=\pi(A_2)$ is the $2$-dimensional manifold of regular points.

\subsection{Identifiable directions and atlas of fundamental charts} \label{subsec:idatl}


We now return to the geometric analysis of the QMC identification problem from section \ref{sec:isometries.manifold}. Based on the orbifold theory detailed earlier we find that Theorems \ref{almostfree} and \ref{thm:orbi} imply the following.
\begin{theorem}[{\bf Orbifold structure of identifiable parameters}] The space of identifiable parameters 
$\orbifirr\simeq \manifirr/\mathscr{G}$ can be endowed with the structure of an orbifold.
\label{th:orbifold.structure}
\end{theorem}
%
Let us make a brief remark on the comparison between QMCs and hidden Markov chains (HMCs).
In general $\orbifirr$ fails to have a manifold structure due to the fact that the stabiliser $\mathscr{G}_V$ is not trivial for isometries $V$ whose quantum channel has non-trivial period. In the case of HMCs, it is not clear whether transition matrices producing the same output distribution at stationarity can be obtained one from another via the action of a certain transformations group; a natural candidate would be the group of permutations acting by relabeling the states of the hidden Markov chain and this has been proved to be the case for a restricted class of HMCs in \cite{Pe69} ). However, even if this was the case, Theorem \ref{thm:clid} shows that in the case of HMCs with deterministic outputs the stabiliser at any point is always trivial.

\vspace{2mm}

In this section we  construct an atlas of fundamental linear charts which will allow us to investigate in detail the geometric structure of $\orbifirr$ and will turn out to be a fundamental tool in section \ref{sec:LAN}. The construction follows the general procedure presented in section \ref{subsec:orbi}, which in our case gains the following statistical meaning: the Riemannian metric used to construct the charts is related to the quantum Fisher information of the statistical model corresponding to the system and output state (Proposition \ref{prop:QFI}); moreover, the subspace of the tangent space at every point $V$ in $\manifirr$ used as coordinates for a neighborhood of $[V]$ in $\orbifirr$, corresponds to infinitesimal changes in the unknown parameter that can be identified from the stationary output state.

\subsubsection*{Identifiable directions} 

Given an element $V \in \manifirr$, its orbit 
$[V]$ under the action of $\mathscr{G}$ is a submanifold of $\manifirr$. We denote by $T^{\rm nonid}_V$ the tangent space of $[V]$ at $V$ (as an element of the submanifold): it represents the infinitesimal changes in $V$ which are not detectable from the output state. Let us consider the restriction of the action
\begin{equation} \label{eq:restract}\mu_V:=\mu(\cdot,V): \mathscr{G} \rightarrow \manifirr, \quad  g \mapsto \mu(g,V),
\end{equation}
then $\tsnonid_V$ is given by the image of $d\mu_V$ at $e:=(1,\mathbf{1}_{\cal H})\in \mathscr{G}$. Note that the tangent space of $\mathscr{G}$ at $e$ can be identified with
\begin{equation} \label{eq:LieAlgebra}
\mathfrak g_V = \{(\theta, K) ~:~ \theta\in \mathbb R, K\in \mathfrak{u}(d), \, {\rm tr}[\rho^{ss}_V K]=0\},
\end{equation}
where we have chosen the natural constraint ${\rm tr}[\rho^{\rm ss}_V K]=0$ to eliminate the $\mathcal U(1)$-direction from the Lie algebra $\mathfrak u(d)$ of $\mathcal U(d)$. Now $d_e\mu_V: \mathfrak g_V \to \TT^{\rm nonid}_V$ is an isomorphism 
given by
\begin{eqnarray}\label{push}
d_e\mu_V:Y =(\theta, K) &\mapsto& 
\left. i\frac{d}{dt} (\exp(tY)V)\right|_{t=0}\nonumber\\
&&=
\left.i\frac{d}{dt}\left(e^{-it\theta}(e^{itK}\otimes \id_{\cal K})Ve^{-itK}\right)\right|_{t=0}\nonumber \\
&&= \theta V-(K\otimes \id_{\cal K})V+VK.
\end{eqnarray}
where $Y\in \mathfrak g_V$ is identified with $(\theta, K)$. The bijectivity follows from the irreducibility of $\TT_V$, but we can also compute the inverse explicitly. In fact, multiplying by $V^*$ from the left we get the relation
\begin{equation}\label{remarkable}
V^*d_e\mu_V(Y) = \theta\mathbf{1}_{\cal H}+ ({\rm Id}-\TT_V)(K)
\end{equation}
which connects the transition operator with the geometry. This is similar to its continuous-time counterpart (eq. (26) in \cite{GK17}), where the role of ${\rm Id}-\TT_V$ is played by the Lindblad generator.

Relation \eqref{remarkable} suggests that we can recover $K$ by inverting ${\rm Id}-\TT_V$. This map is of course not invertible on the full space; however, it is easy to see that \emph{in the irreducible case}, its restriction onto the invariant subspace
$$\mathcal S_V=\{ X \in L^\infty({\cal H}) ~:~ \mathrm{tr}[\rho^{ss}_V X]=0\}$$
has an inverse $({\rm Id}-\TT_V)^{-1}:\mathcal S_V\to\mathcal S_V$. We therefore proceed by computing the stationary mean from a tangent vector, subtracting it and then applying $({\rm Id}-\TT_V)^{-1}$:
\begin{align*}
&\mathsf M_V: \tsnonid_V \to \mathbb C, & \mathsf M_V(A)&={\rm tr}[\rho^{\rm ss}_V V^*A]\\
&\mathsf C_V:\tsnonid_V \to \mathcal S_V & \mathsf C_V(A) &=({\rm Id}-\TT_V)^{-1}(V^*A-\mathsf M_V(A)\id_{\cal H} ).
\end{align*}
This allows us to recover the gauge transformation generators from \eqref{push} via
\begin{align*}
\theta &= M_V\circ d_e\mu_V(Y), & K & = \mathsf C_V\circ d_e\mu_V(Y).
\end{align*}
Moreover, the two mappings can be applied to \emph{any} $A \in T_V(\manifirr)$, thereby setting up a specific way of extracting the ``part'' of $A$ that translates the parameter along an orbit of the gauge group:
\begin{align}\label{connection}
\omega_V: T_V(\manifirr)&\to \mathfrak g_V, & \omega_V(A) &= \left(\mathsf M_V(A), \mathsf C_V(A)\right).
\end{align}
If $A$ lies in the unidentifiable subspace $\tsnonid_V$, we recover $A$ exactly:
$$
A = d_e\mu_V(\omega_V(A)), \quad \text{for all } A\in \tsnonid_V={\rm ran}\, d_e\mu_V.
$$
For a general $A \in T_V(\manifirr)$ we can then extract the part that contributes to an \emph{identifiable} parameter change, by defining a projection (an idempotent linear map) $\mathsf P_V:T_V(\manifirr) \to  T_V(\manifirr)$ and the corresponding range subspace
\begin{equation}\label{projection}
\mathsf P_V := {\rm Id}-d_e\mu_V\circ \omega_V, \quad \tsident_V = {\rm ran}\, \mathsf P_V.
\end{equation}
This provides a splitting of the tangent bundle of $\manifirr$ as a direct sum of vector bundles:
$$
T(\manifirr) = \tsident \oplus \tsnonid.
$$
Notice that from the identification of $\tsnonid_V$ with ${\frak g}_V$ one can easily compute
$$
d_{\rm nonid}:=\dim \tsnonid_V=1 + (d^2-1) = d^2.
$$
Consequently, the dimension of $\tsident_V$ is given by
$$
d_{\rm id}=(2k-1)d^2-d_{\rm nonid}=2d^2(k-1)=2d d_0,$$
where $d_0=d(k-1)$ is the codimension of the range of $V$ inside ${\cal H}\otimes {\cal K}.$

The splitting above is compatible with the action of $\mathscr{G}$ in the following sense. Let us consider the restriction of the action
\[ \mu^g:=\mu(g,\cdot): \manifirr \rightarrow \manifirr, \quad V \mapsto \mu(g,V),
\]
and let $d_V\mu^g: T_V(\manifirr)\to T_{g \cdot V}(\manifirr)$ be the induced tangent map.
\begin{lemma} \label{lem:compa}
    For every $g \in \mathscr{G}$ and every $V \in \manifirr$, the following holds true:
    \begin{equation} \label{eq:groupcomp}
    P_{g \cdot V}d_V\mu^g=d_V\mu^g P_V.
    \end{equation}
Therefore,
$$d_V \mu^g(\tsident_V)=\tsident_{g \cdot V} \qquad 
\text{ and } 
\qquad 
d_V \mu^g(\tsnonid_V)=\tsnonid_{g \cdot V}.$$
\end{lemma}
The proof is just a simple check. The following result collects some equivalent characterisations of identifiable directions, which follow from what we observed so far.

\begin{proposition} \label{horprop} The following are equivalent statements for every $A \in T_V(\manifirr)$:
\begin{itemize}
\item[(i)] $A\in \tsident_V$;
\item[(ii)] $\mathsf{P}_V(A)=A$;
\item[(iii)] $\omega_V(A)=0$;
\item[(iv)] $VV^*A=0$.
\end{itemize}
In particular, we can identify bijectively
\begin{align*}
&L^\infty({\cal H}, R(V)^\perp) \simeq \tsident_V, &A_0\mapsto \begin{pmatrix}0 \\ A_0\end{pmatrix}
\end{align*}
where $L^\infty({\cal H}, R(V)^\perp)$ is considered as a real linear space and the block representation is again induced by the splitting ${\cal H}\otimes {\cal K}=R(V)\oplus R(V)^\perp$.
\end{proposition}
The formulation (iv) is quite transparent in that the number $d_{\rm id}= 2dd_0=2d^2(k-1)$ of identifiable parameters is readily visible. Notice that the choice of the splitting of the tangent bundle is by no means unique. However, this particular choice will turn out to be particularly convenient for computing the singular dimension of a point in $\orbifirr$ and in relation to the \emph{quantum Fisher information} of the system and output state (Proposition \ref{prop:QFI}).


\subsubsection*{Atlas of fundamental charts} 

We now build an atlas for $\orbifirr$ following the path outlined in section \ref{subsec:orbi}, which uses a subset of 
$\tsident_V$ as coordinates around $[V]$; the first step is to find a $\mathscr{G}$-invariant Riemannian metric on $\manifirr$ such that $\tsident_V$ is the orthogonal complement of $\tsnonid_V$ for every $V \in \manifirr$. The proof of next Lemma can be found in appendix \ref{sec:proofsgeom}.

\begin{lemma} \label{lem:rm}
There exists a $\mathscr{G}$-invariant Riemannian metric $\nu$ on $\manifirr$ such that
\begin{enumerate}
    \item for every $V \in \manifirr$, the subspaces $\tsident_V$ and $\tsnonid_V$ of the tangent space are orthogonal.
    
 \item for every $V \in \manifirr$ and $A \in \tsident_V$
 \begin{equation}
 \label{eq:inner.prod.tangent.id}
 \nu_V(A,A)=\tr(\rho_V^{\rm ss}A^*A).
 \end{equation}
\end{enumerate}

\end{lemma}

Note that while the Riemannian metric in Lemma \ref{lem:compa} is not unique, its restriction to vectors in $\tsident_V$ for every $V \in \manifirr$ is completely determined by \eqref{eq:inner.prod.tangent.id}. Now that we have the suitable Riemannian structure, we can construct an atlas of fundamental linear charts: for every $[V] \in \orbifirr$, we pick a representative $V \in [V]$ and we consider the chart given by $(B^{\rm id}_\epsilon(V),\mathscr{G}_V,\pi \circ \exp_{V})$, where
\begin{itemize}
\item $B^{\rm id}_\epsilon(V)$ is the open ball of radius $\epsilon$ centered in $0$ in $\tsident_V$;
\item 
$\mathscr{G}_V$ is the (discrete) 
stabiliser group characterised in Theorem \ref{almostfree}; the differential of its action on $\manifirr$ induces a linear unitary representation of $\mathscr{G}_V$ on $\tsident_V$ given by $g\mapsto  d_V \mu^g$. Since this representation is injective, we use the same notation $\mathscr{G}_V$ for simplicity.

\item $\exp_{V}$ is the exponential map of the Riemannian metric $\nu$ at $V$ and $\pi: \manifirr \rightarrow \orbifirr$ is the projection induced by the equivalence relation. Here $\epsilon$ needs to be small enough in order for the exponential map to be a homeomorphism.
\end{itemize}
We conclude that $\{(B^{\rm id}_\epsilon(V),\mathscr{G}_V,\pi \circ \exp_{V}):[V] \in \orbifirr\}$ is a $2dd_0$-dimensional orbifold atlas, hence $\orbifirr$ is a $2dd_0$-dimensional orbifold.
\subsection{Canonical stratification}
In this section we introduce and study the relevant splitting of $\orbifirr$ into disjoint submanifolds given by the canonical stratification. Regular points correspond to the submanifold of isometries corresponding to primitive quantum channels; the results of this section and section \ref{sec:LAN} show that, if one restricts to such a manifold, all the results in \cite{GK17} can be recovered, in the discrete time setting. However, we will also study the set of singular points and determine the dimensions of the submanifolds contained in it and what features are shared by points on the same singular submanifold.

\vspace{2mm}

The proof of Theorem \ref{almostfree} characterizes the stabilizer at each point. As an abstract group, $\mathscr{G}_V$ is isomorphic to $\mathbb{Z}_{p_V}$; we recall (cf. section \ref{subsec:orbi}) that the class of isomorphism of $\mathscr{G}_V$ as an abstract group does not depend on the choice of a representative for $[V] \in \orbifirr$ 
and we denote it by $\mathscr{G}_{[V]}$ (with an abuse of notation we will use $\mathscr{G}_{[V]}$ to denote also its representative $\mathbb{Z}_{p_{[V]}}$). As we already mentioned, \textit{regular points} are those with a trivial stabilizer, hence the ones that correspond to primitive channels:
$$\manifprim:=\{V \in \mathscr{V}: \, \TT_V \text{ primitive}\}
$$
The set $\manifprim$ is a submanifold of $\manifirr$, which is invariant under the action of $\mathscr{G}$; moreover the proof of Theorem \ref{almostfree} shows that the action of $\mathscr{G}$ on $\manifprim$ is effective and free, hence one obtains the following corollary regarding the geometry of the corresponding quotient space:
$$
\mathscr{P}^{\rm prim}:=\{[V]:V \in \manifprim\}.
$$

\begin{corollary} \label{coro:pb}
We have ${\mathcal P}^{\rm prim}=\manifprim / \mathscr{G}$. This set admits a unique smooth structure such that $\manifprim$ is a principal $\mathscr{G}$-bundle over ${\mathcal P}^{\rm prim}$ and the latter is a submanifold of $\orbifirr$ with the same dimension.
\end{corollary}
Let us now move our attention to the set of singular points, especially regarding its canonical stratification; this set can be further decomposed into a disjoint union of submanifolds according to the \emph{singular dimension} of points. We recall (cf. section \ref{subsec:orbi}) that the singular dimension of a point is the dimension of the tangent space at that point. 
Let us consider a point $[V] \in \orbifirr$ and the fundamental chart $(B^{\rm id}_\epsilon(V), \mathscr{G}_V, \pi \circ \exp_V)$ defined above; the tangent space at $[V]$ is defined as the isomorphism class of the subspace of $\tsident_V$ consisting of vectors that are invariant under the group $\{d_V\mu^g\}_{g \in \mathscr{G}_V}$ and we denote it by $T_{[V]}(\orbifirr)^{\mathscr{G}_{[V]}}$.

The following decomposition into disjoint subsets holds true:
\begin{equation} \label{eq:strata}
\orbifirr =\bigsqcup_{\substack{p=2,\dots,d\\l \in I_p}}\mathscr{P}_{p,l} \sqcup \mathscr{P}^{\rm prim}
\end{equation}
where
\begin{itemize}
\item $\mathscr{P}_{p,l}:=\{[V] \in \orbifirr\setminus \mathscr{P}^{\rm prim}: \,  \mathscr{G}_{[V]}=\mathbb{Z}_p \text{ and $[V]$ has singular dimension $l$}\}$;
\item $\mathscr{P}_{p,l}$ is a $l$-dimensional manifold whose tangent space at a point is canonically identified with $T_{[V]}(\orbifirr)^{\mathscr{G}_{[V]}}$;
\item $I_p$ is the set of $l \in \{0,\dots, d_{\rm id}-1\}$ such that $\mathscr{P}_{p,l} \neq \emptyset$. 
Note that there may exist pairs $(p,l)$ for which $\mathscr{P}_{p,l} = \emptyset$ and it remains an interesting open question to characterize pairs for which $\mathscr{P}{p,l} \neq \emptyset$, given fixed dimensions $d$ and $k$.
\end{itemize}

We remark that the canonical stratification is given by
$$\orbifirr =\bigsqcup_{l=0,\dots, d_{\rm id}-1}\mathscr{P}_{l} \sqcup \mathscr{P}^{\rm prim}, \text{ where } \mathscr{P}_l=\bigsqcup_{\substack{p=0,\dots, d : \\l \in I_p}} \mathscr{P}_{p,l}.$$
The further splitting of $\mathscr{P}_l$ into $\mathscr{P}_{p,l}$ is possible due to the fact that points in the same connected component of $\mathscr{P}_l$ have the same local group and, hence, the same period. Thanks to the atlas of fundamental charts constructed in the previous subsection, we can explicitly compute the singular dimension of every point $[V] \in \orbifirr$. 



\begin{proposition} \label{prop:sd}
    Let $[V] \in \orbifirr$, then $[V] \in \mathscr{P}_{p_{[V]},l_{[V]}}$ where
    $$l_{[V]}=2\left ( 
    \sum_{a=0}^{p_{[V]}-1} 
    \left[d_{a\oplus 1}([V])k-d_a([V])\right]
    d_a([V])
    \right ) \leq 2 d^2(k-1)=d_{\rm id}.$$
    Moreover, for every representative $V \in [V]$ one has
    \begin{equation} \label{eq:singulartangent}
    T_{[V]}(\mathscr{P}_{p_{[V]},l_{[V]}})\simeq \left \{A \in \tsident_V : A P_{a}(V)=(P_{a\oplus 1}(V) \otimes \mathbf{1}_{{\cal K}}) A, \, a=0, \dots, p_{[V]}-1 \right\},\end{equation}
    where $P_a(V)$ are the orthogonal projections onto the subspaces appearing in the periodic decomposition of ${\cal H}$, cf. Theorem \ref{th:periodic.structure}.

    \noindent
    Finally, if two points $[V]$ and $[W]$ belong to the same connected component of any singular manifold $\mathscr{P}_l$ for some $l$, then $p_{[V]}=p_{[W]}$ and $d_a([V])=d_a([W])$ for every $a=0, \dots, p_{[V]}$.
\end{proposition}

The proof of Proposition \ref{prop:sd} can be found in appendix \ref{sec:proofsgeom}. Note that by the definition of periodic projections one has that for every $a=0,\dots, p_V-1$
\begin{equation} \label{eq:periodicprop}
V P_{a}(V)=(P_{a\oplus 1}(V) \otimes \mathbf{1}_{{\cal K}}) V .
\end{equation}
The tangent space at $[V]$ of the submanifold in the canonical stratification can be identified with those vectors in $\tsident_{V}$ that preserve property in Eq. \eqref{eq:periodicprop}.


From the fundamental charts we built, one can see that for every integer $q$ that divides $p_{[V]}$, for every neighborhood of $[V]$, one can find a point $[V^\prime]$ that belongs to such neighborhood and with period $q$. They correspond to those directions in the tangent space 
$\tsident_V$ which are fixed points of the subgroup of $\mathscr{G}_V$ corresponding to the $q$-th roots of $1$, i.e.
$$\left\{(\gamma_{[V]}^{l\cdot p_{[V]}/q},Z_{[V]}^{l \cdot p_{[V]}/q}\right\}, \quad l=0, \dots, q-1.$$
These vectors correspond to those $A:{\cal H} \rightarrow R(V)^\perp$ that put in communication some of the periodic projections:
$$A P_{q,a}(V)=(P_{q,a\oplus 1}(V) \otimes \mathbf{1}_{{\cal K}}) A, \quad a=0,\dots, q-1$$
with
$$P_{q,a}(V)=\sum_{l=0}^{p_{[V]}/q-1}P_{a \oplus lq}(V).$$
Note that
$$d_a([V^\prime])=\sum_{l=0}^{p_{[V]}/q-1}d_{a\oplus lq}([V]).$$
Such a tight mathematical structure reflects the rigidity of the peripheral spectrum of irreducible channels: it can only be the group of the roots of $1$; therefore, any smooth change of the dynamics can only perturb the peripheral eigenvalues bringing some of them inside the unit circle in a way that leaves on the unit circle a subgroup of the original group.

\section{Local approximation of stationary irreducible quantum Markov chains} \label{sec:LAN}

While in the previous section we analysed the \emph{geometric structure} of the  orbifold parameter space, the focus of this section is on the \emph{local properties} of the quantum statistical models corresponding to the system and output state, and the stationary output state.

Let us first introduce the local statistical models that we will analyse in the rest of the section. Consider a Riemannian metric $\nu$ for $\manifirr$ as in Lemma \ref{lem:rm} and a fixed but arbitrary $V_0 \in \manifirr$; there exists a small centered ball of radius $\epsilon>0$,  $B_\epsilon(V_0) \subset T_{V_0}(\manifirr)$, such that $\exp_{V_0}$ is a diffeomorphism between $B_\epsilon(V_0)$ and its image. This provides a convenient local parametrisation of the quantum statistical model of system and output states:
\begin{equation} \label{eq:sosm}
\widehat{\mathbf{Q}}_{V_0}^{\rm s+o}(n)= 
\left \{\left|\Psi_{V(n^{-1/2} A)}(n)\right\rangle ~:~ A \in B_{r_n}(V_0)\right\}, 
\end{equation}
where $V:=\exp_{V_0}$ and $r_n=Cn^{\delta}$ for some $\delta \in (0,1/2)$. Note that for $n$ large enough, one has $n^{-1/2}r_n<\epsilon$ and it makes sense to evaluate $V$ at $n^{-1/2}A$ for $A \in B_{r_n}(V_0)$.

Let us now introduce the sequence of local statistical models corresponding to the stationary output state; in order to work with an identifiable statistical model, we need to consider parameters in $\mathscr{P}^{\rm irr}$, and the orbifold charts built in section \ref{subsec:idatl} provides convenient parametrisations:
\begin{equation} \label{eq:outputsm}
\mathbf{Q}_{V_0}^{\rm out}(n):= \left\{\rho^{\rm out}_{[V(n^{-1/2}A)]}(n) ~:~
 A \in B^{\rm id}_{r_n}(V_0)\right\},
\end{equation}
where $[V(\cdot)]:=\pi \circ V(\cdot)$ and we only consider tangent vectors in $B_{r_n}^{\rm id}(V_0):=B_{r_n}(V_0) \cap \tsident_{V_0}$. A remarkable feature of these family of charts is that they factorise through 
$\manifirr$ by construction; this is fundamental because it provides us with a convenient and smooth choice of representatives for the equivalence classes in 
$[V(B^{\rm id}_{\epsilon}(V_0))] \subseteq \orbifirr$ and this will be essential in order to prove the main results of this section.

\vspace{2mm}

The goal is to show that each sequence of local statistical models introduced above approaches a sequence of ``limit models'' in the sense of Le Cam distance. Concretely, we show that the sequence \eqref{eq:sosm} satisfies quantum local asymptotic normality (Proposition \ref{prop:extsolan}), while in the case of \eqref{eq:outputsm} the limit is a mixture of Gaussian shift models for which we give a complete characterisation (Theorem \ref{thm:nLAN}). We also prove that the limit states corresponding to each $\ket{\Psi_{V(n^{-1/2}A)}(n)}$ only depend on the identifiable component $A^{\rm id}:=P_{V_0}(A)$ (see eq. \eqref{projection} for the definition of $P_{V_0}$) and that the quantum Fisher information of the statistical model of the system and output state at time $n$ evaluated at vectors in $\tsnonid_{V_0}$ grows at most sub-linearly with $n$ (Proposition \ref{prop:QFI}); therefore, if one is interested in asymptotically optimal estimation strategies, it makes sense to consider the following restriction of the model in eq. \eqref{eq:sosm}:
\begin{equation} \label{eq:somid}
\mathbf{Q}_{V_0}^{\rm s+o}(n)= 
\left \{\left|\Psi_{V(n^{-1/2} A)}(n)\right\rangle ~:~ A \in B^{\rm id}_{r_n}(V_0)\right\}.
\end{equation}
Moreover, the limit Gaussian shift model for \eqref{eq:somid} is independent of the initial system state, making it possible to  extend the result in Proposition \ref{prop:extsolan} to initial mixed states as well. Assuming that the system is initially in the unique stationary state, one can check that the model in eq. \eqref{eq:outputsm} can be obtained from the one in eq. \eqref{eq:somid} applying a quantum channel which consists of tracing out the system; we will explicitly construct a quantum channel that does the same for the limit models. Figure \ref{fig:diagram} contains a diagram illustrating this statement.


%
\tikzset{every picture/.style={line width=0.75pt}} 
\begin{figure}
\begin{tikzpicture}[x=0.75pt,y=0.75pt,yscale=-1,xscale=1]
%
\draw    (183,51) -- (385.2,51.4) ;
\draw [shift={(387.2,51.4)}, rotate = 180.11] [color={rgb, 255:red, 0; green, 0; blue, 0 }  ][line width=0.75]    (10.93,-3.29) .. controls (6.95,-1.4) and (3.31,-0.3) .. (0,0) .. controls (3.31,0.3) and (6.95,1.4) .. (10.93,3.29)   ;
\draw    (131.2,71.4) -- (131.2,156.4) ;
\draw [shift={(131.2,158.4)}, rotate = 270] [color={rgb, 255:red, 0; green, 0; blue, 0 }  ][line width=0.75]    (10.93,-3.29) .. controls (6.95,-1.4) and (3.31,-0.3) .. (0,0) .. controls (3.31,0.3) and (6.95,1.4) .. (10.93,3.29)   ;
\draw    (440.2,70.4) -- (440.2,155.4) ;
\draw [shift={(440.2,157.4)}, rotate = 270] [color={rgb, 255:red, 0; green, 0; blue, 0 }  ][line width=0.75]    (10.93,-3.29) .. controls (6.95,-1.4) and (3.31,-0.3) .. (0,0) .. controls (3.31,0.3) and (6.95,1.4) .. (10.93,3.29)   ;
\draw    (183,199) -- (385.2,199.4) ;
\draw [shift={(387.2,199.4)}, rotate = 180.11] [color={rgb, 255:red, 0; green, 0; blue, 0 }  ][line width=0.75]    (10.93,-3.29) .. controls (6.95,-1.4) and (3.31,-0.3) .. (0,0) .. controls (3.31,0.3) and (6.95,1.4) .. (10.93,3.29)   ;
\draw    (178,332) -- (380.2,332.4) ;
\draw [shift={(382.2,332.4)}, rotate = 180.11] [color={rgb, 255:red, 0; green, 0; blue, 0 }  ][line width=0.75]    (10.93,-3.29) .. controls (6.95,-1.4) and (3.31,-0.3) .. (0,0) .. controls (3.31,0.3) and (6.95,1.4) .. (10.93,3.29)   ;
\draw    (131,229) .. controls (156.94,267.61) and (106.23,266.04) .. (133.35,318.4) ;
\draw [shift={(134.2,320)}, rotate = 241.76] [color={rgb, 255:red, 0; green, 0; blue, 0 }  ][line width=0.75]    (10.93,-3.29) .. controls (6.95,-1.4) and (3.31,-0.3) .. (0,0) .. controls (3.31,0.3) and (6.95,1.4) .. (10.93,3.29)   ;
\draw    (434,232) .. controls (398.74,265.49) and (461.27,292.19) .. (438.31,321.65) ;
\draw [shift={(437.2,323)}, rotate = 310.91] [color={rgb, 255:red, 0; green, 0; blue, 0 }  ][line width=0.75]    (10.93,-3.29) .. controls (6.95,-1.4) and (3.31,-0.3) .. (0,0) .. controls (3.31,0.3) and (6.95,1.4) .. (10.93,3.29)   ;

\draw (61,30.4) node [anchor=north west][inner sep=0.75pt]  [font=\LARGE]  {${\cal U}_n(V_0) \subseteq \manifirr$};
\draw (406,30.4) node [anchor=north west][inner sep=0.75pt]  [font=\LARGE]  {$\pi ({\cal U}_n) \subseteq\orbifirr$};
\draw (75,172.4) node [anchor=north west][inner sep=0.75pt]  [font=\LARGE]  {$\left\{\rho ( n)_{V}^{\rm s+o}\right\}$};
\draw (395,172.4) node [anchor=north west][inner sep=0.75pt]  [font=\LARGE]  {$\left\{\rho ( n)_{[V]}^{\rm out}\right\}$};
\draw (120,325) node [anchor=north west][inner sep=0.75pt]  [font=\LARGE]  {$\mathbf{G}_{V_0}$};
\draw (399,325) node [anchor=north west][inner sep=0.75pt]  [font=\LARGE]  {$\mathbf{GM}_{V_0}$};
\draw (268,18.4) node [anchor=north west][inner sep=0.75pt]  [font=\Large]  {$\pi $};
\draw (263,155.4) node [anchor=north west][inner sep=0.75pt]  [font=\Large]  {$\text{tr}_{\mathcal{H}}$};
\draw (272,297.4) node [anchor=north west][inner sep=0.75pt]  [font=\Large]  {$\mathcal{S}_{*}$};
\draw (233,246.4) node [anchor=north west][inner sep=0.75pt]  [font=\Large]  {$n\rightarrow +\infty $};
\end{tikzpicture}
\caption{In the top line of the picture there are the sets of local parameters $\exp_{V_0}(B_{Cn^{\delta-1/2}}(V_0))=: {\cal U}_n(V_0) \subseteq \mathscr{V}^{\rm irr}$ and $\pi \circ \exp_{V_0}(B_{Cn^{\delta-1/2}}(V_0))=:\pi({\cal U}_{n}(V_0))\subseteq \mathscr{P}^{\rm irr}$ and $\pi$ represents the quotient map induced by the equivalence relation introduced in Section \ref{sec:identifiability}. Each parameter space is connected to the statistical model that it parametrizes, i.e. $\mathscr{V}^{\rm irr}$ to $\left\{\rho ( n)_{V}^{\rm s+o}\right\}$ (the state of system and output at time $n$ when the system starts in the stationary state) and $\mathscr{P}^{\rm irr}$ to $\left\{\rho ( n)_{[V]}^{\rm out}\right\}$; one model can be obtained from the other tracing away the system. Finally, $\mathbf{G}_{V_0}$ and $\mathbf{GM}_{V_0}$ are the local limit models, related by the quantum channel $\mathcal{S}_{*}$.}
\label{fig:diagram}
\end{figure}

\subsection{Limit statistical models} \label{sub:limitmodels}

In this section we construct two ``limit'' statistical models which describe the local asymptotic statistics around a fixed given point $V_0\in \mathscr{V}^{\rm irr}$ as shown in our main results, Theorem \ref{thm:SOlan} and Theorem \ref{thm:nLAN}. These are Gaussian shift and mixtures of Gaussian shift models of the type introduced in section \ref{sec:QGSM}, and the construction relies on the general method for building a continuous variables (CV) system outlined in section \ref{cvsystems}.

\vspace{2mm}

Let $V_0\in \mathscr{V}^{\rm irr}$ be a fixed parameter and consider the tangent space of identifiable directions $\tsident_{V_0}$; recall that according to Proposition \ref{horprop}, its vectors can be identified with operators in $L^\infty({\cal H}, R(V_0)^\perp) $ via the map
\begin{align*}
 L^\infty({\cal H}, R(V_0)^\perp) \ni A\mapsto \begin{pmatrix} 0 \\ A\end{pmatrix} \in \tsident_{V_0}.
\end{align*}
The space $\tsident_{V_0}$ has a natural \emph{complex} linear structure and we define a complex inner product  
\begin{equation}
\label{eq:inner.product.id.space}
(A,B)_{V_0} = {\rm tr}\left (\rho^{ss}_{V_0}A^*B\right ), \quad A,B\in L^\infty({\cal H}, R(V_0)^\perp).
\end{equation}
The real and imaginary parts are the following real bilinear forms
\begin{align}\label{complexid}
\beta_{V_0}(A,B)&:={\rm Re}(A,B)_{V_0}, & \sigma_{V_0}(A,B)&:={\rm Im}(A,B)_{V_0}.
\end{align}
such that $\beta_{V_0}$ is strictly positive and $\sigma_{V_0}$ is a nondegenerate symplectic form. Note that $\beta_{V_0}$ coincides with the restriction of any Riemannian metric as in Lemma \ref{lem:rm} to $\tsident_{V_0}$. 

We think of $(\tsident_{V_0}, (\cdot, \cdot)_{V_0}) $ as a space of modes and construct the associated continuous variable system consisting of the Fock space $\mathfrak{H}(\tsident_{V_0})$ and the canonical variables $Z(A)$ acting on this space and satisfying the commutation relations $[Z(A), Z(B)] =2i\mathbf{1} \sigma_{V_0}( A,B)$. By applying the Weyl operators $ W(A) = \exp(-iZ(A))$ to the vacuum we obtain (Gaussian) coherent states $|{\rm Coh}(A)\rangle := W(A) |\Omega\rangle$ with Wigner function  
$$
w_A(B) = \frac{1}{(2\pi)^{d_{\rm id}} } \exp\left(-\frac{1}{2} \|B-A\|_{V_0}^2 \right)
$$
For simplicity, we refer to this construction as the CV system associated to $V_0$ and denote it with the symbol $CV(V_0)$.

\begin{definition}[{\bf Limit gaussian shift model}]
Given parameter $V_0\in \manifirr$, we define the quantum Gaussian shift model consisting of pure coherent states of the CV system $CV(V_0)$ defined above.
\begin{equation} \label{eq:gsm}
\mathbf{G}_{V_0}:=\{ |{\rm Coh}(A)\rangle:= W(A)|\Omega\rangle ~:~ A \in  \tsident_{V_0}\}.
\end{equation}
We further define the sequence of restricted quantum Gaussian shift models
\begin{equation} \label{eq:gsm.restricted}
\mathbf{G}_{V_0}(n):=\{ |{\rm Coh}(A)\rangle:=W(A)|\Omega\rangle ~:~  A \in  B^{\rm id}_{r_n}(V_0)\}.
\end{equation}
where $r_n = C n^{\delta}$, with $0<\delta<1/2$ and $C>0$ fixed constants.
\end{definition}

\vspace{2mm}

We now proceed to construct a second statistical model from the pure Gaussian shift model \eqref{eq:gsm} following the method outlined in section \ref{sec:QGSM}. For this we consider the action of the stabiliser group $\mathscr{G}_{V_0}$ on the tangent space of identifiable parameters 
$\tsident_{V_0}$.
We recall that the stabiliser group is the cyclic group of order $p_{[V_0]}$ whose elements are   
$\{g_k = g_{V_0}^k\}_{k=0}^{p-1}$ with $g_{V_0}=(\overline{\gamma}_{V_0},Z^*_{V_0})$, and whose spectrum consists of  
$p_{[V_0]}$-th roots of the identity $\{\gamma_{V_0}^k\}_{k=0,\dots,p-1}$, as described in Theorem \ref{th:periodic.structure}. Its action on the tangent space is given by the linear map
\begin{equation}
\label{eq:action.stab.group}
U_{V_0}(g)^k : A \mapsto d\mu^{g_k}_{V_0}A =\gamma_{V_0}^k (Z_{V_0}^{*k} \otimes \mathbf{1}_{{\cal K}}) A Z_{V_0}^k, 
\end{equation}
which is a \emph{unitary transformation} with respect to $(\cdot, \cdot)_{V_0}$. Indeed
$$
(U_{V_0}(g) A,U_{V_0}(g) \cdot B)_{V_0}=\tr(\rho^{\rm ss}_{V_0}Z_{V_0}^*A^*(Z_{V_0} \otimes \mathbf{1}_{{\cal K}})(Z_{V_0}^* \otimes \mathbf{1}_{{\cal K}})B Z_{V_0})=\tr(\rho^{\rm ss}_{V_0}A^*B )=(A,B)_{V_0},$$
where we used the fact that $[\rho^{\rm ss}_{V_0},Z_{V_0}]=0$. 
The orthogonal eigenspaces $\mathcal{V}_k$ of the unitary operator $U_{V_0}(g) $ are:
\begin{equation}
\label{eq:def.v_k}
\mathcal{V}_k=\{A \in L^\infty({\cal H}, R(V_0)^\perp): \, AP_a(V_0)= (P_{a \oplus 1-k}(V_0) \otimes \mathbf{1}_{\cal K}) A, \,a=0, \dots, p_{[V_0]}-1\}.
\end{equation}
Note that $\mathcal{V}_0$ is the tangent space at $[V_0]$ of the submanifold 
$\mathscr{S}_{[V_0]}:=\mathscr{S}_{p_{[V_0]},k([V_0])}$ appearing in the canonical stratification (equation \eqref{eq:strata}); the above explicit description of $\mathcal{V}_k$ can be obtained using the same strategy employed for $\mathcal{V}_0$ in the proof of Proposition \ref{prop:sd}. 

\vspace{2mm}

At the level of the CV system $CV(V_0)$, the stabiliser induces the group $G_{V_0}:=\{\alpha_{g_1}^k \, k=0,\dots, p_{[V_0]}-1\}$ of vacuum preserving inner $^*$-automorphisms $\alpha_{g_1}(\cdot)=\Gamma(U_{V_0}(g)) \cdot \Gamma(U_{V_0}(g))^*$ of the algebra of Weyl operators, where $\Gamma(U_{V_0}(g))$ is the second quantization of $U_{V_0}(g)$.

\vspace{2mm}

We can now introduce the second statistical model that will be needed in the following. This will be an instance of the Gaussian mixture models studied in section \ref{sec:QGSM} (see Eq. \eqref{eq:mixture.model}), constructed on the identifiable tangent space.
\begin{definition}[{\bf mixture of Gaussian shifts}]
\label{def:mixture.gaussian}
Given parameter $V_0\in \manifirr$ and $A \in \tsident_{V_0}$ we define the mixed state 
\begin{eqnarray}
\rho(A)&:=&\frac{1}{p_{[V_0]}}\sum_{m=0}^{p_{[V_0]}-1} \alpha_{g_1*}^m(\ket{{\rm Coh}(A)}\bra{{\rm Coh}(A)})\\
\nonumber
&=&\frac{1}{p_{[V_0]}}\sum_{m=0}^{p_{[V_0]}-1} 
\ket{{\rm Coh}(U_{V_0}(g)A)}\bra{{\rm Coh}(U_{V_0}(g)A)}
\end{eqnarray}
and the Gaussian mixture model 
\begin{equation}
\mathbf{GM}_{V_0}:=\left \{ \rho(A) ~:~  A \in \tsident_{V_0}\right \}.
\end{equation}
We further define the sequence of restricted Gaussian mixture model
\begin{equation}
\label{eq:restricted.mixed.Gaussian}
\mathbf{GM}_{V_0}(n):=\left \{ \rho(A) ~:~  A \in B_{r_n}^{\rm id}(V_0)\right \}.
\end{equation}
where $r_n = C n^{\delta}$, with $0<\delta<1/2$ and $C>0$ fixed constants.
\end{definition}
Note that states in ${\bf GM}_{V_0}$ can be obtained from the ones in ${\bf G}_{V_0}$ via the application of the following quantum channel:
\begin{eqnarray*}
\mathcal{S}_*:
L^1({\cal F}(\tsident_{V_0})) 
&\to& 
L^1({\cal F}(\tsident_{V_0}))\\ 
\rho
&\mapsto &
\frac{1}{p_{[V_0]}}\sum_{m=0}^{p_{[V_0]}-1} \alpha_{g_1*}^m(\rho).
\end{eqnarray*}
Let us consider the following orthogonal splitting (with respect to $(\cdot, \cdot)_{V_0}$) \begin{equation}
\tsident_{V_0}=\mathcal{V}_0 \oplus \mathcal{V}_0^\perp
\end{equation}
and let us denote by $A=A_0+A_0^\perp$ the unique splitting of an element $A \in \tsident_{V_0}$ into its components in the two orthogonal subspaces. 
This splitting induces a factorization of the CCR representation into the product of the two CCR representations corresponding to $\mathcal{V}_0$ and $ \mathcal{V}_0^\perp$:
$$W(A)=W(A_0) \otimes W(A_0^\perp), \quad \ket{\Omega}=\ket{\Omega_0}\otimes \ket{\Omega_0^\perp}.$$
By 
Proposition \ref{prop.mixture.model} we can express the gaussian mixture as 
\begin{equation}
\label{eq:decomp.gaussian.mixture}
\rho(A) =
\ket{{\rm Coh}(A_0)}\bra{{\rm Coh}(A_0)}
\otimes \sum_{m=0}^{p-1} \ket{\zeta_m(A)}\bra{\zeta_m{A}},
\end{equation}
 where 
 $\ket{\zeta_m(A)}$ can be expressed as 
 \begin{equation} \label{eq:zed}
 e^{-\beta_{V_0}(A_0^\perp,A_0^\perp)/2} \left ( \delta_0(k) \ket{\Omega_0^\perp} + \sum_{l \geq 1}\frac{1}{\sqrt{l!}}\sum_{m_1\oplus\dots\oplus m_l=m} A_{-m_1} \otimes \cdots \otimes A_{-m_l} \right)   
\end{equation}
where $A_0^\perp=A_1+\dots + A_{p_{[V_0]}-1}$ is the decomposition of $A_0^\perp$ induced by the further splitting of $\mathcal{V}_0^\perp$ into
$\bigoplus_{i=1}^{p_{[V_0]}-1}\mathcal{V}_i $.
While in general the above expression looks rather involved, in the case $p_{[V_0]}=2$ the model is more tractable and we refer to the discussion following Proposition \ref{prop.mixture.model} for more details.


\subsection{Limit results}
\label{sec:limit_results}

In this section we derive asymptotic results concerning the quantum Fisher information (QFI) of the system and output state, and establish the convergence to limit models for system and output and stationary output models. 

\vspace{2mm}

The first result shows that the QFI of system and output state scales linearly with time for the identifiably directions $T^{\rm id}_V$ and sub-linearly for the the non-identifiable directions $T^{\rm nonid}_V$. 
This means that the local asymptotic behavior is determined by the former and the latter can be dropped. 
Recall that any $A \in T_{V_0}$ can be uniquely  decomposed as $A=A^{\rm nonid}+A^{\rm id}$ where 
$A^{\rm nonid}\in T^{\rm nonid}_V$ and 
$A^{\rm id}\in T^{\rm id}_V$ (see equation \eqref{projection}).


\begin{proposition}[{\bf QFI convergence}] \label{prop:QFI}
Let us consider $A,B \in B_\epsilon(V_0)\subset T_{V_0}(\manifirr)$ and let $F_{V_0,n}(A,B)$ be the quantum Fisher information of $|\Psi_{V}(n)\rangle$ at $V_0$ in directions $A$, $B$. Then, one has
$$\lim_{n \rightarrow +\infty} \frac{F_{V_0,n}(A,B)}{n}=F_{V_0}(A^{\rm id},B^{\rm id}),$$
where 
$F_{V_0}(X,Y):=4 \beta_{V_0}(X,Y)=
4{\rm Re }\tr(\rho^{\rm ss}_{V_0} X^*  Y)$ for $X,Y\in T^{\rm id}_{V_0}$.
\end{proposition}
In particular, if either $A$ or $B$ are in $\tsnonid_{V_0}$, then
$$\lim_{n \rightarrow +\infty}\frac{F_{V_0,n}(A,B)}{n}=0.$$
The proof of Proposition \ref{prop:QFI} can be found in Appendix \ref{app:local}. 

\vspace{2mm}

We now consider the local asymptotic theory of the system-output state for states in the neighbourhood of $V_0$. The next result shows that the non-identifiable component of the tangent vector disappears in the limit, in agreement with the fact that its QFI rate per time unit is zero.
Consider the extension of Gaussian shift model $\mathbf{G}_{V_0}(n)$ to all directions in the tangent space:
\begin{equation}\label{eq:redextmodel}\widehat{\mathbf{G}}_{V_0}(n):=
\left\{\ket{{\rm Coh}(A^{\rm id})}: \, A \in T_{V_0}, \, A \in B_{r_n}(V_0)\right\}, \quad 
\end{equation}
We recall that $\nu$ is an arbitrary Riemannian metric on $\manifirr$ as in Lemma \ref{lem:rm}.

\begin{proposition} \label{prop:extsolan}
Given parameter $V_0\in \manifirr$, let us consider the sequence of extended system and output models $\widehat{\mathbf{Q}}^{\rm s+o}_{V_0}(n)$ 
and the extended quantum Gaussian shift sequence $\widehat{\mathbf{G}}_{V_0, }(n)$ defined in \eqref{eq:redextmodel}, with constants $C>0$ and $\delta <(2(d_{\rm id}+3))^{-1}$, where $d_{\rm id}=2d^2(k-1)$.
Then the following limit holds true:
\[
\lim_{n \rightarrow \infty} \Delta(\widehat{\mathbf{Q}}_{V_0}^{\rm s+o}(n),\widehat{\mathbf{G}}_{V_0}(n))=0.
\]

\end{proposition}

The proof of Proposition \ref{prop:extsolan} can be found in Appendix \ref{app:local}. As an immediate consequence, one has the following Theorem (the rates follow from a careful inspection of the proof of the previous Proposition).

\begin{theorem} [{\bf QLAN for system and output state}]\label{thm:SOlan}
Given parameter $V_0\in \manifirr$, let us consider the sequence of system and output models $\mathbf{Q}^{\rm s+o}_{V_0}(n)$ defined in \eqref{eq:sosm} 
and the restricted quantum Gaussian shift sequence $\mathbf{G}_{V_0}(n)$ defined in \eqref{eq:gsm.restricted}, with constants $C>0$ and $\delta <(2(d_{\rm id}+3))^{-1}$, where $d_{\rm id}=2d^2(k-1)$.
Then the following limit holds true:
\[
\lim_{n \rightarrow \infty} \Delta(\mathbf{Q}_{V_0}^{\rm s+o}(n),\mathbf{G}_{V_0}(n))=0.
\]
Moreover, if one picks $\delta < (2(11+2d_{\rm id}))^{-1},$ then
\begin{equation} \label{eq:exrate}\lim_{n \rightarrow +\infty}n^{2\delta}\Delta(\mathbf{Q}_{V_0}^{\rm s+o}(n),\mathbf{G}_{V_0}(n))=0.
\end{equation}

\end{theorem}

Note that since the sequence of statistical models ${\bf G}_{V_0}(n)$ does not depend on the initial state of the system, the same result can be extended to mixed initial system states as well.

\vspace{2mm}

Next, we provide a local approximation of the statistical model of the (reduced) output states and we characterise those sub-models for which LAN holds.
\begin{theorem}[{\bf Limit model for the output state}]\label{thm:nLAN}
Given parameter $V_0\in \manifirr$, let us consider the sequence of output models ${\bf Q}^{\rm out}_{V_0}(n)$ defined in \eqref{eq:outputsm} and the restricted mixed Gaussian shift sequence
${\bf GM}_{V_0}(n)$ defined in \eqref{eq:restricted.mixed.Gaussian}, with constants $C>0$ and $\delta <(2(d_{\rm id}+3))^{-1}$. 
Then the following limit holds true:
\begin{equation}
\label{eq:LeCam.mixed}
\lim_{n \rightarrow \infty} \Delta(\mathbf{Q}_{V_0}^{\rm out}(n),\mathbf{GM}_{V_0}(n))=0.
\end{equation}
\end{theorem}

The proof of Theorem \ref{thm:nLAN} can be found in Appendix \ref{app:local}. Note that the parametrisation of $\mathbf{GM}_{V_0}$ with $\tsident_{V_0}$ is not injective: indeed, Lemma \ref{lem:idgm} shows that the set of states in ${\bf Q}^{\rm out}_{V_0}$ can be parametrized using the tangent cone at $[V_0]$ and this is the maximal parametrization in order to ensure identifiability. Another important remark is that, although LAN does not hold for the full output state model and the Gaussian shift model needs to be replaced by a mixture model, LAN does hold for a a sub-model described locally by the tangent space to the singular manifold containing $[V_0]$ (see \eqref{eq:singulartangent}).
\begin{corollary}[{\bf QLAN for orbifold tangent space}] \label{coro:plan}
Let us consider any singular submanifold $\mathscr{P}_{p,l}$ and, given a parameter $[V_0]\in \mathscr{P}_{p,l}$, let us consider the sequence of output sub-models 
\[
\widetilde{\bf Q}^{\rm out}_{V_0}(n):=
\{
\rho^{\rm out}_{[V(Xn^{-1/2})]}(n)~:~ 
\quad X \in B_{r_n}(V_0)\cap \TT_{[V_0]}(\mathscr{P}_{p,l})
\}
\]
and the restricted quantum Gaussian shift sub-model
\[
\widetilde{\mathbf{G}}_{V_0}(n)=
\{
W(X)\ket{\Omega_0}\bra{\Omega_0}W(X)^* ~:~ \quad X \in B_{r_n}(V_0)\cap \TT_{[V_0]}(\mathscr{P}_{p,l})
\}
\]
If $\delta <(2(l+3))^{-1}$, then for every connected compact set $K\subseteq \mathscr{P}_{p,l}$, the following limit holds true:
\begin{equation} \label{eq:uniflim}
\lim_{n \rightarrow \infty} \sup_{[V_0] \in K}\Delta(\widetilde{\mathbf{Q}}_{V_0}^{\rm out}(n),\widetilde{\mathbf{G}}_{V_0}(n))=0.
\end{equation}
Moreover, if one picks $\delta <  (2(11+2l))^{-1},$ then
\begin{equation} \label{eq:exrate2}\lim_{n \rightarrow +\infty}\sup_{V_0\in K}n^{2\delta}\Delta(\widetilde{\mathbf{Q}}_{V_0}^{\rm out}(n),\widetilde{\mathbf{G}}_{V_0}(n))=0.
\end{equation}
\end{corollary}

The uniformity in the limits in Eq. \eqref{eq:uniflim} and \eqref{eq:exrate2} follows from a careful inspection of the proof of Theorem \ref{thm:nLAN}. It is the formulation of LAN which, together with some preliminary estimator (for instance the ones considered in \cite{GGG23}), allows one to build a two-step estimation procedure that is asymptotically optimal (see Section 9 in \cite{YCH19}). We point out that we cannot strengthen in a similar way the result in Theorem \ref{thm:nLAN}: indeed, a step of the proof relies on the convergence of $T_{V_0}^{p_{[V_0]}l}$ to the ergodic projection $\EE_{V_0}$ (see equation \eqref{eq:erg}) and this is not uniform in $V_0$ if one considers parameters with a different period. A further remark is that if one picks two inindistinguishable isometries $V_0$ and $V_1$, the corresponding models $\widetilde{\mathbf{Q}}^{\rm out}_{V_0}(n)$ and $\widetilde{\mathbf{Q}}^{\rm out}_{V_1}(n)$ (resp. $\widetilde{\mathbf{G}}_{V_0}(n)$ and $\widetilde{\mathbf{G}}_{V_1}(n)$) only change by relabelling of the parameters.

We point out that if one considers regular points, i.e. the manifold $\mathscr{P}^{\rm prim}$ of equivalence classes of primitive isomtries, Corollary \ref{coro:plan} ensures that quantum LAN holds. Due to its relevance, below we restate Corollary \ref{coro:plan} in the case of the set of regular points.

\begin{corollary}[{\bf QLAN for primitive dynamics}] 

Let us consider any $[V_0]\in \mathscr{P}^{\rm prim}$, let us consider the sequence of output models 
\[
{\bf Q}^{\rm out}_{V_0}(n):=
\{
\rho^{\rm out}_{[V(Xn^{-1/2})]}(n)~:~ 
\quad X \in B^{\rm id}_{r_n}(V_0)
\}
\]
and the restricted quantum Gaussian shift sub-model
\[
\mathbf{G}_{V_0}(n)=
\{
W(X)\ket{\Omega_0}\bra{\Omega_0}W(X)^* ~:~ \quad X \in B_{r_n}(V_0)
\}
\]
If $\delta <(2(d_{\rm id}+3))^{-1}$, then for every connected compact set $K\subseteq \mathscr{P}^{\rm prim}$, the following limit holds true:
\begin{equation} \label{eq:uniflim}
\lim_{n \rightarrow \infty} \sup_{[V_0] \in K}\Delta(\mathbf{Q}_{V_0}^{\rm out}(n),\mathbf{G}_{V_0}(n))=0.
\end{equation}
Moreover, if one picks $\delta < (2(11+2d_{\rm id}))^{-1},$ then
\begin{equation} \label{eq:exrate2}\lim_{n \rightarrow +\infty}\sup_{V_0\in K}n^{2\delta}\Delta(\mathbf{Q}_{V_0}^{\rm out}(n),\mathbf{G}_{V_0}(n))=0.
\end{equation}
\end{corollary}


\vspace{2mm}
 
To conclude, we note that the results in this section can be extended to the case of general smooth parametrizations $V:\Theta \rightarrow \manifirr$ where $\Theta \subseteq \mathbb{R}^l$ is an open set and $[V(\theta)] \neq [V(\theta^\prime)]$ for $\theta \neq \theta^\prime$. In this case, for every $\theta_0 \in \Theta$ such that $d_{\theta_0}V$ is injective, one has the same statements as in Theorem \ref{thm:SOlan} and  Theorem \ref{thm:nLAN} substituting $d_{\rm id}$ with $l$ and with the following adapted definitions of statistical models:
\begin{align*}
&\widehat{\mathbf{Q}}_{V_0}^{\rm s+o}(n):=\{|\Psi_{V(\theta_0+hn^{-1/2})}(n) \rangle~:~  h \in B_{Cn^{\delta}}(0)\subseteq \RR^{l}\},\\
&\widehat{\mathbf{G}}_{V_0}(n)=\{\ket{{\rm Coh}((d_{\theta_0}V(h))^{\rm id})}\bra{{\rm Coh}((d_{\theta_0}V(h))^{\rm id})} ~:~ h \in B_{Cn^{\delta}}(0)\subseteq \RR^{l}\},\\
&\widehat{\mathbf{Q}}_{V_0}^{\rm out}(n):=\left \{\rho^{\rm out}_{[V(\theta_0+hn^{-1/2})]}(n) ~:~ h \in B_{Cn^\delta}(0)\subseteq \RR^{l}\right \},\\
&\widehat{\mathbf{GM}}_{V_0}(n)=\left \{\sum_{k=0}^{p-1}\frac{1}{p}\alpha_{g_1*}^k\left (\ket{{\rm Coh}((d_{\theta_0}V(h))^{\rm id}) }\bra{{\rm Coh}((d_{\theta_0}V(h)))^{\rm id}}\right )  ~:~ h \in B_{Cn^\delta}(0)\subseteq \RR^{l}\right \}.
\end{align*}
The fact that the sequence of approximating models $\widehat{\mathbf{GM}}_{V_0}(n)$ depends only on the identifiable component of the tangent vector can be seen by the fact that
$$\rho^{\rm out}_{[V(\theta_0+hn^{-1/2})]}(n)=\rho^{\rm out}_{[e^{-i\left (\sum_{j=1}^{l}a_jh_j\right )n^{-1/2}}e^{i \left (\sum_{j=1}^{l}K_jh_j\right ) n^{-1/2}} \otimes \mathbf{1}_{\cal K}V(\theta_0+hn^{-1/2})e^{-i \left (\sum_{j=1}^{l}a_jh_j\right ) n^{-1/2}}]}(n)$$
for any $\{(a_j,K_j)\}_{j=1}^{l} \subseteq \mathfrak{g}_{V(\theta_0)}$ (see Eq. \eqref{eq:LieAlgebra}) and we can easily choose $(a_j,K_j)$ in a way that cancel the non-identifiable component of the directional derivatives at $\theta_0$.
\section{Analysis of the two-dimensional system and environment unit model}
\label{sec:example}
 In this section we provide a detailed analysis of the simplest possible QMC model: a qubit system interacting with qubit environment units, i.e. ${\cal H}={\cal K}=\cc^2$ and $d=k=2$, cf. section \ref{sec:QMC}. A cartoon illustration of the geometry and the structure of the limit models for three one-parameter families analysed below, is presented in Figure \ref{fig:qubit}. 
 
The manifold $\mathscr{V}^{\rm irr}$ of isometries corresponding to irreducible dynamics is 12 dimensional. The orbifold $\orbifirr$ has dimension $2d^2(k-1)=8$ and it is the union of the manifold of primitive parameters $\mathscr{P}^{\rm prim}$ (which has the same dimension) and the submanifold of isometries with period $2$, which we denoted by $\mathscr{P}_{2,4}$ and has dimension $2\sum_{a=0}^{1}(d_{a\oplus 1}k-d_a)d_a=4$ (it follows using Proposition \ref{prop:sd} and observing that periodic projections in this case can only be one dimensional, i.e. $d_0=d_1=1$). 

\vspace{2mm}
 
 Let us start by describing the periodic submanifold in more detail. Let $\{\ket{0},\ket{1}\}$ denote the standard basis in ${\cal H}$ and ${\cal K}$ and let $\{\ket{00}, \ket{01}, \ket{10}, \ket{11}\}$ be the corresponding product basis in ${\cal H}\otimes {\cal K}$. Let us consider any point $[V_0]$ belonging to $\mathscr{P}_{2,4}$; we can always choose a representative $V_0 \in [V_0]$ such that its matrix representation in the standard basis is 
$$
V_0=\begin{pmatrix} 
0 & x\\
0 & w\\
y &0 \\
z & 0\end{pmatrix}
$$
with $x,y,z,w \in \cc$ satisfying $|x|^2+|w|^2=|y|^2+|z|^2=1$ (normalisation condition) and
$$\det\begin{pmatrix} x & w \\
y & z\end{pmatrix} \neq 0 \quad \text{(irreducibility)}.
$$
The condition for irreducibility has been found using the characterisation in equation \eqref{eq:access}. Indeed, every representative in $[V_0]$ has period $2$, and since the system Hilbert space is two-dimensional, the two periodic projections are one dimensional, hence correspond to an orthonormal basis; thanks to the gauge freedom, we can always find a representative $V_0 \in [V_0]$ for which this basis is the canonical basis and the expression for $V_0$ we presented follows immediately from the condition written in equation \eqref{eq:periodicprop}. 

Equivalently, the Kraus operators of $V_0$ corresponding to the standard basis of ${\cal K}$, and the stabiliser unitary are given respectively by  
$$
K_0 = 
\begin{pmatrix} 
0 & x\\
y &0 
\end{pmatrix}, \qquad 
K_1 = 
\begin{pmatrix} 
0 & w\\
z &0 
\end{pmatrix}, \qquad
Z = \begin{pmatrix} 
1 & 0\\
0 &-1 
\end{pmatrix}.
$$
We can further specify $V_0$ by assuming that $x,y \in \RR$: indeed, if this is not the case, we can consider $c=e^{-i\theta_1}$ and $W \in PU(2)$ acting as follows
$$\ket{0} \mapsto e^{i \theta_2/2 }\ket{0}, \quad \ket{1}\mapsto e^{-i \theta_2/2 }\ket{1}$$
for every $\theta_1, \theta_2 \in \RR$. Then,
$$\overline{c} (W \otimes \mathbf{1}_{\cal K}) V_0 W^*=\begin{pmatrix} 
0 & e^{i(\theta_1 +\theta_2)}x\\
0 & e^{i(\theta_1 +\theta_2)}w\\
e^{i(\theta_1 -\theta_2)}y &0 \\
e^{i(\theta_1 -\theta_2)}z & 0\end{pmatrix}$$
belongs to $[V_0]$ as well and one can always choose $\theta_1$ and $\theta_2$ in order to make $x$ and $y$ real. Therefore, we found the following parametrisation for $\mathscr{P}_{2,4}$:
$$V_0=\begin{pmatrix} 
0 & x:=\sqrt{1-|w|^2}\\
0 & w\\
y:=\sqrt{1-|z|^2} &0 \\
z & 0\end{pmatrix}, \quad |w|, |z| \leq 1, \,\sqrt{1-|w|^2}z \neq \sqrt{1-|z|^2}w,$$
from which we easily see that the (real) dimension of the manifold is $4$.

\vspace{2mm}

We recall that in irreducibile periodic dynamics, the unique stationary state is diagonal with respect to periodic projections and gives the same weight to each one of them, therefore we can immediately conclude that 
$\rho^{\rm ss}_{V_0} = \mathbf{1}/2$. The image of $V_0$ is spanned by the vectors 
$$
V_0|0\rangle = 
|1\rangle \otimes |v_1\rangle =  |1\rangle \otimes 
(y|0\rangle + z|1\rangle), \qquad
V_0|1\rangle = 
|0\rangle \otimes |v_0\rangle =  |0\rangle \otimes 
(x|0\rangle + w|1\rangle), 
$$
which implies that the system-output state for an initial system state $\alpha_0|0\rangle+ \alpha_1|1\rangle$ is 
\begin{eqnarray*}
|\psi^{\rm s+o}(n)\rangle &=& 
\alpha_0 |n ~{\rm mod}~2\rangle \otimes
|v_{n ~{\rm mod}~2} \rangle \otimes \cdots\otimes |v_0\rangle\otimes
|v_1\rangle\\
&+&
\alpha_1 |n+1 ~{\rm mod}~2\rangle \otimes
|v_{n+1 ~{\rm mod}~2} \rangle \otimes \cdots\otimes |v_1\rangle\otimes
|v_0\rangle
\end{eqnarray*}
and the stationary output state is the rank two state
\begin{eqnarray*}
\rho^{\rm out}(n) &=& 
\frac{1}{2} |v_0\rangle\langle v_0| 
\otimes  |v_1\rangle\langle v_1| \otimes \cdots \otimes  |v_{n +1 ~{\rm mod} ~2}\rangle\langle v_{n +1 ~{\rm mod} ~2}|\\
&+& \frac{1}{2}
|v_1\rangle\langle v_1| 
\otimes  |v_0\rangle\langle v_0| \otimes \cdots \otimes  |v_{n ~ {\rm mod} ~2}\rangle\langle v_{n ~ {\rm mod} ~2}|. 
\end{eqnarray*}

We now apply the results of section \ref{subsec:idatl} and in particular Proposition \ref{horprop}
in order to characterise the identifiable tangent space 
$\tsident_{V_0}$. We note that the 
 kernel of $V_0V_0^*$ is spanned by the two orthonormal vectors
 $$ 
 (0,0,|z|,-zy/|z|)^T =: |0\rangle\otimes | v^\perp_0\rangle,\quad {\rm and}\quad  (|w|,-wx/|w|,0,0)^T =: |1\rangle\otimes | v^\perp_1\rangle
 $$ 
 where $\alpha/|\alpha|$ can be interpreted as $1$ if $\alpha=0$. 
 Identifiable vectors 
 $A$ (operators in the kernel of left multiplication by $V_0V_0^*$) are in one-to-one correspondence to $2\times 2$ matrix $A$ with complex entries when expressed in coordinates picking $\{\ket{0}, \ket{1}\}$ and $\{|0\rangle\otimes | v^\perp_0\rangle,|1\rangle\otimes | v^\perp_1\rangle\}$ as basis for ${\cal H}$ and ${\rm Range}(V_0)^\perp$, respectively
 \begin{equation}
 \label{eq:tang.vect.2x2}
 a_{ij} = \langle i\otimes v_i^\perp |A|j\rangle, \quad i,j=0,1.
 \end{equation}
 The inner product on $\tsident_{V_0}$ (see equation \eqref{eq:inner.prod.tangent.id} in Lemma \ref{lem:rm}) is the following:
$$
(A, B)_{V_0}= \tr(\rho^{\rm ss}_{V_0} A^*B)=\frac{1}{2}(\overline{a}_{00}b_{00}+\overline{a}_{10}b_{10}+\overline{a}_{01}b_{01}+\overline{a}_{11}b_{11}),$$
where $A=(a_{ij})$ and $B=(b_{ij})$ for $i,j=0,1$. 
From equation \eqref{eq:def.v_k}, 
we find that $\tsident_{V_0}$ decomposes into the orthogonal eigenspaces of the generator of the stabiliser group
$${\cal V}_0=\left \{ \begin{pmatrix} 0 &  a\\b &0
\end{pmatrix}:a,b \in \cc\right \},$$
while
$${\cal V}_1:=\left \{ \begin{pmatrix} c &  0\\0 &d
\end{pmatrix}:c,d \in \cc\right \}.$$
Indeed, using equation \eqref{eq:action.stab.group} it can be checked that $\mathcal{V}_0$ and $\mathcal{V}_1$ are eigenspaces of the generator of the stabiliser group with eigenvalues $1$ and respectively $-1$. Geometrically, 
$\mathcal{V}_0$ corresponds to parameter 
changes that preserve the periodic character, while 
$\mathcal{V}_1$ is the orthogonal complement within 
$\tsident_{V_0}$.

According to Theorem \ref{thm:nLAN}, this means that the limit output model is a tensor  product of two independent quantum Gaussian shift models whose parameters are the coefficients $a$ and respectively $b$  of basis vectors in $\mathcal{V}_0$, and a Gaussian mixture model whose parameters $c$ and $d$ are the coefficients of the basis vectors in 
$\mathcal{V}_1$. More explicitly, given $A=A_0 + A_1\in \tsident_{V_0}$ with $A_0\in\mathcal{V}_0$ and $A_1\in\mathcal{V}_1$, the limit model is 
$$
\rho(A) := |{\rm Coh}(A_0)\rangle\langle {\rm Coh}(A_0)|\otimes \left (\frac{1}{2}|{\rm Coh}(A_1)\rangle\langle {\rm Coh}(A_1)|+ \frac{1}{2}|{\rm Coh}(-A_1)\rangle\langle {\rm Coh}(-A_1)| \right ).  
$$
We recall that some of the properties of mixed Gaussian models with period $p=2$ were discussed in section \ref{sec:QGSM} following equation \eqref{eq:2pexample}. 

\vspace{2mm}

A complete treatment of the \emph{full} estimation problem for a completely unknown dynamical parameter $V$, goes beyond the scope of the paper which focuses on the local asymptotic structure of the quantum output model. We anticipate that the estimation problem can be tackled in a two stage procedure, whereby in the first stage the parameter is localised using a ``non-optimal'' measurement, and then the asymptotic theory (Theorem \ref{thm:nLAN} and Corollary \ref{coro:plan}) is used to design an asymptotically optimal measurement in the second stage. Such methods have been demonstrated in the i.i.d. setup in \cite{YCH19,LAN6}. Nevertheless, in the rest of this section we make the first steps in this direction by analysing in more detail three representative \emph{one-parameter} sub-models $V_\theta$. In each case, we consider the problem of localising the parameter by a preliminary estimation procedure. The three models are:

\begin{enumerate}
    \item a model sitting inside the space of periodic dynamics $\mathscr{P}_{2,4}$: in this case with high probability we can locate the parameter in a preliminary step with precision $n^{-1/2+\delta}$ for some small $\delta$; inside this local region the model is approximated by a quantum Gaussian shift. 

\vspace{2mm}
    
    \item a model which is contained in the primitive submanifold $\mathscr{P}^{\rm prim}$, except from one point which corresponds to a periodic QMC; the tangent vector to the statistical model in that point does \emph{not} belong to the eigenspace ${\cal V}_1$ of the corresponding tangent space, rather it has non-zero components in both ${\cal V}_0$ and ${\cal V}_1$. In this case we can again localise the parameter with rate $n^{-1/2+\delta}$ for some small $\delta$; the local approximation however, becomes a tensor product between a quantum Gaussian shift and a mixture of shifted coherent states. 

    \vspace{2mm}
    
\item  a model passing through a single periodic point, whose tangent vector in that point \emph{does} belong to ${\cal V}_1$. In this case,  the model converges locally to a mixture of two one-mode Gaussian states with opposite amplitudes (see Theorem \ref{thm:nLAN});
for this model, it remains an open question whether we can localise the parameter with rate $n^{-1/2+\epsilon}$ by performing measurements on the output. We present two measurement strategies with qualitatively different 
behaviour and speculate that one of them may be a candidate for the desired parameter localisation. 
\end{enumerate}

\bigskip \textit{First model: inside the periodic submanifold.} In this example we consider a family of isometries which never leaves the periodic submanifold. 
Let $V_\theta$ be the following family of isometries:
$$V_\theta=\begin{pmatrix} 0 & \sqrt{1-4\theta^2}\\
0& 2\theta\\
\theta & 0\\
i\sqrt{1-\theta^2} & 0
\end{pmatrix}$$
for $\theta \in (1/4,1/2)$; the point $1/2$ is removed in order to have a smooth model, while the choice of $1/4$ is arbitrary but insures that the parameter $\theta= 0$ is not included in the model.

The corresponding Kraus operators are given by
$$
K_{\theta,0} = 
\begin{pmatrix} 
0 & \sqrt{1-4\theta^2}\\
\theta &0 
\end{pmatrix}, \qquad 
K_{\theta, 1} = 
\begin{pmatrix} 
0 & 2\theta\\
i \sqrt{1-\theta^2} &0 
\end{pmatrix}.
$$
The corresponding quantum channel $T_{\theta}$ is irreducible and periodic for every value of $\theta$ in the interval $(1/4,1/2)$.

Let us consider the simple estimator  obtained by measuring the observable $|0\rangle\langle 0|$ on each output unit and taking the empirical frequency of the  outcomes
$$
\bar{X}_n =\frac{1}{n}\sum_{i=1}^n X_i,
$$ where $X_i\in\{0,1\}$. In the stationary regime, the mean of $X_i$ is given by 
$$
\mathbb{E}_\theta(X_i) ={\rm Tr}(\rho^{\rm ss}_{V_\theta} K_{\theta,0}^*K_{\theta,0})= \frac{1}{2}-\frac{3}{2}\theta^2,
$$ hence all parameters in $(1/4,1/2)$ are identifiable. We define the moment-like estimator
$$
\hat{\theta}_n:=  \sqrt{\max\left \{\frac{1}{3}-\frac{2}{3}\overline{X}_n,0\right \}}
$$ 
where maximum is taken in order to deal with possible negative values. 
We use $\hat{\theta}_n$ on the first $n^\prime=n^{1-\epsilon}$ output units with $1>2\delta>\epsilon>0$ in order to locate the parameter with high probability in a neighborhood of size $n^{-1/2+\delta}$: indeed, for every $\alpha \in (13/32,1/2)$ we have that
\[\begin{split}&\mathbb{P}_\theta\left ( \left |  \hat{\theta}_n-\theta \right| >n^{-1/2+\delta} \right )=\\
&\mathbb{P}_\theta\left ( \left |  \hat{\theta}_n-\theta \right| >n^{-1/2+\delta},\,\overline{X}_n \leq \alpha \right )+
\mathbb{P}_\theta\left ( \left |  \hat{\theta}_n-\theta \right| >n^{-1/2+\delta},\,\overline{X}_n > \alpha \right ) \leq \\
& \mathbb{P}_\theta\left ( C(\alpha)\left | \overline{X}_n-\frac{1}{2}-\frac{3\theta^2}{2} \right| >n^{-1/2+\delta},\,\overline{X}_n \leq \alpha \right )+\mathbb{P}_\theta\left ( \overline{X}_n > \alpha \right ) \leq\\
&\mathbb{P}_\theta\left ( C(\alpha)\left | \overline{X}_n-\mathbb{E}_\theta[\overline{X}_n] \right| >n^{-1/2+\delta} \right )+\mathbb{P}_\theta\left ( \overline{X}_n-\mathbb{E}_\theta[\overline{X}_n] > \alpha -\mathbb{E}_\theta[\overline{X}_n]\right )
\end{split}\]
for some constant $C(\alpha)>0$, where we used the fact that the square root is Lipschitz on compact intervals separated from $0$ and that, if $\overline{X}_n \leq \alpha$, then
$$\frac{1}{3}-\frac{2}{3}\overline{X}_n \geq \frac{1}{3}-\frac{2}{3}\alpha>0 \text{ and } \hat{\theta}_n=\sqrt{\frac{1}{3}-\frac{2}{3}\overline{X}_n}.$$ Moreover,
$\alpha -\mathbb{E}_\theta[\overline{X}_n] \geq \alpha-\sup_{\theta \in (1/4,1/2)}\mathbb{E}_\theta[\overline{X}_n]=\alpha -13/32$; therefore, for $n$ big enough, one can upper bound both terms using, for instance, the concentration inequality proved in Theorem 4 in \cite{GGG23}, obtaining the following:
$$\mathbb{P}_\theta\left ( \left |  \hat{\theta}_n-\theta \right| >n^{-1/2+\delta} \right ) =O(\exp(-n^{2(\delta-\epsilon)}/2)).$$

We can now apply Corollary \ref{coro:plan} in order to get a local approximation of the statistical model with a quantum Gaussian shift; 
in order to completely specify the limiting local approximation 
we need to find the identifiable component of the tangent vector to our statistical model. As a consistency check we will verify that the identifiable projection of the tangent vector is always in the tangent space to the singular submanifold corresponding to periodic dynamics. The tangent vector at $\theta_0$ is given by
$$ 
A_{\theta_0} =\left.\frac{dV_\theta}{d\theta}\right|_{\theta_0}=\begin{pmatrix} 0 & \frac{-4\theta_0}{\sqrt{1-4\theta_0^2}}\\
0& 2\\
1 & 0\\
-\frac{i\theta_0}{\sqrt{1-\theta_0^2}} & 0
\end{pmatrix}.$$
Its identifiable component is computed by using the explicit expression of $\mathsf P_{V_\theta}$ as given in Section \ref{subsec:idatl}):
$$A_{\theta_0}^{\rm id}= 
\mathsf P_{V_{\theta_0}} (A_{\theta_0})=
\begin{pmatrix} 0 & \frac{-\theta_0(7+4\theta_0)}{\sqrt{1-4\theta_0^2}}\\
0& 2\\
1-2\theta_0^2 & 0\\
-\frac{i\theta_0(3-\theta_0^2)}{\sqrt{1-\theta_0^2}} & 0
\end{pmatrix}.
$$
\sloppy This implies that $A_{\theta_0}^{\rm id}\ket{i}\bra{i}=(\ket{i\oplus 1}\bra{i \oplus 1} \otimes \mathbf{1}_{\cal K})A_{\theta_0}^{\rm id}$, which is exactly the condition expressed in Proposition \ref{prop:sd}, so $A_{\theta_0}^{\rm id}\in \mathcal{V}_0$. Summing up, the statistical model approximating $\{\rho^{\rm out}_{\theta_0+u/\sqrt{n}}(n)\}$ in a local neighbourhood of $\theta_0$ is the one dimensional quantum Gaussian shift on the one-mode CV 
system with Fock space 
$\mathfrak{H}(\mathbb{C} A_{\theta_0}^{\rm id})$
$$
\left\{ \ket{{\rm Coh}(u A_{\theta_0}^{\rm id})}\bra{{\rm Coh}(u A_{\theta_0}^{\rm id})}:|u| \leq n^{\delta}  \right\}.
$$
whose QFI at $\theta_0$ is, cf. Proposition \ref{prop:QFI}
$$
F_{\theta_0}= 4 \|A_{\theta_0}^{\rm id}\|_{\theta_0}^2 = 4 {\rm Tr}(\rho^{\rm ss}_{\theta_0} A_{\theta_0}^{\rm id *}A_{\theta_0}^{\rm id })  = 4\left[
\frac{\theta^2_0(7+4\theta_0)^2}{1-4\theta_0^2}+ (1-2\theta_0^2)^2 +  \frac{\theta^2_0(3-\theta_0^2)^2}{1-\theta_0^2} +4\right].
$$

\vspace{2mm}

\bigskip \textit{Second model: `good' crossing of the periodic submanifold.} Let us now consider the following family of isometries:
$$V_\theta=\begin{pmatrix} \theta & \sqrt{1-3\theta^2}\\
i\theta& -\theta\\
-\theta & i\theta\\
\sqrt{1-3\theta^2} & -\theta,
\end{pmatrix}$$
with Kraus operators given by
$$ 
K_0(\theta) = 
\begin{pmatrix} 
 \theta & \sqrt{1-3\theta^2}\\
-\theta & i\theta
\end{pmatrix}, \qquad 
K_1(\theta) = 
\begin{pmatrix} 
i\theta& -\theta\\
\sqrt{1-3\theta^2} & -\theta
\end{pmatrix}.
$$
The dynamics is periodic and irreducible for $\theta=0$ and we will focus on an interval $(-\overline{\theta}, \overline{\theta})$ for $\overline{\theta}$ such that the parameters are identifiable and the dynamics is primitive for $\theta \in (-\overline{\theta}, \overline{\theta})$, $\theta \neq 0$. Let us verify that these requirements can be satisfied. Since the dynamics is irreducible at $\theta =0$, this property holds in an interval around $0$. Below we exhibit an estimator that can identify parameters around $\theta=0$. 

Finally, we show that the projection of the tangent vector at $\theta=0$ does not belong to ${\cal V}_0$, hence the statistical curve exits the periodic submanifold at least for an interval of parameters around $0$.

Firstly, one can verify by direct computation that the stationary state of $V_\theta$ is given by 
$\rho^{\rm ss}= \mathbf{1}/2$. Let $\overline{X}_n$
be the estimator obtained measuring each output unit in the basis $\ket{\pm}=\frac{1}{\sqrt{2}}(\ket{0}\pm\ket{1})$ and taking the empirical frequency of the 
``$+$'' outcome; its stationary mean is given by
$$
\mathbb{E}_\theta (X_i) =m_\theta=\tr(\rho^{\rm ss}_\theta K_+(\theta)^*K_+(\theta))=
\frac{1}{2}(1-2\theta \sqrt{1-3\theta^2}),  $$
$K_+(\theta):=\frac{1}{\sqrt{2}}(K_0(\theta)+K_1(\theta))$.
Since the first derivative of $m_\theta$ at $0$ is different from $0$, the map $\theta \mapsto m_\theta$ is invertible in a neighbourhood of $\theta=0$; moreover, $m_\theta$ and its inverse are Lipschitz around $\theta=0$, so we can use $\overline{X}_n$ in order to find a confidence interval of length of the order $n^{-1/2+\delta}$ for $\delta$ small enough. Finally, let us find the identifiable component of the tangent vector around $\theta=0$. 
By direct computation and using the explicit expression of $\mathsf P_{V_\theta}$ as given in Section \ref{subsec:idatl} we obtain
$$
A_0= \left.\frac{dV_\theta}{d\theta}\right|_{\theta=0}
=
\begin{pmatrix} 1 & 0\\
i& -1\\
-1 & i\\
0 & -1\end{pmatrix},
\quad 
\text{ and }
\quad 
A_0^{\rm id}=\mathsf P_{V_{0}} (A_{0})=\begin{pmatrix} 0 & 0\\
-1+i& -1\\
-1 & 1+i\\
0 & 0\end{pmatrix}.
$$
To verify that $A_0^{\rm id}$ does not belong to $\mathcal{V}_0$, nor $\mathcal{V}_1$ we look at the action of the stabiliser on $A_0^{\rm id}$, cf. \eqref{eq:action.stab.group}:
\begin{equation}\label{eq:tangentex2}
U_{V_0}(g)A^{\rm id}_0= -\begin{pmatrix} 1 &0&0&0\\
0 & 1&0 &0 \\
0 &0 & -1& 0\\
0 & 0& 0 & -1\end{pmatrix}A_0^{\rm id}\begin{pmatrix}1 &0\\
0 & -1 \end{pmatrix}= \begin{pmatrix} 0 & 0\\
1-i& -1\\
-1 & -(1+i)\\
0 & 0\end{pmatrix} 
\neq \pm A_0^{\rm id}.\end{equation}
From \eqref{eq:tangentex2} we can actually see how the vector splits into the sum of two eigenvectors of the stabiliser 
(which is necessary to determine the limit local approximation):
$$A_0^{\rm id}=
\begin{pmatrix} 0 & 0\\
0& -1\\
-1 & 0\\
0 & 0\end{pmatrix}
+
\begin{pmatrix} 0 & 0\\
-1+i& 0\\
0 & 1+i\\
0 & 0\end{pmatrix}
.
$$
In terms of the explicit identification with $2\times 2$ matrices defined in equation \eqref{eq:tang.vect.2x2} the two components are 
$$
B_0= 
\begin{pmatrix} 
0& -1\\
-1 & 0
\end{pmatrix}
\in { \cal V}_0, \qquad 
B_1= 
\begin{pmatrix} 
-1+i& 0\\
0 & 1+i
\end{pmatrix}
\in { \cal V}_1
$$

This means that in the neighbourhood of 
$\theta =0$ the model $\rho^{\rm out}_{u/\sqrt{n}}(n)$ can be approximated by a product of a Gaussian shift (corresponding to $B_0$) and a mixed Gaussian model (corresponding to $B_1$)
$$
\rho(u):=\ket{{\rm Coh}(uB_0)}\bra{{\rm Coh}(uB_0)} \otimes \frac{1}{2}\left (\ket{{\rm Coh}(uB_1)}\bra{{\rm Coh}(uB_1)} +\ket{{\rm Coh}(-uB_1)}\bra{{\rm Coh}(-uB_1)}\right ). $$
In particular, the parameter is identifiable thanks to the Gaussian shift component of the model.

\bigskip \textit{Third model: `bad' crossing of the periodic submanifold.} Let $V_\theta$ be the family of isometries 
$$V_\theta=\begin{pmatrix} \sqrt{\frac{2}{3}}\sin(\theta) & \sqrt{\frac{1}{3}}\cos(\theta)\\
\sqrt{\frac{1}{3}}\sin(\theta) & -\sqrt{\frac{2}{3}}\cos(\theta)\\
\sqrt{\frac{1}{2}}\cos(\theta) & \sqrt{\frac{1}{2}}\sin(\theta)\\
-\sqrt{\frac{1}{2}}\cos(\theta) & \sqrt{\frac{1}{2}}\sin(\theta)
\end{pmatrix}$$
for $\theta \in 
[-\pi/2,3\pi/2)$. The corresponding Kraus operators are given by
$$ 
K_0(\theta) = 
\begin{pmatrix} 
\sqrt{\frac{2}{3}}\sin(\theta) & \sqrt{\frac{1}{3}}\cos(\theta)\\
\sqrt{\frac{1}{2}}\cos(\theta) & \sqrt{\frac{1}{2}}\sin(\theta)
\end{pmatrix}, \qquad 
K_1(\theta) = 
\begin{pmatrix} 
\sqrt{\frac{1}{3}}\sin(\theta) & -\sqrt{\frac{2}{3}}\cos(\theta)\\
-\sqrt{\frac{1}{2}}\cos(\theta) & \sqrt{\frac{1}{2}}\sin(\theta)
\end{pmatrix}.
$$
Let us first restrict $\theta$ to a set of identifiable values: note that given any $\theta \in 
[-\pi/2,3\pi/2)$ one has
$$
K_0(\pi-\theta) = ZK_0(\theta)Z^*,
\qquad
K_1(\pi -\theta) =ZK_1(\theta)Z^*
$$
Therefore $\theta$ and $\pi-\theta$ have the same output state and we can restrict to $\theta \in [-\pi/2,\pi/2].$ Moreover, since
$$
K_0(-\theta) = -ZK_0(\theta)Z^*,
\qquad
K_1(-\theta) = -ZK_1(\theta)Z^*
$$
we can further restrict to $\theta \in [0,\pi/2].$ Finally, we need to remove $\theta=\pi/2$, because the corresponding isometry is not irreducible; for $\theta \in (0,\pi/2)$ the corresponding isometry is primitive, while for $\theta=0$ it is irreducible with period $2$. Note that for every value of $\theta$ in the considered interval one has that 
$\rho_\theta^{\rm ss}=\mathbf{1}/2$. In order to simplify the analysis, we reduce to consider parameters in an interval $[0, \overline{\theta})$ for $\overline{\theta}$ small enough.

We will consider two different output measurement strategies, and argue that while the first one cannot localise the parameter at the desired rate, the second appears to be more likely to have this property. 

The first measurement strategy is similar to that of the previous examples and consists of measuring the observable 
$\ket{0}\bra{0}$ and taking the average $\bar{X}_n$ of the outcomes. 
Adopting the notation used in section \ref{sec:identifiability} we find that the average value is 
$$m_\theta(\ket{0}\bra{0}):=\tr(\rho_{\theta}^{\rm ss}K_0(\theta)^*K_0(\theta))=\frac{7}{12}-\frac{1}{6}\cos(\theta)^2,
$$
Therefore, all the values of $\theta \in [0,\pi/2)$ are identifiable. 
For $\theta\neq 0$ the empirical average $\bar{X}_n$ satisfies Central Limit Theorem and one can use concentration inequalities \cite{GGG23} to show that $\theta$ can be estimated at rate $n^{-1/2}$. However, approaching $\theta=0$, the performance of $\overline{X}_n$ deteriorates: due to the quadratic dependence on $\theta$, the sensibility of $m_\theta(\ket{0}\bra{0})$ decreases when approaching $0$, as it can be seen from the fact that its first derivative tends to $0$. 
In the same time, as one approaches $\theta=0$, the variance converges to a non-zero constant, which means that the signal to noise ratio vanishes.

\vspace{2mm}

We now consider an alternative output measurement and show that, in contrast to the previous case, the signal to noise ratio (SNR) of the proposed estimator has standard scaling with a constant that isseparated from $0$ uniformly for $\theta$ in a neighbourhood of zero. Therefore, during the argument we will feel free to make adjustments to the interval that guarantee the SNR lower bound.

The measurement consists of repeatedly measuring a positive observable localised on \emph{two}
output units and whose kernel includes the support of the reduced output state at 
$\theta=0$. Let $|\omega\rangle \in \mathbb{C}^2\otimes \mathbb{C}^2$ be a state such that 
\begin{equation}\label{eq:ortho}\langle \omega| v_0\otimes v_1\rangle = 
\langle \omega| v_1\otimes v_0\rangle =0\end{equation}
and let $P_\omega:= |\omega\rangle \langle \omega|$ and let 
$Y_i\in \{0,1\}$ denote the outcomes of measuring $P_\omega$ on non-overlapping blocks of 2 output units. Then the average count
$$
\bar{Y}_n :=   \frac{2}{n}\sum_{j=1}^{n/2} Y_i
$$
has expectation
$$
\mathbb{E}_\theta (\bar{Y}_n)= m_\theta (P_\omega)  = {\rm Tr} (\rho^{\rm ss}_\theta \tilde{K}_\omega(\theta)^*\tilde{K}_\omega(\theta)),
$$
where $\tilde{K}_\omega(\theta)= 
\langle \omega | V^{(2)}_\theta V^{(1)}_\theta $ is a ``two-steps" Kraus operator, with superscript indicating the output unit on which the isometry acts. Notice that equation \eqref{eq:ortho} implies that $\tilde{K}_\omega(0)=0$ and  $m_0(P_\omega)=m_0^\prime(P_\omega)=0$; in addition, for a suitable choice of $\ket{\omega}$, one has that $m^{\prime\prime}_0(P_\omega) \neq 0$. This condition is satisfied for instance if we choose 
$|\omega\rangle = (\sqrt{2}|00\rangle - |11\rangle)/\sqrt{3}$ but the analysis below is independent of the choice of such $|\omega\rangle$.
Finally, since $m^{\prime\prime}_0(P_\omega) \neq 0$, then $m_\theta(P_\omega)$ is injective on $[0,\overline{\theta})$ for $\overline{\theta}$ small enough.

\vspace{2mm}

Let us now analyse ${\rm var}_{\theta}(\overline{Y}_n)$: since $Y_i\in \{0,1\}$, the variance 
\[\begin{split}
{\rm Var}_\theta(\bar{Y}_n) &=\mathbb{E}_\theta\left[(\bar{Y}_n - m_\theta(P_\omega))^2\right] = 
\frac{2}{n}m_\theta(P_\omega)\\
&+  
\frac{8}{n^2}\sum_{1\leq j<k\leq n/2} 
{\rm Tr} \left(
\rho^{\rm ss}_\theta K_\omega(\theta)^* \mathcal{T}_\theta^{2(k-j-1)}(N_\theta)K_\omega(\theta)
\right) 
- m_\theta(P_\omega)^2,
\end{split}\]
where $N_\theta= \tilde{K}_\omega(\theta)^*\tilde{K}_\omega(\theta)$. In order to better examine the dependence of ${\rm var}_{\theta}(\overline{Y}_n)$ on $\theta$ and $n$, it is convenient to rewrite the third term in the previous equation; introducing the notation $\widetilde{\TT}_\theta(\cdot)=\TT^2_\theta(\cdot)-\tr(\rho_\theta^{\rm ss} \cdot)\mathbf{1}$, one has
\[\begin{split}
&\frac{8}{n^2}\sum_{1\leq j<k\leq n/2} 
{\rm Tr} (\rho^{\rm ss}_\theta K_\omega(\theta)^* \mathcal{T}_\theta^{2(k-j-1)}(N_\theta)K_\omega(\theta)) 
- m_\theta(P_\omega)^2\\
&=\frac{2}{n}\left (\frac{n}{2}-1\right )m_\theta(P_\omega)^2-m_\theta(P_\omega)^2\\
&+\frac{8}{n^2}\sum_{1\leq j<k\leq n/2} 
{\rm Tr} (\rho^{\rm ss}_\theta K_\omega(\theta)^* \widetilde{\mathcal{T}}_\theta^{k-j-1}(N_\theta)K_\omega(\theta))\\
&=-\frac{2}{n}m_\theta(P_\omega)^2+\frac{8}{n^2}\sum_{1\leq j<k\leq n/2} 
{\rm Tr} (\rho^{\rm ss}_\theta K_\omega(\theta)^* \widetilde{\mathcal{T}}_\theta^{k-j-1}(N_\theta)K_\omega(\theta)).
\end{split}\]
The following lemma shows that we can use the spectral radius of $\widetilde{\TT}_\theta$ in order to upper bound its norm and provides the exact value of the spectral radius. By $\| \, \|_2$ and $\| \, \|_{2 \rightarrow 2}$ we mean the Hilbert-Schmidt norm and the induced operator norm.
\begin{lemma}
\label{lemma.SNR}
    If $\overline{\theta}$ is small enough, then
    \begin{enumerate}
        \item the spectral radius of $\widetilde{\TT}_\theta$ is given by $(1-2\sin^2(\theta))^2$;
        \item there exists a constant $C>0$ such that
        $$\sup_{\theta \in (0,\overline{\theta})}\|\widetilde{\TT}_\theta^k\|_{2 \rightarrow 2} \leq C(1-2\sin^2(\theta))^2.$$
    \end{enumerate}
\end{lemma}
Therefore, if $\theta \in (0,\overline{\theta})$ for $\overline{\theta} $ small enough one has
\[\begin{split}
   & \left | \frac{8}{n^2}\sum_{1\leq j<k\leq n/2} 
{\rm Tr} (\rho^{\rm ss}_\theta K_\omega(\theta)^* \widetilde{\mathcal{T}}_\theta^{k-j-1}(N_\theta)K_\omega(\theta))\right |\\
&  \leq \frac{8}{n^2}\sum_{1\leq j<k\leq n/2} 
\|K_\omega(\theta)\rho^{\rm ss}_\theta K_\omega(\theta)^*\|_2 \|\widetilde{\mathcal{T}}_\theta^{k-j-1}\|_{2 \rightarrow 2}\|N_\theta\|_2\\
&\leq \frac{8C}{n^2}\sum_{1\leq j<k\leq n/2} (1-2\sin^2(\theta))^{2(k-j-1)}\theta^4 \\
&=\frac{8C\theta^4}{n^2} \left ( \left (\frac{n}{2}-1\right ) \frac{1}{1-(1-2\sin^2(\theta))^2}-\frac{(1-2\sin^2(\theta))^2-(1-2\sin^2(\theta))^{n}}{(1-(1-2\sin^2(\theta))^2)^2}\right )\\
&\leq \frac{C^\prime \theta^2}{n}
\end{split}\]
for some positive constant $C^\prime$. Therefore,
$${\rm var}_\theta(\overline{Y}_n)\leq \frac{2}{n}m_\theta(P_\omega)(1-m_\theta(P_\omega))+ \frac{C^{\prime} \theta^2}{n}.$$

This implies that for every $\overline{\theta}$ small enough one has the following uniform lower bound for the 
SNR of $\bar{Y}_n$
$$\inf_{\theta \in (0,\overline{\theta})}\frac{\left(\frac{d\mathbb{E}_\theta[\overline{Y}_n]}{d\theta}\right)^2}{{\rm var}_\theta(\overline{Y}_n)} \geq \widetilde{C}n $$
for some positive constant $\widetilde{C}$.
In conclusion, $\theta$ can be estimated at rate $n^{-1}$ with a constant that does not vanish in a neighbourhood of zero. This provides a compelling argument that a two step procedure (localise and apply asymptotic approximation) can be applied similarly to the previous models. 

For completeness, we now identify the tangent vector and the limit model for our third parametric family. By direct computation we find
$$
C_0= \left.\frac{dV_\theta}{d\theta}\right|_{\theta=0}
=
\begin{pmatrix} \sqrt{\frac{2}{3}} & 0\\
\sqrt{\frac{1}{3}}& 0\\
0 & \sqrt{\frac{1}{2}}\\
0 & \sqrt{\frac{1}{2}}\end{pmatrix},
\quad 
\text{ and }
\quad 
C_0^{\rm id}=\mathsf P_{V_{0}} (C_{0})= C_0
$$
and we can check $C_0\in \mathcal{V}_1$ similarly to equation \eqref{eq:tangentex2}. The limit model is the Gaussian mixture 
$$
\frac{1}{2}\left (\ket{{\rm Coh}(uC_1)}\bra{{\rm Coh}(uC_1)} +\ket{{\rm Coh}(-uC_1)}\bra{{\rm Coh}(-uC_1)}\right ).
$$
The optimal measurement of the local parameter for such model and the connection to quantum imaging was discussed in section \ref{sec:QGSM}. 

To conclude, we showed how the general results can be applied to study the class of 2 dimensional system and environment unit QMCs and provided a pathway towards a complete minimax estimation theory for QMCs using a two step procedure, similar to that used in the i.i.d. theory \cite{YCH19,LAN6}.

\section{Conclusions and outlook}

In this paper we developed the identifiability and asymptotic convergence theory for irreducible quantum Markov chains (QMCs). Previous studies dealt with primitive QMCs \cite{Gu11,GK15} or continuous-time Markov processes \cite{CG15,GK17}, but did not cover the case of periodic dynamics. We showed that when one has access to the stationary output of an irreducible QMC, the space of identifiable dynamical parameters is the quotient of the space of isometries corresponding to different dynamics, with a symmetry group of system transformations. Geometrically, the quotient is a stratified space called orbifold which consists of a manifold corresponding to equivalence classes of primitive dynamics, together with lower dimensional strata corresponding to periodic dynamics. Using orbifold theory we gave a detailed description of the quotient in terms of orbifold charts, tangent spaces and tangent cones. This structure offers new geometric insight into the general theory of QMCs, and the interplay between periodicity and identifiability. As the class of QMCs stationary output states coincides with that of purely generated finitely correlated states \cite{Werner1994}, our result applies to the latter, and it would be interesting to see how this analysis can be extended to more general finitely correlated states \cite{FNW92}. We expect that this extension will be intimately related to the identifiability theory of classical hidden Markov chains \cite{Pe69,IAK92}. 

The second main thrust of the paper was towards the asymptotic estimation theory of QMCs. Here we proved the local convergence of the stationary output process to a limit model consisting of a product between a Gaussian shift model and a mixture of quantum Gaussian shift models. The convergence is strong (in the Le Cam sense \cite{LeCam}), holds over growing neighbourhoods of local parameters, and is equipped with explicit rates. The limit objects are determined by local geometric data of the parameter space: the tangent space of identifiable directions, the complex inner product on this space and the unitary action of the stabiliser group. From this data we constructed an algebra of canonical commutation relations  \cite{Petz90} over the tangent space, together with a quantum Gaussian shift of coherent states, and the second quantised action of the stabiliser group. The limit model is the average of the Gaussian shift model under this action, and splits into a tensor product between a pure Gaussian shift corresponding to the tangent space to the orbifold (non-singular part) and a mixed model corresponding to the tangent cone (singular part). Notably, our result shows that quantum local asymptotic normality holds only around primitive dynamics while more general Gaussian mixtures describe the local asymptotic behaviour around periodic QMCs. In the i.i.d. setup, quantum local asymptotic normality \cite{LAN1,GJ07,GJK08,GK09,GG13,YFG13,FY20,FY23} has been used to construct asymptotically optimal estimation strategies for finite dimensional states \cite{LAN1,GG13,YCH19,FY23}
as well as for pure states non-parametric models \cite{BGM18,LN24}. Our convergence results can be used in a similar way to map the QMC estimation problem onto one for mixtures of Gaussian states and derive optimal rates and procedures, in conjunction with recent measurement techniques such as output post-processing \cite{GG23,YPH23,GGG25}. While this goes beyond the scope of the current project, we have made a first step in providing a detailed analysis of the class of QMCs with two dimensional system and environment unit, including the procedure for preliminary estimation in different one-parameter sub-models. In the paper we noted an interesting connection with recent research in quantum imaging \cite{Tsang16,Tsang19} which suggests that in certain Gaussian mixture models the QFI of the local parameter does not vanish when approaching the singularity, unlike the case of classical Gaussian mixtures  \cite{Chen95}. 
Along the way we developed general 
 convergence results and a new geometric picture, which we believe are applicable to other areas of quantum statistical inference.

\vspace{2mm}

{\bf Acknowledgements.} 
M.G. acknowledges the support from the Engineering and Physical Sciences Research Council, grant no. EP/T022140/1. F.G. has been partially supported by the MUR grant Dipartimento di Eccellenza 2023–2027 of Dipartimento di Matematica, Politecnico di Milano and is a member of the INdAM-GNAMPA group. We thank Ian Dryden, Karthik Bharath and Theodore Kypraios for fruitful discussions during the manuscript preparation. 

\begin{appendix}
\section{Proof of the results regarding the mixed quantum Gaussian model} \label{app:equivmixedGauss}

\begin{proof}[Proof of Lemma \ref{lem:idgm}]
Let us consider two states 
$\rho(x), \rho(y)$ in $\mathbf{GM}$ corresponding to $x, y \in X$ respectively. From equation \eqref{eq:alternative} one can check that the states coincide if and only if $x_0=y_0$ and there exist $\theta_0, \dots, \theta_{p-1} \in \mathbb{R}$ such that $\ket{\zeta_m(y_0^\perp)}=e^{\frac{2\pi\theta_{m}i}{p}}\ket{\zeta_m(x_0^\perp)}$ for $m=0,\dots, p-1$; notice that $\theta_m$ are only determined modulo $p$. Using the explicit expression for $\ket{\zeta_m(x_0^\perp)}$ and $\ket{\zeta_m(y_0^\perp)}$, we can rewrite the previous condition as
\[\begin{split}
    &e^{-\|y_0^\perp\|^2/2} \left ( \delta_0(m) \ket{\Omega_0^\perp} + \sum_{l \geq 1}\frac{1}{\sqrt{l!}}\sum_{m_1\oplus\dots\oplus m_l=m} R_{m_1}\ket{y_0^\perp} \otimes \cdots \otimes R_{m_l}\ket{y_0^\perp} \right)=\\
    &e^{\frac{2\pi\theta_{m}i}{p}} e^{-\|x_0^\perp\|^2/2} \left ( \delta_0(m) \ket{\Omega_0^\perp} + \sum_{l \geq 1}\frac{1}{\sqrt{l!}}\sum_{m_1\oplus\dots\oplus m_l=m} R_{m_1}\ket{x_0^\perp} \otimes \cdots \otimes R_{m_l}\ket{x_0^\perp} \right),\end{split}\]
    which, in turn, is equivalent to
    \begin{equation} \label{eq:phase1} e^{-\|y_0^\perp\|^2/2} =e^{\frac{2\pi\theta_{0}i}{p}} e^{-\|x_0^\perp\|^2/2} \end{equation}
    and for every $l\geq 1$, $m=0, \dots, p-1$, $m_1 \oplus \cdots m_l=m$
    \begin{equation} \label{eq:phase2}
    e^{-\|y_0^\perp\|^2/2}R_{m_1}\ket{y_0^\perp} \otimes \cdots \otimes R_{m_l}\ket{y_0^\perp}=e^{-\|x_0^\perp\|^2/2}e^{\frac{2\pi\theta_m i}{2}}R_{m_1}\ket{x_0^\perp} \otimes \cdots \otimes R_{m_l}\ket{x_0^\perp}.
    \end{equation}
    Notice that Eq. \eqref{eq:phase1} implies that $\theta_0\equiv_p0$ and $\|y_0^\perp\|=\|x_0^\perp\|$. Moreover, if we choose $l=1$ in Eq. \eqref{eq:phase2}, we obtain
    $$R_{m}\ket{y_0^\perp}=e^{2\pi\theta_m i}R_{m}\ket{x_0^\perp}, \quad m=1, \dots, p-1.$$ Choosing $l=m=1,\dots, p-1$ and $m_1=\dots=m_m=1$, one gets
    $$m\theta_1\equiv_p \theta_m,\quad p\theta_1\equiv_p 0,$$
    which implies that
    $\theta_1\equiv_pk$ for some $k=0,\dots, p-1$ and $\theta_m\equiv_p mk$. Recalling that $U^k \cdot R_{m}\ket{x_0^\perp}=e^{\frac{2 m k\pi i }{p}}R_{m}\ket{x_0^\perp}$ for every $k,m=0, \dots, p-1$, one obtains the statement.

\end{proof}

\begin{proof}[Proof of Proposition \ref{prop.mixture.model}] From the splitting ${\cal V}_0^\perp=\bigoplus_{m=1}^{p-1}{\cal V}_m$, for every $n \geq 1$, we can write 
$$ 
({\cal V}_0^\perp)^{\otimes_s n}=\bigoplus_{m=0}^{p-1}{\frak H}_{n,m}({\cal V}_0^\perp), \quad {\frak H}_{n,m}({\cal V}_0^\perp)={\rm Symm}\left[
\bigoplus_{\substack{m_1,\dots, m_n =0,\dots, p-1 \\m_1 \oplus \cdots \oplus m_n=m}}
{\cal V}_{m_1} \otimes \cdots \otimes {\cal V}_{m_n}
\right].
$$
For $n=0$, we denote ${\frak H}_{0,0}({\cal V}_0^\perp)=\mathbb{C} \ket{\Omega_0^\perp}$ and ${\cal H}_{0,m}({\cal V}_0^\perp)=\{0\}$ for $m =1,\dots, p-1$.
Let us define $Q_m$ as the orthogonal projection onto
$$
{\frak V}_m:= \bigoplus_{n\geq 0}{\frak H}
_{n,m}({\cal V}_0^\perp)=
\bigoplus_{n\geq 0}\,
\left[
\bigoplus_{\substack{m_1,\dots, m_n =0,\dots, p-1 \\m_1 \oplus \cdots \oplus m_n=m}}{\cal V}_{m_1} \otimes \cdots \otimes {\cal V}_{m_n}.
\right]
$$
Note that $\Gamma(U)=\sum_{m=0}^{p-1} \gamma^{m} \id \otimes Q_m$: indeed, this follows from the fact that
$$\Gamma(U)_{|{\cal V}_{m_1} \otimes \cdots \otimes {\cal V}_{m_n}}=U^{\otimes n}_{|{\cal V}_{m_1} \otimes \cdots \otimes {\cal V}_{m_n}}=\gamma^{m_1 \oplus \cdots \oplus m_n} \id_{|{\cal V}_{m_1} \otimes \cdots \otimes {\cal V}_{m_n}}.$$
Therefore, for every $\ket{x}=\ket{x_0}+\ket{x_0^\perp}$ such that $\ket{x_0}\in {\cal V}_0$ and $\ket{x_0^\perp} \in {\cal V}_0^\perp$, we can write
\[\begin{split}&\frac{1}{p}\sum_{k=0}^{p-1} \alpha_{U*}^k(W(x)|\Omega\rangle\langle \Omega|W(x)^*)\\
&=W(x_0)|\Omega_0\rangle\langle \Omega_0|W(x_0)^*\otimes \sum_{m,m^\prime=0}^{p-1} \underbrace{\left(\frac{1}{p}\sum_{k=0}^{p-1} \gamma^{k(m^\prime-m)}\right )}_{=\delta_{m,m^\prime}}Q_{m^\prime} W(x_0^{\perp})|\Omega_0^\perp\rangle\langle \Omega_0^\perp|W(x_0^{\perp})^*Q_{m}\\
&=W(x_0)|\Omega_0\rangle\langle \Omega_0|W(x_0)^*\otimes \sum_{m=0}^{p-1} Q_{m} W(x_0^{\perp})|\Omega_0^\perp\rangle\langle \Omega_0^\perp|W(x_0^{\perp})^*Q_{m}\\
&=W(x_0)|\Omega_0\rangle\langle \Omega_0|W(x_0)^*\otimes \sum_{m=0}^{p-1}\ket{\zeta_m(x_0^{\perp})}\bra{\zeta_m(x_0^{\perp})},\end{split}\]
where
\[\begin{split}&\ket{\zeta_m(x_0^{\perp})}=Q_m W(x_0^{\perp})|\Omega_0^\perp\rangle\\
    &=e^{-\|x_0^\perp\|^2/2} \left ( \delta_0(m) \ket{\Omega_0^\perp} + \sum_{l \geq 1}\frac{1}{\sqrt{l!}}\sum_{m_1\oplus\dots\oplus m_l=m} R_{m_1}\ket{x_0^\perp} \otimes \cdots \otimes R_{m_l}\ket{x_0^\perp} \right).\end{split}\]
In the last equation we applied formula \eqref{eq:Weyl.operator.action} for the action of Weyl operators.

\end{proof}
\section{Proofs of lemmas on convergence of statistical models}
\label{app:lemmas.convergence.models}

In this section we improve on Lemma 5 in \cite{GK15} showing that under suitable assumptions one can upgrade weak convergence to strong convergence in pure state models even when the parameter space is not compact. Before proceeding, we recall the following well known fact; since we could not find any reference, we report the short proof for the reader's convenience.
\begin{lemma} \label{lem:sqrt}
Let $A$ and $B$ be two positive semidefinite $d\times d$ matrices with complex entries. The following holds true:
$$\|\sqrt{A}-\sqrt{B}\|_\infty \leq \|A-B\|_\infty^{1/2}.$$
\end{lemma}
\begin{proof}
Let $x$ be any normalized eigenvector at which $\sqrt{A}-\sqrt{B}$ attains the norm, i.e.
$$(\sqrt{A}-\sqrt{B})x=\lambda x \text{ with } |\lambda|=\|\sqrt{A}-\sqrt{B}\|_\infty.$$
Notice that $A-B=\sqrt{A}(\sqrt{A}-\sqrt{B})+(\sqrt{A}-\sqrt{B})\sqrt{B}$, therefore
$$\langle x, (A-B) x \rangle=\langle x, \sqrt{A}(\sqrt{A}-\sqrt{B})x \rangle+\langle x(\sqrt{A}-\sqrt{B})\sqrt{B}x \rangle=\lambda \langle x(\sqrt{A}+\sqrt{B})x \rangle.$$
Let us consider
$$\|\sqrt{A}-\sqrt{B}\|_\infty^2=(\langle x,(\sqrt{A}-\sqrt{B}) x \rangle)^2\leq |\lambda|\langle x,(\sqrt{A}+\sqrt{B}) x \rangle=|\langle x, (A-B) x \rangle| \leq \|A-B\|_\infty. $$
\end{proof}

\begin{lemma}[{\bf Weak to strong convergence for rank one operators}] \label{lem:wts1}
Let $\{r_n\}_{n \in \mathbb{N}}$ be a sequence of positive numbers such that $r_n \uparrow +\infty$, $C$ a positive constant, $X$ be a normed space of finite dimension $d$, $A$ a finite set. Moreover, let us consider the sequence of collections of rank-one operators given by
\begin{align*}
    &\mathbf{Q}(n):=\{\rho_n(x):=\ket{v_n(x,a)}\bra{v_n(x,a)}\}_{\{(x,a) \in {\cal U}(n)\times A\}}, \\
    &\mathbf{Q}^{\infty}:=\{\rho(x):=\ket{v(x,a)}\bra{v(x,a)}\}_{\{(x,a) \in X\times A\}},
    \end{align*}
where ${\cal U}(n):=B_{C r_n}(0)\subseteq H$ is the ball centered in $0$ with radius $C r_n$ and we assume that $\|v_n(x,a)\|,\|v(x,a)\| \leq 1$ . We will use $\mathbf{Q}^{\infty}(n)$ to denote the model $\mathbf{Q}^{\infty}$ restricted to ${\cal U}(n)\times A$.

Let us assume that the following hypotheses hold true:
\begin{enumerate}
\item \label{it:Lp1} \textbf{(H\"older parametrization)} there exist positive constants $C_1,C_2, \alpha_1,\alpha_2>0$ such that for every $x,y \in X$ and $a \in A$
\begin{align*}&\|v(x,a)\|^2\|v(y,a)\|^2-|\langle v(x,a),v(y,a)\rangle|^2 \leq C_1 \|x-y\|^{\alpha_1},\\
&(\|v(x,a)\|^2-\|v(y,a)\|^2)^2 \leq C_2\|x-y\|^{\alpha_2}\end{align*}
\item \label{it:fc1} \textbf{(Fast enough uniform weak convergence)} let
$$f(n):=\sup_{(x,a),(y,b) \in {\cal U}(n)\times A}|\langle v_n(x,a),v_n(y,b) \rangle -\langle v(x,a),v(y,b) \rangle|;$$
assume that
$$\lim_{n \rightarrow +\infty} r_n^df(n)=0.$$
Notice that, in this case, we can find a sequence of positive numbers $\{\delta_n\}_{n \in \mathbb{N}}$ such that $\delta_n \downarrow 0$ and $r_n^df(n)=o(\delta_n^d).$

Then
$$\Delta(\mathbf{Q}(n),\mathbf{Q}^\infty(n))=O\left (\delta_n^{\min\{\alpha_1,\alpha_2\}}+f(n)^{\frac{1}{2}}+\left(\frac{r_n}{\delta_n}\right )^{\frac{d}{4}}f(n)^{\frac{1}{4}}\right )$$
and, therefore,
$$\lim_{n \rightarrow +\infty}\Delta(\mathbf{Q}(n),\mathbf{Q}^\infty(n))=0.$$
\end{enumerate}

\end{lemma}


\begin{proof}[Proof of Lemma \ref{lem:wts1}]
Note that it is always possible to find a finite set $I_n=\{x_1,\dots, x_{k(n)}\}\subseteq {\cal U}(n)$ with $k(n) \lesssim (r_n/\delta_n)^d$ such that for every $x \in {\cal U}(n)$, there exists $x_i \in I_n$ with $\|x-x_i\| \leq \delta_n$. This implies that the corresponding operators are close as well: using item \ref{it:Lp1}. one has that for $n$ big enough the following holds true:
\begin{equation}\label{eq:ho}\begin{split}
   \|\rho(x,a)-\rho(x_i,a)\|_1 &=\sqrt{(\|v(x,a)\|^2-\|v(x_i,a)\|^2)^2+}\\
   &\overline{+4(\|v(x,a)\|^2\|v(x_i,a)\|^2-|\langle v(x,a),v(x_i,a)\rangle|^2)} \leq \widetilde{C}\delta_n^{\min\{\alpha_1,\alpha_2\}/2}.
\end{split}\end{equation}
Moreover, note that
\[\begin{split}
    \|\rho_n(x,a)-\rho_n(x_i,a)\|_1 &=\|\rho(x,a)-\rho(x_i,a)\|_1 \quad (I)\\
    &+\|\rho_n(x,a)-\rho_n(x_i,a)\|_1-\|\rho(x,a)-\rho(x_i,a)\|_1 \quad (II).
\end{split}\]
Eq. \eqref{eq:ho} implies that $(I) \leq \widetilde{C}\delta_n^{\min\{\alpha_1,\alpha_2\}/2}$ for $n$ big enough. Using that $x \mapsto \sqrt{x}$ is H\"older $1/2$ and item \eqref{it:fc1}, one has that 
\[
\begin{split}
 &|\|\rho_n(x,a)-\rho_n(x_i,a)\|_1-\|\rho(x,a)-\rho(x_i,a)\|_1| \leq\\
 &\sqrt{|(\|v(x,a)\|^2-\|v(x_i,a)\|^2)^2-(\|v_n(x,a)\|^2-\|v_n(x_i,a)\|^2)^2+}\\
 &\overline{+4(\|v(x,a)\|^2\|v(x_i,a)\|^2-\|v_n(x,a)\|^2\|v_n(x_i,a)\|^2)+}\\
 &\overline{+4(|\langle v_n(x,a),v_n(x_i,a)\rangle|^2)-|\langle v(x,a),v(x_i,a)\rangle|^2)|}=O\left (f(n)^{\frac{1}{2}}\right)
\end{split}
\]
uniformly in $(x,a) \in {\cal U}(n) \times A$. Let us consider the Gram matrices of size $|A|\cdot k(n)$
$$(G_n):=\langle v_n(x_i,a),v_n(x_j,b) \rangle, \quad (G^\infty_n):=\langle v(x_i,a),v(x_j,b) \rangle.$$
Notice that
\begin{equation}
    \|G_n-G^\infty_n\|_\infty \leq |A|k(n)f(n) \lesssim \left ( \frac{r_n}{\delta_n}\right )^{d}f(n) \xrightarrow[n \rightarrow +\infty]{}0.
\end{equation}
Let $\{e_{i,a}\}$ be the canonical basis of $\widetilde{\frak h}_n:=\mathbb{C}^{|A|k(n)}$; we can define the following mappings:
$$v_n(x_i,a) \mapsto \sqrt{G_n} e_{i,a} \in \widetilde{\frak h}_n, \quad v(x_i,a) \mapsto \sqrt{G^\infty_n} e_{i,a} \in \widetilde{\frak h}_n$$ and extend them by linearity to two isometries denoted by $V_n$ and, respectively, $V_n^{\infty}$, on
$$
{\frak h}_n:={\rm span}_{\mathbb{C}}\{v_n(x_i,a):(x_i,a) \in I_n\times A \} \text{ and }  {\frak h}^\infty_n:={\rm span}_{\mathbb{C}}\{v(x_i,a):(x_i,a) \in I_n\times A\}.
$$
Let us consider the following Hilbert spaces:
$$\overline{{\frak h}}_n:=\overline{{\rm span}_{\mathbb{C}}\{v_n(x,a):(x,a) \in {\cal U}(n)\times A\}}, \quad \overline{{\frak h}}^\infty_n:=\overline{{\rm span}_{\mathbb{C}}\{v(x,a):(x,a) \in {\cal U}(n)\times A\}}. $$
At the cost of possibly enlarging the dimension of $\widetilde{\frak h}_n$, we can consider two arbitrary extensions of $V_n$ and $V_n^{\infty}$ to $\overline{{\frak h}}_n$ and $\overline{{\frak h}}^\infty_n$, respectively; with an abuse of notation we will still denote them by $V_n$ and $V_n^{\infty}$. Let us introduce the following quantum channels:
$$A_n(\cdot):=V_n^{\infty*}V_n \cdot V_n^* V_n^{\infty}, \quad B_n(\cdot):=V_n^{*}V^\infty_n \cdot V_n^{\infty*} V_n.$$
Let us consider any $x_i \in I_n$ such that $\|x-x_i\| \leq \delta_n$; one has
\[\begin{split}\|A_n(\rho_n(x_i,a))-\rho(x_i,a)\|_1 &=\|V_n^{\infty*}V_n \rho_n(x_i,a) V_n^* V_n^{\infty}-\rho(x_i,a)\|_1\\
&\leq \|V_n \rho_n(x_i,a) V_n^* -V_n^{\infty}\rho(x_i,a)V_n^{\infty*}\|_1\\
&=\left \|\ket{\sqrt{G_n}e_{i,a}}\bra{\sqrt{G_n}e_{i,a}}-\ket{\sqrt{G^\infty_n}e_{i,a}}\bra{\sqrt{G^\infty_n}e_{i,a}}\right \|_1\\
&=\sqrt{(\|\sqrt{G_n}e_{i,a}\|^2-\|\sqrt{G^\infty_n}e_{i,a}\|^2)^2+}\\
&\overline{+4\left (\|\sqrt{G_n}e_{i,a}\|^2\|\sqrt{G^\infty_n}e_{i,a}\|^2-\left |\langle \sqrt{G_n}e_{i,a},\sqrt{G^\infty_n}e_{i,a}\rangle \right |^{2}\right )}.
\end{split}\]
Note that
\[
(\|\sqrt{G_n}e_{i,a}\|^2-\|\sqrt{G^\infty_n}e_{i,a}\|^2)^2=(\langle e_{i,a}, (G_n-G_n^\infty)e_{i,a} \rangle )^2 \leq \|G_n-G_n^{\infty}\|_\infty^2 
\]
and
\[
\begin{split}
&\|\sqrt{G_n}e_{i,a}\|^2\|\sqrt{G^\infty_n}e_{i,a}\|^2-\left |\langle \sqrt{G_n}e_{i,a},\sqrt{G^\infty_n}e_{i,a}\rangle \right |^{2}\\
&=\|\sqrt{G_n^\infty}e_{i,a}\|^2\langle \sqrt{G_n}e_{i,a},(\sqrt{G_n}-\sqrt{G_n^\infty})e_{i,a}\rangle\\
&+\langle \sqrt{G^\infty_n}e_{i,a},(\sqrt{G^\infty_n}-\sqrt{G_n})e_{i,a}\rangle\langle \sqrt{G_n}e_{i,a},\sqrt{G^\infty_n}e_{i,a}\rangle\\
&\leq 2 \|\sqrt{G^\infty_n}-\sqrt{G_n}\|_{\infty} \leq 2 \|G^\infty_n-G_n\|^{1/2}_{\infty}.
\end{split}
\]
In the last inequality we used Lemma \ref{lem:sqrt}. We obtained that
$$\max_{(x_i,a) \in I_n\times A}\|A_n(\rho_n(x_i,a))-\rho(x_i,a)\|_1 =O \left ( \left(\frac{r_n}{\delta_n}\right )^{\frac{d}{4}}f(n)^{\frac{1}{4}}\right ).$$
Analogously one can show that
$$\max_{(x_i,a) \in I_n\times A}\|\rho_n(x_i,a)-B_n(\rho(x_i,a))\|_1=O\left ( \left(\frac{r_n}{\delta_n}\right )^{\frac{d}{4}}f(n)^{\frac{1}{4}}\right ).$$

We can finish our proof: let us consider $x \in {\cal U}(n)$, $a \in A$ and $x_i \in I_n$ such that $\|x-x_i\| \leq \delta_n$, then for $n$ big enough one has
\[\begin{split}\|A_n(\rho_n(x,a))-\rho(x,a)\|_1&\leq \|A_n(\rho_n(x,a)-\rho_n(x_i,a))\|_1+\|A_n(\rho_n(x_i,a))-\rho(x_i,a)\|_1\\
&+\|\rho_n(x_i,a)-\rho(x,a)\|_1\leq\widetilde{C}\delta^{\min\{\alpha_1,\alpha_2\}/2}_n+O\left ( \left(\frac{r_n}{\delta_n}\right )^{\frac{d}{4}}f(n)^{\frac{1}{4}}\right ).\end{split}\]
We can repeat the same steps in order to get the statement for $\|\rho_n(x,a)-B_n(\rho(x,a))\|_1$ as well.
\end{proof}

If we deal with statistical models, the statement of the previous Lemma becomes that of Lemma \ref{lem:wts2} in the main text.

\vspace{2mm}

\begin{proof}[Proof of Lemma \ref{lemm:mixed}]
Lemma 3.12 in \cite{GJ07} ensures that for every $n$ there exist quantum channels ${\cal T}_n$ and ${\cal S}_n$ such that $\Delta(\mathbf{P}(n),\mathbf{P}^{\infty}(n))$ is the maximum between
$$\sup_{(x,a) \in {\cal U}(n)\times A} \|\TT_n(\ket{v_n(x,a)}\bra{v_n(x,a)})-\ket{v(x,a)}\bra{v(x,a)}\|_1 $$
and
$$\sup_{(x,a) \in {\cal U}(n)\times A} \|\ket{v_n(x,a)}\bra{v_n(x,a)}-{\cal S}_n(\ket{v(x,a)}\bra{v(x,a)})\|_1. $$

Observe that for every $x \in {\cal U}(n)$ one has
\[\begin{split}
   & \|\TT_n(\rho_n(x))-\rho(x)\|_1=\left \| \sum_{a \in A }\TT_n(\ket{v_n(x,a)}\bra{v_n(x,a)})-\ket{v(x,a)}\bra{v(x,a)}\right \| \leq \\
    & \sum_{a \in A }\left \| \TT_n(\ket{v_n(x,a)}\bra{v_n(x,a)})-\ket{v(x,a)}\bra{v(x,a)}\right \| \leq |A| \Delta(\mathbf{P}(n),\mathbf{P}^{\infty}(n)).
\end{split}\]
We used linearity of the quantum channel and triangular inequality. Analogously, one obtains
$$\|\rho_n(x)-{\cal S}_n(\rho(x))\|_1 \leq |A|\Delta(\mathbf{P}(n),\mathbf{P}^{\infty}(n)). $$
Therefore,
$\Delta(\mathbf{Q}(n),\mathbf{Q}^{\infty}(n)) \leq |A| \Delta(\mathbf{P}(n),\mathbf{P}^{\infty}(n))$ and the statement follows immediately.
\end{proof}
\section{Proof of the limit theorems for fluctuation observables} \label{ap:fo}

\begin{proof}[Proof of Lemma \ref{lem:ltflu}]
First, notice that the convergence of $\mathbb{P}(\overline{Q}_n,\rho^{\rm out}_V(n,\rho))$ follows from the Central Limit Theorem for fluctuations operators derived below, by a simple application of Slutsky's theorem. Let us now prove the Central Limit Theorem.

\textbf{1. Reduction to an initial state supported only on one ${\cal H}_a$.} Notice that
\[\begin{split}
\rho^{\rm out}_V(n,\rho) 
&=\sum_{\mathbf{i}, \mathbf{j} \in I_n} \tr( K_{\mathbf{i}}\rho K_{\mathbf{j}}^*)\ket{\mathbf{i}}\bra{\mathbf{j}}=\sum_{a,b=0}^{p-1}
\sum_{\mathbf{i}, \mathbf{j} \in I_n}\tr( K_{\mathbf{i}}P_a\rho P_b K_{\mathbf{j}}^*)\ket{\mathbf{i}}\bra{\mathbf{j}}\\&=\sum_{a,b=0}^{p-1}
\sum_{\mathbf{i}, \mathbf{j} \in I_n}\tr( P_{a \oplus n} K_{\mathbf{i}}P_a \rho P_b K_{\mathbf{j}}^*P_{b\oplus n})\ket{\mathbf{i}}\bra{\mathbf{j}}=\sum_{a=0}^{p-1} \sum_{\mathbf{i}, \mathbf{j} \in I_n}\tr(K_{\mathbf{i}}P_a\rho P_a K_{\mathbf{j}}^*)\ket{\mathbf{i}}\bra{\mathbf{j}}\\
&=\sum_{a=0}^{p-1}\tr(\rho P_a) \rho^{\rm out}_V(n,\rho_a)
,
\end{split}\]
where $\rho_a:=(P_a\rho P_a)/\tr(\rho P_a)$ (we interpret $\rho^{\rm out}_V(n,\rho_a)$ as $0$ if $\tr(\rho P_a)=0$). Therefore, we can reduce to prove the theorem for initial states which are supported in the range of only one $P_a$.

\bigskip \textbf{2. Reduction to a ``coarser'' output chain.} We recall that $Q$ acts on a finite number of output units, let us say $k \in \mathbb{N}^*$; moreover, without loss of generality, we can assume that ${\rm Tr}(\rho^{\rm out}_V(n)Q)=0$ (we can always reduce to this case centering $Q$). Let $M$ be the smallest multiple of $p$ which is bigger or equal than $k$ and let us define $m=M/p$, $\tau_l$ as the operator that sends any observable $X$ acting on the first $k$ output units into $X^{(l+1)}$, which acts on the output units $\{l+1,\dots,l+k\}$, and $\widetilde{\tau}_l=\tau_{lM}$ for every $l \in \nn$.

With a slight abuse of notation, for every $k,n \in \mathbb{N}^*$ such that $k <n$, we will often make use of the identification $L^\infty({\cal K}^{\otimes k}) \simeq L^\infty({\cal K}^{\otimes k}) \otimes \mathbf{1}_{\cal K}^{\otimes (n-k)} \subset L^\infty({\cal K}^{\otimes n})$; for instance, we will use the notation $\tr(\rho_V(n,\rho)Q)$ to denote $\tr(\rho_V(n,\rho)Q\otimes \mathbf{1}_{\cal K}^{\otimes (n-k)})$.

We need to introduce the following sequence of operators:
\begin{align*}
    \widetilde{F}_n(\widetilde{Q})&:=\frac{1}{\lfloor (n-k)/M \rfloor^{1/2}}\sum_{l=0}^{\lfloor (n-k)/M \rfloor-1}\widetilde{\tau}_l(\widetilde{Q}) \in L^\infty({\cal K}^{\otimes n}) , \\ \widetilde{Q}&:=\frac{1}{M^{1/2}}\sum_{l=0}^{M-1} \tau_l(Q) \in L^\infty({\cal K}^{\otimes (M+k)}) \subseteq L^\infty({\cal K}^{\otimes 2M}).
\end{align*}
For every $n$, let us use the notation $f(n)=\lfloor (n-k)/M \rfloor$ and $r(n)=n-k+1-M\lfloor (n-k)/M \rfloor$; notice that $0 \leq r(n) <M$ and, if $n\rightarrow +\infty$, then $f(n) \rightarrow +\infty$ as well. One has the following estimates:
\[\begin{split}
\|F_n(Q)-\widetilde{F}_n(\widetilde{Q})\|_\infty &\leq \left | \frac{1}{(f(n)M+r(n))^{1/2}}-\frac{1}{(f(n)M)^{1/2}}\right |\left \| \sum_{l=0}^{f(n)M-1} \tau_l(Q)\right \|_\infty\\
&+\frac{1}{(n-k+1)^{1/2}} \left \| \sum_{l=f(n) M}^{n-k}\tau_l(Q)\right \| _\infty \\
&\leq \frac{1}{(f(n)M)^{1/2}}\underbrace{\left |\left (1+\frac{r(n)}{f(n)M} \right)^{-1/2}-1 \right |}_{=O((f(n)M)^{-1})}f(n)M\|Q\|_\infty \\
&+ \frac{1}{(n-k+1)^{1/2}} r(n)\|Q\|_\infty =O(n^{-1/2}). \end{split}\]
Therefore, for every $\alpha \in \mathbb{R}$ one has
$$\|e^{i\alpha F_n(Q)}-e^{i\alpha \widetilde{F}_n(\widetilde{Q})}\|_\infty=\left \|i\alpha \int_0^1 e^{i\alpha s \widetilde{F}_n(\widetilde{Q})} (F_n(Q)- \widetilde{F}_n(\widetilde{Q})) e^{i(1-s)\alpha F_n(Q)} ds \right \|_\infty =O(n^{-1/2}) $$
and, consequently,
$$\lim_{n \rightarrow \infty} 
\left|{\rm Tr}(\rho^{\rm out}_V(n,\rho) e^{i\alpha F_n(Q)})-{\rm Tr}(\rho^{\rm out}_V(n,\rho) e^{i\alpha \widetilde{F}_n(\widetilde{Q})})
\right|=0.$$
Therefore we can restrict to study the limit of ${\rm Tr}(\rho^{\rm out}_V(n,\rho) e^{i\alpha \widetilde{F}_n(\widetilde{Q})})$. Note that $\widetilde{F}_n(\widetilde{Q})$ only changes when $n-k$ is a multiple of $M$ and, hence, $p$. Therefore,
$$
\lim_{n \rightarrow\infty} {\rm Tr}(\rho^{\rm out}_V(n,\rho) e^{i\alpha \widetilde{F}_n(\widetilde{Q})})=\lim_{n \rightarrow \infty}{\rm Tr}(\rho^{\rm out}_V(nM+k,\rho)e^{i\alpha \widetilde{F}_{nM+k}(\widetilde{Q})}).
$$
Notice that $\widetilde{F}_{nM+k}(\widetilde{Q})$ is a sequence of fluctuation operators obtained from a two site local operator $\widetilde{Q} \in L^\infty(\widetilde{{\cal K}}^{\otimes 2})$, where each new output unit is obtained grouping $M$ original output units and, therefore, is described by the Hilbert space $\widetilde{{\cal K}}:={\cal K}^{\otimes M}$.

\bigskip \textbf{3. Reduction to stationarity.} Let us pick $j<n$, $j,n \in \nn^*$, then
\[
\begin{split}
    \left \|\widetilde{F}_{nM+k}(\widetilde{Q})-\widetilde{\tau}_j\left (\widetilde{F}_{(n-j)M+k}(\widetilde{Q})\right ) \right \|_\infty&\leq \frac{1}{n^{1/2}}\left \|\sum_{l=0}^{j-1}\widetilde{\tau}_l(\widetilde{Q}) \right \|_\infty\\
    &+\left ( \frac{1}{\sqrt{n-j}}-\frac{1}{\sqrt{n}} \right ) \left \| \sum_{l=j}^{n-1} \widetilde{\tau}_l(\widetilde{Q}) \right \|_\infty \\
    &=O(n^{-1/2}). 
\end{split}
\]
Note that
$${\rm Tr}\left (\rho^{\rm out}_V(nM+k,\rho)e^{i\alpha\widetilde{\tau}_j\left (\widetilde{F}_{(n-j)M}(\widetilde{Q})\right )} \right )=
{\rm Tr}\left ( \rho^{\rm out}_V((n-j)M+k,\TT_*^{Mj}(\rho))
e^{i\alpha\widetilde{F}_{(n-j)M}(\widetilde{Q})} \right ) $$
and, since $M$ is a multiple of $p$ and $\rho$ is supported in the range of $P_a$ for some $a \in \{0,\dots, p-1\}$,
$$\lim_{j\rightarrow \infty} \TT_*^{Mj}(\rho)=p \rho_a^{\rm ss}.$$
Since $\rho \mapsto \rho^{\rm out}_V(n,\rho) $ is a norm one linear mapping, we obtain that
$$\lim_{n \rightarrow \infty}|{\rm Tr}(\rho^{\rm out}_V(nM+k,\rho) e^{i\alpha \widetilde{F}_{nM+k}(\widetilde{Q})})-{\rm Tr}(\rho^{\rm out}_V(nM+k,p\rho^{\rm ss}_a)e^{i\alpha \widetilde{F}_{nM+k}(\widetilde{Q})})|=0.$$
Notice that the collection $(\rho_V^{\rm out}(nM,{p \rho^{\rm ss}_a}))_{n \geq 0}$ determines a translational invariant ${\rm C}^*$-finitely correlated state on the coarser output chain satisfying the assumptions of Proposition 4.2 in \cite{Ma03}, therefore one has that if $\rho$ is supported in the rank of $P_a$ and if $\tr(\rho_V^{\rm out}(M+k,{p \rho^{\rm ss}_a})\widetilde{Q})=0$, then the convergence in law holds
$$
\mathbb{P}(\widetilde{F}_{nM+k}(\widetilde{Q}), \rho^{\rm out}_V(nM+k,p\rho^{\rm ss}_a)) \longrightarrow {\cal N}(0,\sigma^2_{Q,a})$$
with
$$\sigma^2_{Q,a}=\lim_{n \rightarrow \infty}{\rm Tr}(\rho^{\rm out}(nM+k, p\rho_a^{\rm ss})(\widetilde{F}_{nM+k}(\widetilde{Q})^2).$$
If we check that for every $a=0,\dots, p-1$, $\rho^{\rm out}_V(M+k,p\rho^{\rm ss}_a)(\widetilde{Q})=0$ and $\sigma_{Q,a}^2$ does not depend on $a$ and is given by the expression in Eq. \eqref{eq:asvar}, we are done. Notice that
\[\begin{split}\tr(\rho^{\rm out}_V(M+k,p\rho^{\rm ss}_a)\widetilde{Q})&=\frac{1}{\sqrt{M}}\tr\left (\rho^{\rm out}_V(M+k,p\rho^{\rm ss}_a)\left (\sum_{l=0}^{M-1}\tau_l(Q)\right )\right)\\
&=\frac{p}{\sqrt{M}}\sum_{l=0}^{M-1}\tr \left (\rho^{\rm out}_V(k,\TT_*^l(\rho_a^{\rm ss}))Q\right )\\
&=\frac{pm}{\sqrt{M}}\sum_{a=0}^{p-1}\tr \left (\rho^{\rm out}_V(k,\rho_a^{\rm ss})Q\right )\\
&=\sqrt{M}\tr(\rho_V^{\rm out}(k)Q)=0.\end{split}\]
We used the fact that $\TT_*(\rho^{\rm ss}_a)=\rho^{\rm ss}_{a\oplus 1}$, the linearity of $\rho \mapsto \rho^{\rm out}_V(n,\rho)$ and the fact that $\rho^{\rm ss}=\sum_{a =0}^{p-1}\rho^{\rm ss}_a$.
Moreover,
\[\begin{split}\tr \left (\rho^{\rm out}_V(nM+k,p\rho_a^{\rm ss})\widetilde{F}_{nM+k}(\widetilde{Q})^2\right )&=\frac{1}{n}\sum_{l,j=0}^{n-1} \tr \left (\rho^{\rm out}_V(nM+k,p\rho_a^{\rm ss}) \widetilde{\tau}_l(\widetilde{Q})\widetilde{\tau}_j(\widetilde{Q})\right )\\
&=\frac{1}{n}\sum_{l=0}^{n-1} \tr \left (\rho^{\rm out}_V(nM+k,p\rho_a^{\rm ss})\widetilde{\tau}_l(\widetilde{Q}^2)\right )\\
&+\frac{2}{n}\sum_{0 \leq l<j \leq n-1}\tr \left (\rho^{\rm out}_V(nM+k,p\rho_a^{\rm ss})\widetilde{\tau}_l(\widetilde{Q})\circ\widetilde{\tau}_j(\widetilde{Q})\right )\\
&= \tr \left (\rho^{\rm out}_V(M+k,p\rho_a^{\rm ss})\widetilde{Q}^2\right )\\
&+\frac{2}{n}\sum_{l=0}^{n-2} \sum_{j=1}^{n-1-l}\tr \left (\rho^{\rm out}_V((j+1)M+k,p\rho_a^{\rm ss})\widetilde{Q} \circ \widetilde{\tau}_j(\widetilde{Q})\right ) \\
&\xrightarrow[n \rightarrow +\infty]{}\tr \left (\rho^{\rm out}_V(M+kp\rho_a^{\rm ss})\widetilde{Q}^2\right)\\
&+2 \sum_{j=1}^{+\infty}\tr \left (\rho^{\rm out}_V((j+1)M+k,p\rho_a^{\rm ss})\widetilde{Q}\circ\widetilde{\tau}_j(\widetilde{Q})\right ), 
\end{split}\]
where $A \circ B=(AB+BA)/2.$
The limit is finite since $\tr \left (\rho^{\rm out}_V((j+1)M+k,p\rho_a^{\rm ss})\widetilde{Q}\circ \widetilde{\tau}_j(\widetilde{Q})\right )$ decays exponentially with respect to $j$ since $\TT^p_*$ restricted to $L^1({\cal H}_a)$ is primitive. Notice that the previous limit can be rewritten as
\[\begin{split}
    &\tr \left (\rho_V^{\rm out}(M+k,p\rho_a^{\rm ss})\widetilde{Q}^2\right )+2 \sum_{j=1}^{+\infty}\tr\left (\rho_V^{\rm out}((j+1)M+k,p\rho_a^{\rm ss})\widetilde{Q}\circ \widetilde{\tau}_j(\widetilde{Q})\right )\\
    &=\frac{p}{M}\sum_{l,s=0}^{M-1}\tr \left (\rho_V^{\rm out}(M+k,p\rho_a^{\rm ss})\tau_l(Q)\tau_s(Q)\right )\\
    &+\frac{2p}{M}\sum_{l,s=0}^{M-1}\sum_{j=1}^{+\infty}\tr \left (\rho_V^{\rm out}((j+1)M+k,p\rho_a^{\rm ss})\tau_l(Q)\circ \tau_{jM+s}(Q)\right )\\
    &=\tr(\rho^{\rm out}_V(k)Q^2)+\frac{2}{m}\sum_{l=0}^{M-1}\sum_{s=0}^{M-l-1}\tr \left (\rho_V^{\rm out}(s+k,\rho^{\rm ss}_{a\oplus l})Q\circ \tau_{s}(Q)\right)\\
    &+\frac{2}{m}\sum_{l=0}^{M-1}\sum_{s=M-l}^{+\infty}\tr \left (\rho_V^{\rm out}(s+k,\rho^{\rm ss}_{a\oplus l})Q\circ \tau_{s}(Q)\right )\\
    &=\tr(\rho^{\rm out}_V(k)Q^2)+\frac{2}{m}\sum_{s=0}^{+\infty}\sum_{l=0}^{M-1}\tr \left (\rho_V^{\rm out}(s+k,\rho^{\rm ss}_{a\oplus l})Q\circ \tau_{s}(Q)\right )\\
    &=\tr(\rho_V^{\rm out}(k)Q)+2\sum_{s=0}^{+\infty}\tr(\rho_V^{\rm out}(k+s)Q\circ \tau_{s}(Q)).
\end{split}\]
\end{proof}

\section{Proof of the characterisation of the output equivalence} \label{ap:oe}
In the following, we fix two isometries $V_l:\mathcal{H}_l\to \mathcal{H}_l\otimes \mathcal{K}$ with the same space $\mathcal K$ and possibly different spaces $\mathcal{H}_l$, for $l=1,2$. We let $\TT_l:L^\infty(\mathcal H_l)\to L^\infty(\mathcal H_l)$ be the corresponding channels and denote the corresponding Kraus operators by $K_{l,i}$, cf. equation \eqref{eq:dilation.V}. We assume that $\TT_l$ is irreducible with stationary state $\rho^{\rm ss}_{l}$. 
We now introduce a contractive (non-positive) map 
\begin{equation}
\label{eq:T12}
\TT_{12}:L^\infty(\mathcal H_2,\mathcal H_1)\to L^\infty(\mathcal H_1,\mathcal H_2),\quad \TT_{12}(X) = V_1^*(X\otimes \id_{\cal K})V_2=\sum_i K_{1,i}^*XK_{2,i}.
\end{equation}
The following Proposition characterizes the eigenvectors of $\TT_{12}$ corresponding to peripheral eigenvalues. A similar result was proved for \textit{primitive} channels in \cite[Lemma 1]{GK15}; here we show that with a slight modification, the proof works also in the irreducible case.

\begin{proposition}\label{eigenvalueprop} The following statements are equivalent:
\begin{enumerate}
    \item $\TT_{12}$ admits an eigenvalue $c \in \mathbb{C}$ of modulus one;
    \item there exist a unitary operator $W:{\cal H}_2 \rightarrow {\cal H}_1$ and a complex number $c$ with modulus $1$ satisfying the intertwining relations
\begin{equation}\label{Wint}
W\rho^{\rm ss}_{2}=\rho^{\rm ss}_{1}W,\quad WK_{2,i} = c K_{1,i}W, \quad \text{ for all }i=1,\ldots, k.
\end{equation}
In this case, $\TT_{12}(W)=cW$;
\item there exist a unitary operator $W:{\cal H}_2 \rightarrow {\cal H}_1$ and a complex number $c$ with modulus $1$ satisfying the intertwining relations
\begin{equation}\label{Wint2}
(W\otimes \mathbf{1}_{\cal K})V_2=cV_1W.
\end{equation}
In this case, $\TT_{12}(W)=cW$.
\end{enumerate}
\end{proposition}
\begin{proof}
The implication $2. \Rightarrow 1.$ and the equivalence $2. \Leftrightarrow 3.$ are trivial. We are going to show that $1.$ implies $2.$ as well.
Let us consider $F \in L^\infty({\cal H}_2,{\cal H}_1)$ such that $\TT_{12}(F)=cF$; without any loss of generality we can assume that $\|F^*F\|_{\infty}=1$. Since $V_1$ is an isometry, we have 
$V_1V_1^*\leq \id$. Hence,
\begin{align*}
\TT_{2}(F^*F)&=V_2^*(F^*\otimes \id_{\cal K})(F\otimes \id_{\cal K})V_2\geq V_2^*(F^*\otimes \id_{\cal K})V_1V_1^*(F\otimes \id_{\cal K})V_2
\\
&=\TT_{12}(F)^* \TT_{12}(F) =|c|^2F^*F =F^*F,
\end{align*}
so that $\TT_{2}^n(F^*F)\geq F^*F$ for all 
$n\in \mathbb{N}$ by the positivity of $\TT_{2}$. Therefore, ${\cal S}_n(F^*F)\geq F^*F$ for all $n\in \mathbb{N}$, where ${\cal S}_n := \frac 1n \sum_{k=0}^n \TT_{2}^k$.
Let $P$ be the projection onto the eigenspace of $F^*F$ corresponding to its largest eigenvalue $\|F^*F\|_\infty=1$. Using \eqref{average.Schrodinger} (as $\TT_{2*}$ is irreducible), we have
$$
{\rm tr}(\rho^{\rm ss}_{2}F^*F)=\lim_{n\rightarrow\infty} {\rm tr}(P)^{-1}{\rm tr}(P{\cal S}_n(F^*F))\geq {\rm tr}(P)^{-1}{\rm tr}(PF^*F)=\|F^*F\|_\infty=1.
$$
Now, since ${\rm tr}(\rho^{\rm ss}_{2}F^*F)\leq {\rm tr}(\rho^{\rm ss}_{2}) \|F^*F\|_\infty=\|F^*F\|_\infty=1$, in any case, we obtain
${\rm tr}[\rho^{\rm ss}_{2}F^*F]=1$, i.e. $P=\mathbf{1}$ and $F^*F=\mathbf{1}_{{\cal H}_1}$.
Now fix a unit vector $\varphi\in \mathcal H_2$, then one has
\begin{align}\label{inner}
\sum_i \langle K_{1,i}F\varphi|FK_{2,i}\varphi\rangle&=\langle \varphi|F^*T_{12}(F)\varphi\rangle=c \langle \varphi|F^*F\varphi\rangle = c,\\
\sum_i \|K_{1,i}F\varphi\|^2&=\sum_i \langle F\varphi|K_{1,i}^*K_{1,i}F\varphi\rangle = \langle \varphi|F^*F\varphi\rangle= 1,\nonumber\\
\sum_i \|FK_{2,i}\varphi\|^2&= \sum_i \langle K_{2,i}\varphi|F^*FK_{2,i}\varphi\rangle = \sum_i \langle K_{2,i}\varphi|K_{2,i}\varphi\rangle = 1.\nonumber
\end{align}
Hence,
$$
1=\left| \sum_i \langle K_{1,i}F\varphi|FK_{2,i}\varphi\rangle\right|^2 \leq \sum_i \|K_{1,i}F\varphi\|^2 \sum_i \|FK_{2,i}\varphi\|^2 =1.
$$
This shows that Cauchy-Schwartz inequality is saturated, and hence there is a constant $\lambda\in \mathbb C$ such that
\begin{equation}\label{equiv}
FK_{2,i}\varphi=\lambda K_{1,i}F\varphi \text{ for all }\varphi\in \mathcal H_2, \, i=1,\ldots, k.
\end{equation}
Putting this back in \eqref{inner}, we see that $\lambda=c$. Next we show that $\rho^{\rm ss}_1=F\rho^{\rm ss}_2F^*$. To that end we define $\rho^{\rm ss}:= F\rho^{\rm ss}_{2}F^*$; then $\rho^{\rm ss}$ is positive, and ${\rm tr}(\rho^{\rm ss}) = {\rm tr}(F^*F \rho^{\rm ss}_{2})=1$, so $\rho^{\rm ss}$ is a state. We have $K_{1,i}F\rho^{\rm ss}_{2}(K_{1,i}F)^*= FK_{2,i}\rho^{\rm ss}_{2}(FK_{2,i})^*$ for each $i$ by \eqref{equiv}, and hence
$$\TT_{1*}(\rho^{\rm ss}) =\sum_{i=1}^k K_{1,i}F\rho^{\rm ss}_{2}(K_{1,i}F)^*= F\TT_{2*}(\rho^{\rm ss}_{2})F^*= F\rho^{\rm ss}_{2}F^* =\rho^{\rm ss},$$
so by irreducibility, $\rho^{\rm ss}$ is the unique stationary state of $\TT_{1*}$, that is,
\begin{equation}\label{equiv2}
F\rho^{\rm ss}_{2}F^*=\rho^{\rm ss}_{1}.
\end{equation}
\end{proof}

The following lemma provides information on the structure of the output state.
\begin{lemma}\label{outputstructure}
Consider an irreducible QMC defined by an isometry $V$, and let $\rho^{\rm ss}$ be the stationary state with spectral decomposition 
$
\rho^{\rm ss} = \sum_{m=1}^d \pi_m |\phi_m\rangle \langle \phi_m|
$.
Then
\begin{itemize}
\item[(a)] The stationary output state has the form
$$
\rho^{\rm out}_V(n): =\sum_{m,p=1}^{d} \pi_{p} |\psi_{mp}(n) \rangle\langle \psi_{mp}(n)|,
$$
where $\psi_{mp}(n) =\sum_{{\bf i}\in I^{(n)}} \langle \phi_{m}|{\bf K}^{(n)}_{\bf i}\phi_{p}\rangle |{\bf i}\rangle$ for each $m,p = 1,\ldots, d$.
\item[(b)] The rank of $\rho^{\rm out}_V(n)$ does not exceed $d^2$ for any $n\in \mathbb N$. In particular, ${\rm tr}((\rho^{\rm out}_V(n))^2 )\geq d^{-2}$.
\end{itemize}
\end{lemma}
\begin{proof} We have
$$
V^{(n)}\rho^{\rm ss}V^{(n)*}= \sum_{p=1}^{d} \pi_p \sum_{{\bf i},{\bf i}'\in I^{(n)}} |K^{(n)}_{\bf i}\phi_p\rangle\langle K^{(n)}_{{\bf i}'}\phi_p| \otimes |{\bf i}\rangle\langle {\bf i}'|.
$$
Moreover, we have $${\rm tr}(|K^{(n)}_{\bf i}\phi_p\rangle\langle K^{(n)}_{{\bf i}'}\phi_p|) =\sum_{m=1}^{d} \langle \phi_{m} |K^{(n)}_{\bf i}\phi_p\rangle \overline{\langle \phi_{m} |K^{(n)}_{{\bf i}'}\phi_{p}\rangle}.$$ Part (a) follows from this. To prove (b) let $Q_n$ and $r_n$ be the support projection and the rank of $\rho_{U}^{out}(n)$, respectively. From (a) we see that $r_n\leq d^2$ (where the latter is just the number of vectors $\psi_{m'm}(n)$), so ${\rm tr}((\rho_{V}^{out}(n))^2) \geq {\rm tr}(\rho_{U}^{out}(n) (Q_n/\sqrt{r_n}))^2 = 1/r_n \geq 1/d^{2}$.
\end{proof}

The following lemma connects  the notion of output equivalence to the spectrum of the map $\TT_{12}$ defined in Eq. \eqref{eq:T12}.

\begin{lemma}\label{overlaplemma}
\begin{itemize}
\item[(a)] Let $V_1$ and $V_2$ be two isometries, 
and let $\rho^{\rm ss}_l = \sum_i \pi_{l,i} |\phi_{l,i}\rangle\langle\phi_{l,i}| $ be the spectral decompositions of their stationary states, for $l=1,2$.
The overlap of the corresponding output states is given by
$$
{\rm tr}(\rho_{V_1}^{\rm out}(n)\rho_{V_2}^{\rm out} (n))= \sum_{m,p=1}^{d}\sum_{m',p'=1}^{d} \pi_{1,p}\pi_{2,p'}|\langle \phi_{1,p}|\TT_{12}^n(|\phi_{1,m}\rangle\langle \phi_{2,m'}|)|\phi_{2,p'}\rangle|^2.
$$
\item[(b)] If $V_1$ and $V_2$ are equivalent, then the map $\TT_{12}$ has an eigenvalue of modulus one.
\end{itemize}
\end{lemma}
\begin{proof} Define $\psi_{l,mp}(n) =\sum_{{\bf i}\in I^{(n)}} \langle \phi_{l,m}|{\bf K}^{(n)}_{l,{\bf i}}\phi_{l,p}\rangle |{\bf i}\rangle$ for each $m,p = 1,\ldots, d$, and $l=1,2$. Using part (a) of the previous Lemma, we obtain
\begin{align*}
{\rm tr}(\rho_{V_1}^{\rm out}(n)\rho_{V_2}^{\rm out} (n))&= \sum_{m,p=1}^{d}\sum_{m',p'=1}^{d} \pi_{1,p}\pi_{2,p'} |\langle\psi_{1,mp}(n)|\psi_{2,m'p'}(n)\rangle|^2,
\end{align*}
which has the form claimed in (a), because
\begin{align*}
\langle \psi_{1,mp}(n)|\psi_{2,m'p'}(n)\rangle &=
\sum_{{\bf i}\in I^{(n)}} \overline{\langle \phi_{1,m}|{\bf K}_{1,\bf i}^{(n)}\phi_{1,p}\rangle}\langle \phi_{2,m'}|{\bf K}_{2,\bf i}^{(n)} \phi_{2,p'}\rangle\\
&=\langle \phi_{1,p}|\TT_{12}^n(|\phi_{1,m}\rangle\langle \phi_{2m'}|)|\phi_{2,p'}\rangle.
\end{align*}
In order to prove (b), recall from the previous section that the spectral radius of $\TT_{12}$ does not exceed 1. Hence, if $\TT_{12}$ does not have any eigenvalue of modulus one, then all its eigenvalues have modulus strictly less than 1, which (by the standard Jordan decomposition) implies $\lim_{n\rightarrow\infty} \TT_{12}^n=0$, and hence $\lim_{n\rightarrow \infty} {\rm tr}(\rho_{V_1}^{\rm out}(n)\rho_{V_2}^{\rm out} (n)) =0$ by (a). But then the output states cannot be equal, because ${\rm tr}((\rho_{V_l}^{\rm out}(n))^2)$ must be bounded away from zero by part (b) of Lemma \ref{outputstructure}.
\end{proof}

We are now ready to prove the characterisation of the output equivalence.

\begin{proof}[Proof of Theorem \ref{thm:identification}] Since (iii) is just a reformulation of (ii), it suffices to show that (i) is equivalent to (ii). From the explicit form of the stationary output state, one can easily verify that (ii) implies (i). In order to prove the other implication, first note that (i) implies that the map $T_{12}$ defined in equation \eqref{eq:T12} has eigenvalue of modulus one by part (b) of Lemma \ref{overlaplemma}, which implies (ii) by Proposition \ref{eigenvalueprop}. 
\end{proof}

\section{Proofs of the results in Section \ref{sec:global}} \label{sec:proofsgeom}

\begin{proof}[Proof of Theorem \ref{almostfree}] The nontrivial part consists of showing that $\mathcal P^{\rm irr}=\manifirr / \mathscr{G}$ and proving that the action is almost free. After proving this, one can transfer the orbifold structure of $\manifirr/\mathscr{G}$ (Proposition 1.5.1 in \cite{Ca22}) to $\orbifirr$.

Theorem \ref{thm:identification} shows that $V$ and $V^\prime$ are output equivalent if and only if there exists $c \in U(1)$, and $W \in PU(d)$ such that
\begin{equation}\label{chieq}
V' = \overline{c}(W\otimes \id_{\cal K})VW^*,
\end{equation}
which holds exactly if $V' = g \cdot V$ for some $g=(c,W)$. This shows that the orbit of any given $V$ under the action of $\mathscr{G}$ coincides exactly with the equivalence class $[V]\in \mathcal P$. This means that $\mathcal P=\manifirr / \mathscr{G}$.

In order to show the rest of the claim we need to prove that the group action $\mu$ defined above is smooth, proper and almost free (see e.g. Proposition 1.5.1 \cite{Ca22}). To that end, we first note that $g\mapsto g \cdot V$ is clearly smooth with respect to $W$ as an element of $U(d)$, and hence also the lift to $PU(d)$ is smooth. Since $\mathscr{G}$ is compact, the action is proper. Let us prove that the action is almost free. Suppose that $g \cdot V=g'V$ for some $V$ and $g=(c,W)$, $g'=(c',W')$. This implies, first of all, that $c \overline{c}^\prime VW^{\prime*}W= (W^{\prime *}W\otimes \id_{\cal K})V$, and consequently $$\TT_V(W^{\prime *}W) = V^*(W^{\prime *}W\otimes \id_{\cal K})V = c\overline{c}^\prime V^*VW^{\prime *}W=c\overline{c}^{\prime}W^{\prime *}W,$$
that is, $W^{\prime *}W$ is an eigenvector of $\TT_V$ with eigenvalue $c\overline{c}^\prime$. Since this has modulus one, irreducibility implies that $c\overline{c}'=\gamma_V^k$ and $W^{\prime *}W=Z_V^k$ for some $k=0,\dots, p_V-1$. Hence $c=c^\prime\gamma^k$ and $W=W^\prime Z^k$ as elements of $PU(d)$.
Hence $g=g^\prime (\gamma_V^k,Z_V^k)$ and
$$
\mathscr{G}_V=\{(\gamma_V^k,Z_V^k): \, k=0,\dots, p_V-1\}.
$$
\end{proof}

\begin{proof}[Proof of Lemma \ref{lem:rm}]
If we consider a Riemannian metric $\nu$ on $\manifirr$ such that for every $V \in \manifirr$
$$\tsident_V\perp\tsnonid_V,$$
then we can consider the new metric
$$
\nu^\prime_V(X,Y):=\int_\mathscr{G} \nu_{g \cdot V}(d_V\mu^g(X),d_V\mu^g(Y))d\xi(g), \quad X,Y \in T_V(\manifirr),$$
where $\xi$ is the unique normalized right invariant Haar measure on $\mathscr{G}$; it is trivial to verify that $\mu^g$ is an isometry for $\nu^\prime$ for every $g \in \mathscr{G}$ and Lemma \ref{lem:compa} ensures that $\tsident_V\perp\tsnonid_V$ under $\nu^\prime$ as well.

In order to conclude, we need to show that there exists a metric $\nu$ with the property above. A possibility is given by
$$
\nu_V(X,Y)=\langle \omega_V(X),\omega_V(Y) \rangle + 2\Re(\tr(\rho_V^{\rm ss} P_V(X)^* P_V(Y))),$$
where $\langle \cdot, \cdot \rangle$ is any inner product on $\{(\theta,K): \, \theta \in \mathbb{R}, \, K \in L^\infty({\cal H}), \, K^*=K\}$.

Finally, we only need to check that averaging over $\mathscr{G}$ does not change the value of $\nu_V(A,A)$ for $A \in \tsident_V$. Notice that for every $A \in \tsident_V$, $g \in \mathscr{G}$ one has
\[\begin{split}
    \nu_{g \cdot V}(d_V\mu^g (A),d_V\mu^g (A))&=\nu_{g \cdot V}(d_V\mu^g (P_V(A)),d_V\mu^g (P_V(A)))\\
    &=\nu_{g \cdot V}(P_{g \cdot V}d_V\mu^g (A),P_{g \cdot V}d_V\mu^g (A))\\
    &=\tr(\rho_{g \cdot V}^{\rm ss}d_V\mu^g (A)^*d_V\mu^g (A))=\tr(\rho_{V}^{\rm ss} A^*A)=\nu_V(A,A).
    \end{split}\]
    We used that for $g=(c, W)$, $\rho^{\rm ss}_{g \cdot V}=W\rho^{\rm ss}_VW^*$ and Lemma \ref{lem:compa}.
\end{proof}

\begin{proof}[Proof of Proposition \ref{prop:sd}]
In order to prove the first two statements, we only need to determine the vector subspace of $\tsident_V$ that is fixed by $d_V\mu^g$ for every $g \in \mathscr{G}_{V}$.  Let us consider $A \in \tsident_V$, then $d_V\mu^g(A)=A$ for $g=(\gamma_V^{-k},Z_{V}^{*k})$ if and only if
\begin{equation} \label{eq:condition}
\gamma_V^k (Z_V^{*k} \otimes \mathbf{1}_{\cal K})A Z_{V}^k=A.
\end{equation}

Let us consider the spectral resolution of $Z_{V}$:
$$Z_{V}^k=\sum_{a =0}^{p_{[V]}-1}\gamma^{ak} P_a(V)$$
and notice that condition in Eq. \eqref{eq:condition} implies that
\begin{equation} \label{eq:condition2}
A=\sum_{a=0}^{p_{[V]}-1}(P_{a \oplus 1}(V)\otimes \mathbf{1}_{\cal K}) A P_a(V).\end{equation}
Indeed, plugging in the spectral resolution of $Z_V^k$ in Eq. \eqref{eq:condition} one obtains that the following equation
$$\sum_{a,b=0}^{p_V-1} \gamma_V^{k(a-b+1)}(P_b(V) \otimes \mathbf{1}_{\cal K})AP_a(V)=\sum_{a,b=0}^{p_V-1} (P_b(V) \otimes \mathbf{1}_{\cal K})AP_a(V)$$
needs to hold for every $k=0,\dots, p_V-1$. Using that $\gamma_V=e^{i2\pi/p_V}$ and that 
$\{P_a(V) : a=0,\dots , p-1\}$ is a resolution of the identity, one gets that $P_b(V) \otimes \mathbf{1}_{\cal K}AP_a(V)$ can only be different from zero when $a-b+1 \equiv_p 0$, i.e. $b \equiv_pa+1$.

Moreover, note that equation \eqref{eq:periodicprop} is equivalent to
\[
VP_a(V) V^*=(P_{a\oplus 1}(V) \otimes \mathbf{1}_{\cal K})VV^*,
\]
which shows that $(P_{a\oplus 1}(V) \otimes \mathbf{1}_{\cal K})$ and $VV^*$ are commuting projections and that the dimension of the range of $(P_{a\oplus 1}(V) \otimes \mathbf{1}_{\cal K})VV^*$ is equal to $d_a([V])$, therefore the dimension of the range of $(P_{a\oplus 1}(V) \otimes \mathbf{1}_{\cal K})(\mathbf{1}_{{\cal H}\otimes {\cal K}}-VV^*)$ is equal to $d_{a\oplus 1}([V])k-d_a([V])$. Since the range of $A \in \tsident_V$ is included in ${\rm Range}(V)^\perp$ (cf. Proposition \ref{horprop}), the real dimension of those tangent vectors which satisfy condition in Eq. \eqref{eq:condition2} is equal to
$$2 \left ( \sum_{a=0}^{p_{[V]}-1} (d_{a\oplus 1}([V])k-d_a([V]))d_a([V])\right ). $$

Let us prove the last statement: $p_{[V]}=p_{[W]}$ follows from the fact that elements laying in the same submanifold of the canonical stratification have the same local group. Moreover, if $[V]$ and $[W]$ are two points sitting on the same connected component of the submanifold $\mathscr{P}_{p, l}$, then they must be connected by a smooth path $\alpha:[0,1] \rightarrow \mathscr{P}_{p, l}$. Since an atlas for $\mathscr{P}_{p, l}$ can be obtained restricting the collection of fundamental charts we built in Subsection \ref{subsec:idatl}, one can easily see that the path $\alpha$ can be lifted to a path $\widetilde{\alpha}$ in $\manifirr$. Let us consider the smooth family of quantum channels $(T_{\widetilde{\alpha}(t)})_{t \in [0,1]}$: since peripheral eigenvalues are algebraically simple, perturbation theory ensures that the family of matrices $Z_{\widetilde{\alpha}(t)}$ is smooth as well. Using the fact that
$$P_a(\widetilde{\alpha}(t))=\sum_{m=0}^{p(\widetilde{\alpha}(t))-1}\overline{\gamma}_{\widetilde{\alpha}(t)}^{am}Z_{\widetilde{\alpha}(t)}^{m},$$
we can deduce that the family of projections $P_a(\widetilde{\alpha}(t))$ is smooth as well, therefore $d_a(\widetilde{\alpha}(t))$ must be constant for all $t \in [0,1]$.
\end{proof}

\section{Proofs of local approximations} \label{app:local}

\begin{proof}[Proof of Proposition \ref{prop:QFI}]
Thanks to the fact that QFI is a nonnegative bilinear form, in order to show that the limit of the scaled QFI only depends on the identifiable component, it is enough to show that for every $A \in \tsnonid_{V_0}$,
\begin{equation} \label{eq:zero}
    \lim_{n \rightarrow +\infty} \frac{F_{V_0,n}(A,A)}{n}=0.
\end{equation}

Let us write the explicit expression of $F_{V_0,n}(A,A)$ for any $A \in T_{V_0}(\manifirr)$: we recall that
$$\ket{\partial_A\Psi^{\rm s+o}_{V_0}(n)}:=i\sum_{i=1}^{n}V_0^{(n)} \cdots V_0^{(i+1)} A V_0^{(i-1)} \cdots V_0^{(1)} \ket{\varphi}$$
and that, for pure statistical models, one has
$$F_{V_0,n}(A,A)=4\left(\|\partial_A\Psi^{\rm s+o}_{V_0}(n)\|^2- \left|\bra{\partial_A\Psi^{\rm s+o}_{V_0}(n)}\ket{\Psi^{\rm s+o}_{V_0}(n)}\right|^2
\right).
$$
It is easy to verify that
$$\langle\Psi^{\rm s+o}_{V_0}(n)|\partial_A\Psi^{\rm s+o}_{V_0}(n)\rangle=i\sum_{i=0}^{n-1}\langle \varphi|\TT_{V_0}^{i}(V_0^*A)\varphi \rangle$$
and
\[\begin{split}
\|\partial_A\Psi^{\rm s+o}_{V_0}(n)\|^2&=\sum_{i=0}^{n-1}\langle \varphi |\TT_{V_0}^{i}(A^*A)\varphi \rangle \quad (I)\\
&+\sum_{1 \leq i<j\leq n}\langle \varphi |\TT_{V_0}^{i-1}(A^*\TT_{V_0}^{j-i-1}(V_0^*A)\otimes \mathbf{1}_{{\cal K}}V_0)\varphi \rangle \quad (II)\\
&+\sum_{1 \leq i<j\leq n}\langle \varphi |\TT^{i-1}(V_0^*\TT_{V_0}^{j-i-1}(A^*V_0)\otimes \mathbf{1}_{{\cal K}}A)\varphi \rangle. \quad (III)\end{split}\]

Note that non identifiable tangent vectors are of the form $\theta V_0 -(K \otimes \mathbf{1}_{\cal K})V_0+V_0 K$ for some $\theta \in \RR$ and for $K \in L^\infty({\cal H})$ such that $K =K^*$ and $\tr(\rho^{\rm ss}_{V_0} K)=0$.

\vspace{2mm}

{\bf Case 1.} Let us first consider a vector of the form $A=\theta V_0$; in this case, using that $V_0^*V_0=\mathbf{1}_{\cal H}$ and that $\TT_{V_0}$ is unital, one immediately sees that $F_{V_0,n}(A,A)=0$ and we are done.

\vspace{2mm}

{\bf Case 2.} Let us now consider $A=-(K \otimes \mathbf{1}_{\cal K})V_0+V_0 K.$
In this case $V_0^*A=({\rm Id}-\TT_{V_0})(K)$ and
$$A^*A=\TT_{V_0}(K^2)-\TT_{V_0}(K)K-K\TT_{V_0}(K)+K^2.$$
Therefore,
$$\langle\Psi^{\rm s+o}_{V_0}(n)|\partial_A\Psi^{\rm s+o}_{V_0}(n)\rangle=i\sum_{i=0}^{n-1}\langle \varphi|\TT_{V_0}^{i}({\rm Id}-\TT_{V_0})(K)\varphi \rangle=i\langle \varphi|({\rm Id}-\TT_{V_0}^n)(K)\varphi \rangle$$
and
$$\lim_{n \rightarrow +\infty}\frac{|\langle\Psi^{\rm s+o}_{V_0}(n)|\partial_A\Psi^{\rm s+o}_{V_0}(n)\rangle|^2}{n}=0.$$
On the other hand,
\[\begin{split}
(I)&=n \tr(\rho_{V_0}^{\rm ss}(\TT_{V_0}(K^2)-\TT_{V_0}(K)K-K\TT_{V_0}(K)+K^2)) +O(1)\\
&=n\tr(\rho_{V_0}^{\rm ss}(({\rm Id}-\TT_{V_0})(K)K+K({\rm Id}-\TT_{V_0})(K)))+O(1).
\end{split}\]
Using the spectral decomposition of $\TT_{V_0}$ as
$$\TT_{V_0}=\sum_{a=0}^{p_{V_0}}\gamma_{V_0}^{a}\ket{Z_{V_0}^{a}}_{HS} \bra{J^*_{V_0,a}}_{HS}+ {\cal R}$$
with $Z^a_{V_0}, J_{V_0,a}$ given by \eqref{eq:pereigen} and $\|{\cal R}^n\|\leq C \zeta^n$ where $C>0$ and $\zeta \in (0,1)$, one can write
\[\begin{split}
(II)&=\sum_{a,b=0}^{p_{V_0}-1}\langle \varphi |Z_{V_0}^{a}\varphi \rangle \langle J^*_{V_0,a}|A^*Z_{V_0}^b \otimes \mathbf{1}_{\cal K}V_0 \rangle_{HS}\langle J^*_{V_0,b}|V_0^*A \rangle_{HS}\cdot \left (\sum_{j=2}^{n}\sum_{i=1}^{j-1} \gamma_{V_0}^{a(i-1)+b(j-i-1)} \right )\\
&+\sum_{j=2}^{n}\sum_{i=1}^{j-1}\tr(\rho_{V_0}^{\rm ss}(A^*{\cal R}^{j-i-1}(V_0^*A)\otimes \mathbf{1}_{\cal K}V_0)\\
&+\sum_{j=2}^{n}\sum_{i=1}^{j-1}\langle \varphi |{\cal R}^{i-1}(A^*V_0)\varphi\tr(\rho_{V_0}^{\rm ss}V_0^*A) \varphi\rangle+O(1).
\end{split}\]
where in the middle term we separated the two contributions from $\mathcal{T}$ and used the contractivity property of $\mathcal{R}$. 
Note that $\tr(\rho_{V_0}^{\rm ss}V_0^*A)=0$ and the last term in the previous equation disappears. Let us rewrite the second term as well as
\[\begin{split}\sum_{i=1}^{n-1}\sum_{k=0}^{n-i-1}\tr(\rho_{V_0}^{\rm ss}(A^*{\cal R}^k(V_0^*A)\otimes \mathbf{1}_{\cal K}V_0)=n\tr(\rho_{V_0}^{\rm ss}(A^*({\rm Id}-{\cal R})^{-1}(V_0^*A)\otimes \mathbf{1}_{\cal K}V_0)+O(1).
\end{split}\]
Let us now focus on the first term: notice that if $a=b$, then using that $Z_{V_0}^a \otimes \mathbf{1}_{\cal K}V_0=\gamma_{V_0}^a  V_0 Z_{V_0}^a$ and that $J_{V_0,a}Z_{V_0}^a=\rho_{V_0}^{\rm ss}$, one has
$$\langle J^*_{V_0,a}|A^*Z_{V_0}^a \otimes \mathbf{1}_{\cal K}V_0 \rangle_{HS}=\gamma_{V_0}^a\langle J^*_{V_0,a}|A^*V_0Z_{V_0}^a\rangle_{HS} =\gamma_{V_0}^a\tr(\rho_{V_0}^{\rm ss}A^*V_0)=0.$$
Likewise, if $b=0$, then
$$\langle J^*_{V_0,b}|V_0^*A \rangle_{HS}=\tr(\rho_{V_0}^{\rm ss}V_0^*A)=0.$$
On the other hand, if $a \neq b$ and $b \neq 0$, then
$$\sum_{j=2}^{n}\sum_{i=1}^{j-1} \gamma_{V_0}^{a(i-1)+b(j-i-1)}=\begin{cases} (1-\gamma_{V_0}^b)^{-1}n +O(1) \text{ if} a=0, \, b \neq 0\\
O(1) \text{ in all the other cases.}
\end{cases}$$
Therefore the first term in $(II)$ is 
$$
n\sum_{b=1}^{p_{V_0}-1}(1-\gamma_{V_0}^b)^{-1}\langle{J^*_{V_0,b}}| V_0^*A \rangle_{HS}\tr\left ( \rho_{V_0}^{\rm ss}A^*Z_{V_0}^{b}  \otimes \mathbf{1}_{\cal K} V_0\right )+O(1).$$
Putting everything together and using $V_0^*A=({\rm Id}-\TT_{V_0})(K)$, we obtain that
\[\begin{split}(II)&=n \tr(\rho_{V_0}^{\rm ss}A^*({\rm Id}-\TT_{V_0})^{-1}(V_0^*A)\otimes \mathbf{1}_{\cal K}V_0) + O(1)\\
&=n \tr(\rho_{V_0}^{\rm ss}A^* (K \otimes \mathbf{1}_{\cal K})V_0)+O(1)\\
&=\tr(\rho_{V_0}^{\rm ss}(K(\TT_{V_0}-{\rm Id})(K))+O(1).\end{split}\]
Observing that $(III)$ is just the complex conjugate of $(II)$ we are done.

Let us focus on the limit of the scaled QFI along identifiable directions: let us consider $A \in \tsident_{V_0}$; remembering that (see Proposition \ref{horprop}) one has $V_0^*A=0$, then
$$\frac{1}{n}F_{V_0,n}(A,A)=\frac{4}{n}\sum_{i=0}^{n-1}\langle \varphi |\TT_{V_0}^{i}(A^*A)\varphi \rangle\xrightarrow[n \rightarrow +\infty]{}4\tr(\rho_{V_0}^{\rm ss}A^*A)=F_{V_0}(A,A)$$
and the statement follows from polarisation identities.
\end{proof}

\begin{proof}[Proof of Proposition \ref{prop:extsolan}]
Both statistical models are pure, therefore we will show that for $X,Y \in B_{r_n}(V_0)\subset T_{V_0}(\manifirr)$ one has
$$\lim_{n \rightarrow \infty}|\langle \Psi_{V(Xn^{-1/2})}(n)| \Psi_{V(Yn^{-1/2})}(n)\rangle- \langle W(X^{\rm id})\Omega| W(Y^{\rm id}) \Omega \rangle|=0$$
uniformly in $X$ and $Y$ and with a suitable rate in order to apply Lemma \ref{lem:wts2}. We recall that
$$\langle W(X^{\rm id})\Omega| W(Y^{\rm id}) \Omega \rangle=e^{ -\frac{1}{2}\beta_{V_0}(X^{\rm id}-Y^{\rm id},X^{\rm id}-Y^{\rm id}) + i \sigma_{V_0}(X^{\rm id},Y^{\rm id})},$$
therefore
\[\begin{split}
    1-|\langle W(X^{\rm id})\Omega| W(Y^{\rm id}) \Omega \rangle|^2&=1-e^{-\beta_{V_0}(X^{\rm id}-Y^{\rm id},X^{\rm id}-Y^{\rm id})} \leq F_{V_0}(X^{\rm id}-Y^{\rm id},X^{\rm id}-Y^{\rm id}) \\
    &\leq \nu_{V_0}(X-Y,X-Y)
\end{split}\]
and condition \ref{it:Lp2} in Lemma \ref{lem:wts2} holds,  with $\alpha=2, C=1$.
Using the explicit expression of the system and output state (Eq. \eqref{eq:sostate}) we obtain \[\begin{split}
&\langle \Psi_{V(Xn^{-1/2})}(n)| \Psi_{V(Yn^{-1/2})}(n)\rangle=\langle \varphi| \TT_{X,Y,n}^n(\mathbf{1}_{\cal H}) \varphi \rangle, \\
&\TT_{X,Y,n}(A):=V(Xn^{-1/2})^* (A \otimes \mathbf{1}_{\cal K}) V(Yn^{-1/2}).\end{split}\]
Note that
\[\begin{split}&V(Xn^{-1/2})=V_0 + \frac{1}{\sqrt{n}}iX+ \frac{1}{2n}\partial^2_X V(0)+O(n^{3(\delta-1/2)}), \\ &V(Yn^{-1/2})=V_0 + \frac{1}{\sqrt{n}}iY+ \frac{1}{2n}\partial^2_Y V(0)+O(n^{3(\delta-1/2)}).\end{split}\]
where the Taylor expansions are taken seeing $V$ as a smooth mappings from $B_\epsilon(V)$ to $L^\infty({\cal H},{\cal H} \otimes {\cal K})\simeq \mathbb{R}^{2d^2k}$; the first order approximation is given by $iX,iY$ because the differential of $V$ in $0$ is the identity mapping (the factor $i$ in front comes from the fact that we removed it in the definition of the abstract tangent space in equation \eqref{eq:tangentmanirr}. We remark that the reminder is $O(n^{3(\delta-1/2)})$ uniformly in $X$ and $Y$ in $B_{Cn^\delta}(0)$. Consequently, one has
$$\TT_{X,Y,n}=\TT_{V_0}+ \frac{1}{\sqrt{n}}\TT_{X,Y,1}+\frac{1}{2n}\TT_{X,Y,2}+O(n^{3(\delta-1/2)}),$$
where
\begin{align*}
    \TT_{X,Y,1}(A)&=-iX^*A\otimes \mathbf{1}_{\cal K}V_0+iV_0^*A\otimes \mathbf{1}_{\cal K} Y,\\
    \TT_{X,Y,2}(A)&=\partial^2_X V(0)A\otimes \mathbf{1}_{\cal K} V_0+2X^*A\otimes \mathbf{1}_{\cal K}Y + V_0^* A\otimes \mathbf{1}_{\cal K} \partial^2_YV(0).
\end{align*}
We now follow similar steps as in Theorem 2 in \cite{Gu11}, but reproduce the calculations in detail in order to carefully track the order of the remainder. Thanks to the irreducibility of $\TT_{V_0}$, perturbation theory ensures that for $n$ big enough $\TT_{X,Y,n}$ admits an algebraically simple eigenvalue $\lambda_{X,Y,n}$ with the largest absolute value; moreover, if we call $x_{X,Y,n}$ the corresponding eigenvector (normalized in a way that $\tr(\rho_{V_0}^{\rm ss}x_{X,Y,n})\equiv 1$), one has that
\[\begin{split}\lambda_{X,Y,n}=1+\frac{1}{\sqrt{n}}\dot{\lambda}_{X,Y}+\frac{1}{2n}\ddot{\lambda}_{X,Y}+O(n^{3(\delta-1/2)}), \\
x_{X,Y,n}=\mathbf{1}_{\cal H}+\frac{1}{\sqrt{n}}\dot{x}_{X,Y}+\frac{1}{2n}\ddot{x}_{X,Y}+O(n^{3(\delta-1/2)}).
\end{split}\]
Differentiating in $0$ the eigenvalue equation $\TT_{X,Y,n}(x_{X,Y,n})=\lambda_{X,Y,n}x_{X,Y,n}$ one obtains
\begin{align}
    &\TT_{X,Y,1}(\mathbf{1}_{\cal H})+\TT_{V_0}(\dot{x}_{X,Y})=\dot{x}_{X,Y}+\dot{\lambda}_{X,Y}\mathbf{1}_{\cal H}, \label{eq:fd}\\
    &\TT_{X,Y,2}(\mathbf{1}_{\cal H})+2\TT_{X,Y,1}(\dot{x}_{X,Y})+\TT_{V_0}(\ddot{x}_{X,Y})=\ddot{x}_{X,Y}+2\dot{\lambda}_{X,Y}\dot{x}_{X,Y}+\ddot{\lambda}_{X,Y}\mathbf{1}_{\cal H}. \label{eq:sd}
\end{align}
Note that
$$\TT_{X,Y,1}(\mathbf{1}_{\cal H})=-iX^*V_0+iV_0^*Y=-iX^{{\rm noid}*}V_0+iV_0^*Y^{\rm noid}$$
because of Proposition \ref{horprop}. Moreover, we recall that $X^{\rm noid}$ is of the form $\theta_X V_0-(K_X \otimes \mathbf{1}_{\cal K})V_0+V_0K$ and the same holds for $Y$; without loss of generality, we can we can assume that $\theta_X=\theta_Y=0$: indeed, if this is not the case, one can consider $e^{-i\theta_Xn^{-1/2}}V(Xn^{-1/2})$ and $e^{-i\theta_Y n^{-1/2}}V(Yn^{-1/2})$ instead of $V(Xn^{-1/2})$ and $V(Yn^{-1/2})$ (notice that the states considered do not change). Therefore,
$$\TT_{X,Y,1}(\mathbf{1}_{\cal H})=i({\rm Id}-\TT)(K_Y-K_X).$$
On the other hand, by differentiating 
$\tr(\rho_{V_0}^{\rm ss}x_{X,Y,n})\equiv 1$, we obtain $\tr(\rho_{V_0}^{\rm ss}\dot{x}_{X,Y})=0$ and plugging into  \eqref{eq:fd} we get 
$\dot{\lambda}_{X,Y}=0$ and
$$\dot{x}_{X,Y}=i(K_Y-K_X).$$

Equation \eqref{eq:sd} becomes
$$\ddot{\lambda}_{X,Y}\mathbf{1}_{\cal H}=\TT_{X,Y,2}(\mathbf{1}_{\cal H})+2\TT_{X,Y,1}(\dot{x}_{X,Y})+(\TT_{V_0}-{\rm Id})(\ddot{x}_{X,Y})$$
and if we evaluate the state corresponding to $\rho^{\rm ss}_{V_0}$ at both sides we obtain that
$$\ddot{\lambda}_{X,Y}=\tr(\rho_{V_0}^{\rm ss}\TT_{X,Y,2}(\mathbf{1}))+2\tr(\rho^{\rm ss}_{V_0}\TT_{X,Y,1}(\dot{x}_{X,Y})).$$
Notice that $\tr(\rho_{V_0}^{\rm ss}\TT_{X,Y,2}(\mathbf{1}))=\tr(\rho_{V_0}^{\rm ss}(\partial^2_X V(0)^* V_0+2X^*Y + V_0^*  \partial^2_YV(0)));$ using that $V(\cdot)^*V(\cdot) \equiv \mathbf{1}_{\cal H}$, one has that
$$\partial^2_X V(0)^* V_0+V_0^*\partial^2_X V(0) =-2X^*X.$$
Therefore,
$$\partial^2_X V(0)^* V_0=\frac{1}{2}(\partial^2_X V(0)^* V_0+V_0^*\partial^2_X V(0))+i \left ( \frac{1}{2i} \partial^2_X V(0)^* V_0-\frac{1}{2i}V_0^*\partial^2_X V(0)\right )$$
and we can write
$$\tr(\rho_{V_0}^{\rm ss}\TT_{X,Y,2}(\mathbf{1}))=-\tr(\rho_{V_0}^{\rm ss} (X^*X-2X^*Y+Y^*Y)) + i \alpha_X -i\alpha_Y,$$
where $\alpha_X,\alpha_Y$ are irrelevant phases that can be absorbed in the states. Moreover,
\[\begin{split}
   &-\tr(\rho_{V_0}^{\rm ss} (X^*X-2X^*Y+Y^*Y))\\
   &=-\tr(\rho_{V_0}^{\rm ss} (X^{{\rm id}*}X^{\rm id}-2X^{{\rm id}*}Y^{\rm id}+Y^{{\rm id}*}Y^{\rm id}))\\
   &+\tr(\rho_{V_0}^{\rm ss} X^{{\rm id}*}K_X \otimes \mathbf{1}_{\cal K} V_0)+\tr(\rho_{V_0}^{\rm ss} V_0^*K_X \otimes \mathbf{1}_{\cal K} X^{\rm id})\\
   &- \tr(\rho_{V_0}^{\rm ss}({\rm Id}-\TT_{V_0})(K_X)K_X+K_X({\rm Id}-\TT_{V_0})(K_X))\\
   &+\tr(\rho_{V_0}^{\rm ss} Y^{{\rm id}*}K_Y \otimes \mathbf{1}_{\cal K} V_0)+\tr(\rho_{V_0}^{\rm ss} V_0^*K_Y \otimes \mathbf{1}_{\cal K} Y^{\rm id})\\
   &- \tr(\rho_{V_0}^{\rm ss}({\rm Id}-\TT_{V_0})(K_Y)K_Y+K_Y({\rm Id}-\TT_{V_0})(K_Y))\\
   &-2\tr(\rho_{V_0}^{\rm ss} X^{{\rm id}*}K_Y \otimes \mathbf{1}_{\cal K} V_0)\\
   &-2\tr(\rho_{V_0}^{\rm ss} V_0^*K_X \otimes \mathbf{1}_{\cal K} Y^{\rm id})\\
   &+2\tr(\rho_{V_0}^{\rm ss}({\rm Id}-\TT_{V_0})(K_X)K_Y+K_X({\rm Id}-\TT_{V_0})(K_Y)).
\end{split}\]
On the other hand,
\[
\begin{split}
   &\tr(\rho^{\rm ss}_{V_0}\TT_{X,Y,1}(\dot{x}_{X,Y}))\\
   &=\tr(\rho^{\rm ss}_{V_0}X^*(K_Y-K_X)\otimes \mathbf{1}_{\cal K} V_0+V_0^*(K_Y-K_X)\otimes \mathbf{1}_{\cal K}Y)\\
   &=\tr(\rho^{\rm ss}_{V_0}X^{{\rm id}*}(K_Y-K_X) \otimes \mathbf{1}_{\cal K} V_0)\\
   &-\tr(\rho^{\rm ss}_{V_0}K_X({\rm Id}-\TT_{V_0})(K_Y)+K_X({\rm Id}-\TT_{V_0})(K_X))\\
   &+\tr(\rho^{\rm ss}_{V_0}V_0^* (K_X-K_Y)\otimes \mathbf{1}_{\cal K} Y^{\rm id})\\
   &-\tr(\rho^{\rm ss}_{V_0}({\rm Id}-\TT_{V_0})(K_X)K_Y)+\tr(\rho^{\rm ss}_{V_0}({\rm Id}-\TT_{V_0})(K_Y)K_Y).
\end{split}
\]
Summing up, one gets
\[
\begin{split}
&-\tr(\rho_{V_0}^{\rm ss} (X^*X-2X^*Y+Y^*Y)+2\tr(\rho^{\rm ss}_{V_0}\TT_{X,Y,1}(\dot{x}_{X,Y}))\\
&=-\tr(\rho_{V_0}^{\rm ss} (X^{{\rm id}*}X^{\rm id}-2X^{{\rm id}*}Y^{\rm id}+Y^{{\rm id}*}Y^{\rm id}))\\
&+\tr(\rho_{V_0}^{\rm ss} (V_0^* K_X \otimes \mathbf{1}_{\cal K} X^{\rm id}-X^{{\rm id}*}K_X \otimes \mathbf{1}_{\cal K} V_0))\\
&+\tr(\rho_{V_0}^{\rm ss} (K_X({\rm Id}-\TT_{V_0})(K_X)-({\rm Id}-\TT_{V_0})(K_X)K_X)\\
&+\tr(\rho_{V_0}^{\rm ss} (Y^{{\rm id}*}K_Y \otimes \mathbf{1}_{\cal K} V_0-V_0^* K_Y \otimes \mathbf{1}_{\cal K} Y^{\rm id}))\\
&+\tr(\rho_{V_0}^{\rm ss} (({\rm Id}-\TT_{V_0})(K_Y)K_Y-K_Y({\rm Id}-\TT_{V_0})(K_Y)).
\end{split}
\]
Notice that
\[\begin{split}&\tr(\rho_{V_0}^{\rm ss} (V_0^* K_X \otimes \mathbf{1}_{\cal K} X^{\rm id}-X^{{\rm id}*}K_X \otimes \mathbf{1}_{\cal K} V_0))\\
&+\tr(\rho_{V_0}^{\rm ss} (K_X({\rm Id}-\TT_{V_0})(K_X)-({\rm Id}-\TT_{V_0})(K_X)K_X)\end{split}\]
and
\[\begin{split}
&\tr(\rho_{V_0}^{\rm ss} (Y^{{\rm id}*}K_Y \otimes \mathbf{1}_{\cal K} V_0-V_0^* K_Y \otimes \mathbf{1}_{\cal K} Y^{\rm id}))\\
&+\tr(\rho_{V_0}^{\rm ss} (({\rm Id}-\TT_{V_0})(K_Y)K_Y-K_Y({\rm Id}-\TT_{V_0})(K_Y)).
\end{split}\]
are purely imaginary and depend on $X$ and $Y$ only, respectively. Therefore, they can be absorbed in $i\alpha_X$ and $i\alpha_Y$ and we can finally write
$$\ddot{\lambda}_{X,Y}=-\tr(\rho_{V_0}^{\rm ss} (X^{{\rm id}*}X^{\rm id}-2X^{{\rm id}*}Y^{\rm id}+Y^{{\rm id}*}Y^{\rm id})) + i \alpha_X -i\alpha_Y.$$
From the definition and the properties of $\beta_{V_0}$ and $\sigma_{V_0}$ one has
$$\tr(\rho_{V_0}^{\rm ss} (X^{{\rm id}*}X^{\rm id}-2X^{{\rm id}*}Y^{\rm id}+Y^{{\rm id}*}Y^{\rm id}))=\beta_{V_0}(X^{\rm id}-Y^{\rm id},X^{\rm id}-Y^{\rm id}) + 2i \sigma_{V_0}(X^{\rm id},Y^{\rm id}).$$

We are finally ready to prove the limit and establish the rate of decay of the error: for $n$ big enough one has
\[\begin{split}|\langle \varphi | \TT_{X,Y,n}^n(\mathbf{1}_{\cal H}) \varphi \rangle-e^{\ddot{\lambda}_{X,Y}/2}| &\leq |\langle \varphi | \TT_{X,Y,n}^n(\mathbf{1}_{\cal H}-x_{X,Y,n}) \varphi \rangle| \quad (I)\\
&+|\lambda_{X,Y,n}|^n |1-\langle \varphi| x_{X,Y,n} \varphi \rangle|\quad (II)\\
&+|\lambda_{X,Y,n}^n-e^{\ddot{\lambda}_{X,Y}/2}| \quad (III).\end{split}\]
Using that $\TT_{X,Y,n}$ is a contraction (therefore $|\lambda_{X,Y,n}|\leq 1$ as well), one gets that $(I)$ and $(II)$ are $O(n^{2(\delta-1/2)})$. Regarding $(III)$, one has that
$$|\lambda_{X,Y,n}-e^{\ddot{\lambda}_{X,Y}}|=\left |\sum_{k=1}^{n}\lambda_{X,Y,n}^{n-k}(\lambda_{X,Y,n}-e^{\ddot{\lambda}_{X,Y}/2n})e^{\ddot{\lambda}_{X,Y}(k-1)/n} \right | =O(n^{3\delta-1/2}),$$
where we used that $|\lambda_{X,Y,n}|,|e^{\ddot{\lambda}_{X,Y}}| \leq 1$ and that
$$\lambda_{X,Y,n}-e^{\ddot{\lambda}_{X,Y}/2n}=O(n^{3(\delta-1/2)}).$$
Putting everything together, we got that
$$\sup_{X,Y \in B_{r_n}^{\rm ext}(V_0)}|\langle \Psi_{V(Xn^{-1/2})}|\Psi_{V(Yn^{-1/2})}\rangle-e^{ -\frac{1}{2}\beta_{V_0}(X^{\rm id}-Y^{\rm id},X^{\rm id}-Y^{\rm id}) + i \sigma_{V_0}(X^{\rm id},Y^{\rm id})}|=O(n^{3\delta-1/2}).$$
Therefore we can apply Lemma \ref{lem:wts2} where $r_n=n^{\delta}$, $\alpha=2$ and $f(n)=O(n^{3\delta-1/2})$. Let us choose $\delta_n$ of the form $n^{-\beta}$ for some $\beta>0.$ The Le Cam distance between the two statistical models scales as
$$n^{-2\beta}+n^{3\delta/2-1/4}+n^{(3+d_{\rm id})\delta/4+d_{\rm id}\beta/4-1/8},$$
therefore it converges to $0$ if $\delta <(2(3+d_{\rm id}))^{-1}$ and $\beta <1/2{d_{\rm id}}-(3+d_{\rm id})\delta/d_{\rm id}.$ If we want the distance above to decay faster than $n^{-2\delta}$ we need to impose that
$$\delta<\min\{1/14,(2(11+d_{\rm id}))^{-1}\}, \quad \delta <\beta<1/2d_{\rm id}-(11+d_{\rm id})\delta/d_{\rm id}.$$
Notice that $[\delta <\beta<1/2d_{\rm id}-(11+d_{\rm id})\delta/d_{\rm id}]$ is bigger than a single point if $\delta<(2(11+d_{\rm id}))^{-1}.$
\end{proof}

\begin{proof}[Proof of Theorem \ref{thm:nLAN}]
    For the sake of conciseness and since it does not generate any confusion, we drop the label $V_0$ from all the objects corresponding to such parameter, putting the emphasis on the local parameters. Let us consider a spectral resolution of $\rho^{\rm ss} = \rho^{\rm ss}_{V_0} $:
    \[\rho^{\rm ss}=\sum_{a=0}^{p-1} \sum_{i \in I_a} \pi_i^a \ket{\phi_i^a}\bra{\phi_i^a}.\]
We recall that $\{\ket{\phi_i^a}\}_{i \in I_a}$ is an orthonormal basis for the range of $P_a$. We recall that the stationary output can be expressed as
\begin{equation} \label{eq:sscm}
\rho^{\rm out}_{X}(n):= \rho^{\rm out}_{[V(Xn^{-1/2})]}(n)=\sum_{\substack{a,b=0,\dots,p-1\\i \in I_a, j \in I_b}}\pi_{i}^a |\psi_{ij,X}^{ab}(n) \rangle\langle \psi_{ij,X}^{ab}(n)|,
\end{equation}
where $|\psi^{ab}_{ij,X} (n)\rangle = \sum_{\mathbf{i} \in I^{(n)}} \langle \phi_{j}^b |\mathbf{K}_{\mathbf{i},X}^{(n)} | \phi_{i}^a\rangle \otimes \ket{\mathbf{i}}$.

\vspace{2mm}

In the primitive case ($p=1$), the output QLAN result  is based on the fact that $|\psi_{ij,X}^{00} (n) \rangle $ are quasi-orthogonal to each other for local parameters, so one only needs to prove QLAN for each component.
Consider an initial state in $\mathcal{H}_a$. Then after $n$ evolution steps for $X=0$, the system state belongs to the subspace $\mathcal{H}_{a\oplus n}$ (where $\oplus$ denotes addition modulo $p$). Therefore
$$
|\psi^{ab}_{ij} (n)\rangle = 0 ,	\qquad {\rm for ~} b\neq n\oplus a.
$$
This suggests that in order to have well defined limits, \textit{one needs to consider subsequences of the type $n = p\cdot l +r$ where $r$ is fixed and $l\in \mathbb{N}$ is allowed to grow}.

\vspace{2mm}

We now let $X,Y \in B_{r_n}^{\rm id}(0)$ be arbitrary fixed local parameters and we study the limit for $l \rightarrow \infty$ of the inner products
\[\langle \psi^{ab}_{ij,X}(n) |\psi^{a^\prime b^{\prime}}_{i^\prime j^\prime ,Y}(n)  \rangle
=
\left. \left. \tr\left( 
\left|\phi_{ i^\prime}^{a^\prime}\right\rangle
\right\langle \phi_{i}^a \right| 
\TT^{n}_{X,Y,n}
\left (\ket{\phi_{j}^b} \bra{\phi_{j^\prime }^{b^\prime}}
\right )
\right )  .
\]
where $\TT^{n}_{X,Y,n}$ is the deformed channel which has been defined in the proof of Theorem \ref{thm:SOlan}.
\begin{lemma} \label{lem:inner}
Let us consider the following eigenvectors of the stationary state $\ket{\phi_{i}^a}$, $\ket{\phi_{i^\prime}^{a^\prime}}$, $\ket{\phi_{j}^b}$, $\ket{\phi_{ j^\prime}^{b^\prime}}$ and the local parameters $X$, $Y$. One has
\[\sup_{X,Y \in B_{C(pl+r)^\delta}(0)} |\langle \psi^{ab}_{ij,X}(pl+r) |\psi^{a^\prime b^{\prime}}_{i^\prime j^\prime,Y}(pl+r)  \rangle|=O(l^{\delta-1/2}\log(l))
\]
unless $a=a^\prime$, $b=b^\prime$, $i=i^\prime$ and $j=j^\prime$, in which case
 \[\sup_{{X,Y \in B_{C(pl+r)^\delta}(0)}} |\langle \psi^{ab}_{ij,X}(pl+r) |\psi^{ab}_{ij,Y}(pl+r)  \rangle-\pi_{j}^b p\tr(\rho_{a}^{\rm ss} \TT^{pl}_{X,Y,n}(P_{b \ominus r}))|=O(l^{\delta-1/2}\log(l)).
\]   
\end{lemma}
\begin{proof}

Let $A = \ket{\phi_{j}^b} \bra{\phi_{ j^\prime}^{b^\prime}}$, $A_\infty:=\delta_{b,b^\prime} \delta_{j,j^\prime} \pi_{j}^b p P_{b\ominus r} $,
we have
\begin{eqnarray*}
&&
\| \TT^{pl+r}_{X,Y,n} (A) -\TT^{pl}_{X,Y,n}(A_\infty) \| \\
&&\leq
\| \TT^{p(l-k)}_{X,Y,n}\TT^{pk + r}_{X,Y,n} (A)- \TT^{p(l-k)}_{X,Y,n}\TT^{pk+r} (A) \| + \| \TT^{p(l-k)}_{X,Y,n} \TT^{pk+r} (A)- \TT^{p(l-k)}_{X,Y,n}\TT^{pk}_{X,Y,n}(A_\infty) \|\\
&&\leq 
  \| \TT^{pk+r}_{X,Y,pl+r} (A)- \TT^{pk+r} (A) \|+ \|\TT^r\TT^{pk}(A) -\TT^{pk}_{X,Y,n}(A_\infty) \| ,
\\
&&\leq 
  \| \TT^{pk+r}_{X,Y,pl+r} (A)- \TT^{pk+r} (A) \|
  + \|\TT^{pk+r}(A) -A_\infty \|
  +\|\TT^{pk}_{X,Y,pl+r} (A_\infty) -A_\infty\|
\end{eqnarray*}
where we used that $\TT^p_{u,v,n}$ is a contraction. Since $\TT$ and $\TT_{X,Y,pl+r}$ are contractions, we have
$$
\| \TT^{pk+r}_{X,Y,pl+r} (A)- \TT^{pk+r} (A) \| \leq (pk+r)\|\TT - \TT_{X,Y,pl+r}\|  = O(kl^{\delta-{1/2}})
$$
which can be controlled if $k$ increases sufficiently slowly. Next, we recall that 
$$
\mathcal{E}(A)=
p\sum_{a=0}^{p-1} \tr(\rho^{\rm ss}_a A) P_a=
\delta_{b,b^\prime} \delta_{j,j^\prime} p \pi^b_j P_{b}
$$ 
so that $\mathcal{T}^r (\mathcal{E}(A))) = A_\infty$, and from \eqref{eq:erg} it follows that
$$
\|\TT^r\TT^{pk}(A) -A_\infty\| =\|\TT^r\TT^{pk}(A) -\TT^r( \mathcal{E}(A))\| \leq \|\TT^{pk} - \mathcal{E}\| \rightarrow 0
$$
exponentially fast in $k$. Moreover 
$$
\|
\TT^{pk}_{X,Y,pl+r} (A_\infty) -A_\infty\|=
\|
\TT^{pk}_{X,Y,pl+r} (A_\infty) -\TT^{pk}(A_\infty)\|\leq pk \|\TT - \TT_{X,Y,pl+r}\|  = O(kl^{-{1/2}+\delta}). 
$$

This proves that the limit of the inner product is zero if either $b\neq b^\prime$ or $j\neq j^\prime$ and that we can substitute $P_{b \ominus r}$ to $\TT^{r}_{X,Y,n}(A)$ when computing the limit. Moving the map $\TT_{X,Y,n}$ on $\ket{\phi^{a^\prime}_{ i^\prime}}\bra{\phi_{i}^a}$ by duality and repeating the same computations, we can conclude by choosing $k$ increasing as $\log l$.

\end{proof}


Recall that the spectral projections  of $Z$ can be expressed as
\[
P_a=\sum_{k=0}^{p-1} \overline{\gamma}^{ak} Z^k.
\]
In order to compute the limit of $\TT^{pl}_{X,Y,n}(P_{b\ominus r})$ when $l \rightarrow \infty$, we can therefore study the limit of the powers of the same map applied to $Z^k$ for $k=0,\dots, p-1$.

First, we need to introduce some some notation. Given any $p$-dimensional vectors $h$, $g$ we denote in the following way the Fourier transform, the reflection around $0$ and the convolution:
$$
\hat{h}_k=\sum_{m=0}^{p-1} \gamma^{mk}h_m, \quad \widetilde{h}_k=h_{-k}, \quad (h*g)_k=\sum_{m=0}^{p-1}h_m g_{k-m}.
$$
We use $h^{*n}$ to denote the convolution of $h$ with itself $n$ times ($(h^{*0})_m=\delta_{0,m}$) and we recall that $\widehat{h^{*n}}=\hat{h}^n$, $\widehat{h^n}=\widehat{h}^{*n}$ and $\hat{\hat{h}}=\widetilde{h}$.

We also recall that $\tsident$ decomposes as the direct sum 
$\bigoplus_{m=0}^{p-1}\mathcal{V}_m$ where $\mathcal{V}_m$ are the orthogonal eigenspaces of the unitary operator $d\mu^{g_1}_{V_0} $, cf. equation \eqref{eq:def.v_k}. 
For $A\in T^{\rm id}$ we denote its projection onto 
$\mathcal{V}_m$ by $A_m$. Let us introduce the sesquilinear mapping $\eta:\tsident\times \tsident \rightarrow \mathbb{C}^{p}$ defined as:
\[
\eta(A,B)_0:=0, \quad \eta(A,B)_m=\tr(\rho^{\rm ss}A_m^*B_m)=\beta(A_m,B_m)+i\sigma(A_m,B_m).
\] 
\begin{lemma}
\label{lemma:tplzk}
Let $k=1,\dots, p-1$. Then the following holds true:
\[
\sup_{X,Y \in B_{C(pl+r)^\delta}(0)}
\left\| \TT^{pl}_{X,Y,pl}(Z^k)-e^{ \lambda_k(X,Y)}Z^k
\right\|_\infty=O(l^{3\delta-1/2})
\]
where

\[\begin{split}
\lambda_k(X,Y)&= -\frac{1}{2} \beta_{V_0}(X_0-Y_0,X_0-Y_0)+i \sigma_{V_0}(X_0,Y_0) +i\alpha_X-i\alpha_Y\\
&-\frac{1}{2} (\beta_{V_0}(X^\perp,X^\perp)+\beta_{V_0}(Y^\perp,Y^\perp))+ \widehat{\eta(X,Y)}_k,
\end{split}
\]
where $\alpha_X$ and $\alpha_Y$ are two (irrelevant) phases depending on $X$ and $Y$ only, respectively.
\end{lemma}
\begin{proof}
In this proof we use the notation $\langle A|B
\rangle_{\rm HS}= {\rm Tr}(A^*B)$.
Since the peripheral eigevalues are isolated and algebraically simple, we can repeat the perturbation argument in Theorem \ref{thm:SOlan} to show that for every $k=0,\dots, p-1$
$$
\sup_{X,Y \in B_{Cn^\delta}(0)}\|\TT^n_{X,Y,n} (Z^k) - \gamma^{nk}  e^{ \gamma^{-k} \widetilde{\lambda}_k(X,Y)}  Z^k\|=O(n^{3\delta-1/2}),
$$
where $\widetilde{\lambda}_k(X,Y)$ has corresponding coefficient $  \langle J_k |\TT_{X,Y,2}(Z^k)\rangle_{\rm HS} /2 $. 
We refer to the proof of Theorem \ref{thm:SOlan} for the definition of $\TT_{X,Y,1}$ and $\TT_{X,Y,2}$. Since the proof is the same, we will only report below the computations which are different.
By perturbation theory one has the following expansions for the perturbations of $\gamma^k$ and $Z^k$:
\begin{align*}
 &\gamma^{-k}\gamma^{(k)}_{X,Y,n}=1+\frac{1}{\sqrt{n}}\gamma^{-k}\dot{\gamma}^{(k)}_{X,Y}+\frac{1}{2n}\gamma^{-k}\Ddot{\gamma}^{(k)}_{X,Y}+O(n^{3(\delta-1/2)}),\\
 &Z^{(k)}_{X,Y,n}=Z^k+\frac{1}{\sqrt{n}}\dot{Z}^{(k)}_{X,Y} + \frac{1}{2n}\Ddot{Z}^{(k)}_{X,Y}+O(n^{3(\delta-1/2)}),
\end{align*}
where we assume without loss of generality that 
$\langle J_k |{Z}^{(k)}_{X,Y} \rangle_{\rm HS}\equiv1$. By differentiting the latter we get 
$\langle J_k |\dot{Z}^{(k)}_{X,Y} \rangle_{\rm HS} =0$. 
On the other hand, differentiating the eigenvector equation, we obtain
\begin{align*}
\TT_{X,Y,1}(Z^k)+\TT(\dot{Z}^{(k)}_{X,Y})&=\dot{\gamma}^{(k)}_{X,Y}Z^k+\gamma^k \dot{Z}^{(k)}_{X,Y},\\
\TT_{X,Y,2}(Z^k)+2\TT_{X,Y,1}(\dot{Z}^{(k)}_{X,Y})+\TT(\Ddot{Z}^{(k)}_{X,Y})&=\Ddot{\gamma}^{(k)}_{X,Y}Z^k+2\dot{\gamma}^{(k)}_{X,Y}\dot{Z}^{(k)}_{X,Y}+\gamma^k \Ddot{Z}^{(k)}_{X,Y}.
\end{align*}
Since $Z^k\otimes \mathbf{1}_{\cal K} V=\gamma^k V Z^k$ we obtain 
$$\TT_{X,Y,1}(Z^k)=-iX^*Z^k\otimes \mathbf{1}_{\cal K}V_0+iV_0^*Z^k\otimes \mathbf{1}_{\cal K}Y=\gamma^k(-iX^*V_0Z^k+iZ^kV_0^*Y)=0,
$$
where we have used that $X,Y\in T^{\rm id}$ implies $V_0^*X = V_0^*Y =0$ (cf. Proposition \ref{horprop}). Tracing the first equation against $J_k$ and using $\langle J_k |\dot{Z}^{(k)}_{X,Y} \rangle_{\rm HS}  =\langle J_k |\mathcal{T}(\dot{Z}^{(k)}_{X,Y}) \rangle_{\rm HS}=0$ one obtains $\dot{\gamma}^{(k)}_{X,Y}=0$.
As a consequence
\[
\dot{Z}^{(k)}_{X,Y}=(\gamma^k{\rm Id}-\TT)^{-1}\TT_{X,Y,1}(Z^k)=0
\]
and
\[
2\widetilde{\lambda}_k(X,Y):=\Ddot{\gamma}^{(k)}_{X,Y}=\langle J_k, \TT_{X,Y,2}(Z^k)\rangle_{\rm HS}.
\]
Using again that $Z^k\otimes \mathbf{1}_{\cal K} V=\gamma^k V Z^k$ and the fact that $J_k Z^k=\rho^{\rm ss}$, one can write
\[\begin{split}
\widetilde{\lambda}_k(X,Y)&=\frac{\gamma^k}{2}  \tr ( \rho^{\rm ss} (\partial^2_X V^*(0)V_0+V_0^*\partial^2_YV(0)))+ \langle J_k, X^* Z^k\otimes \mathbf{1}_{\cal K} Y \rangle.
\end{split}
\]
The statement follows taking $\lambda_k(X,Y)=\overline{\gamma}^{k}\widetilde{\lambda}_k(X,Y)$; the other expression of $\lambda_k$ follows from plugging in the  expression of $J_k$ and $Z^k$ in terms of $P_a$'s.
\end{proof}

Putting together the results of Lemma \ref{lem:inner} and Lemma \ref{lemma:tplzk} we obtain that
\begin{equation} \label{eq:inner2}
\sup_{X,Y \in B_{C(pl+r)^\delta}(0)} 
\left|\langle\psi_{ij,X}^{ab}(pl+r)|\psi_{ij,Y}^{ab} (pl+r)\rangle- \pi_{j}^b p \sum_{k=0}^{p-1} \gamma^{(a-b+r)k} e^{\lambda_k(X,Y)}
\right|=O(l^{3\delta-1/2}).
\end{equation}

Recall that in section \ref{sub:limitmodels} we introduced the mixed Gaussian model ${\bf GM}_{V_0}$ whose states can be expressed as mixtures according to equation \eqref{eq:decomp.gaussian.mixture}.
$$
\rho(X) =
W(X_0)\ket{\Omega_0}\bra{\Omega_0}W(X_0)^* \otimes \sum_{m=0}^{p-1} \ket{\zeta_m(X^\perp)}\bra{\zeta_m({X^\perp})}
$$
The inner products of the right-side components are computed in the following lemma.
\begin{lemma}
\label{lemma:zeta.inner.prod}
Let us consider $X,Y \in \tsident$ and $m,m^\prime =0,\dots,p-1$. Then
\begin{equation}
\label{lemma:zeta.inner.products}
\langle \zeta_m(X^\perp)|\zeta_{m^\prime}(Y^\perp) \rangle=\delta_{m,m^\prime}\cdot e^{-(\beta(X^\perp, X^\perp)+\beta(Y^\perp, Y^\perp))/2} \sum_{k=0}^{p-1}\gamma^{-mk}e^{\widehat{\eta(X,Y)}_k}.
\end{equation}
As a consequence
$$
\langle W(X_0) \Omega_0 \otimes \zeta_m(X^\perp)|W(Y_0) \Omega_0 \otimes \zeta_{m^\prime}(Y^\perp)\rangle = 
\sum_{k=0}^{p-1}\gamma^{-mk} e^{\lambda_k(X,Y)}.
$$
\end{lemma}
\begin{proof}
    The fact that $\ket{\zeta_m(X^\perp)} \perp \ket{\zeta_{m^\prime}(Y^\perp)}$ if $m \neq m^\prime$ is an easy consequence of the fact that they arise from the projection of $W(\widetilde{X})\ket{\Omega_0^\perp}$ into orthogonal subspaces (see Eq. \eqref{eq:zed}). On the other hand,
    \[\begin{split}
\langle \zeta_m(X^\perp)| \zeta_m(Y^\perp) \rangle&=  e^{-(\beta(X^\perp, X^\perp)+\beta(Y^\perp, Y^\perp))/2}\left ( \delta_0(m)+\right .\\
&\left .+\sum_{l \geq 1}\frac{1}{l!}\sum_{m_1 \oplus \dots \oplus m_l=-m} \widetilde{\eta(X,Y)}_{m_1} \cdots \widetilde{\eta(X,Y)}_{m_l}\right )\\
&=e^{-(\beta(X^\perp, X^\perp)+\beta(Y^\perp, Y^\perp))/2}\left ( \delta_0(m)+ \sum_{l \geq 1} \frac{1}{l!}{\widetilde{\eta(X,Y)}}^{*l}_{-m}\right )\\
&=e^{-(\beta(X^\perp, X^\perp)+\beta(Y^\perp, Y^\perp))/2}\left ( \delta_0(m)+ \sum_{l \geq 1} \frac{1}{l!} \sum_{k=0}^{p-1}\gamma^{-km}(\widehat{\eta(X,Y)})^l_k\right )=\\
&e^{-(\beta(X^\perp, X^\perp)+\beta(Y^\perp, Y^\perp))/2} \sum_{k=0}^{p-1}\gamma^{-mk}e^{\widehat{\eta(X,Y)}_k}.  \end{split}\]
    
\end{proof}

In the last part of the proof we use equations \eqref{eq:inner2} and \eqref{lemma:zeta.inner.products} to prove the convergence of the Le Cam distance between the stationary output model and the mixed Gaussian model. For this we consider the following additional families of 
(un-normalised) vectors  
$$\widetilde{\mathbf{Q}}(pl+r):=\left \{ \ket{\psi_{ij,X}^{ab}(pl+r)} ~:~ a,b=0,\dots, p-1, \, i\in I_a,\,j \in I_b, \, X \in B^{\rm id}_{C(pl+r)^{\delta}} \right \},
$$
and
\begin{align*}
\widetilde{\bf G}(pl+r):=&\left \{ |\varphi^{ab}_{ij,X}\rangle:= \sqrt{\pi_j^b p}\, W(X_0)\ket{\Omega_0} \otimes \ket{\zeta_{b\ominus a \oplus r}(X^\perp)}\otimes  \ket{a,b,i,j} ~:~ a,b=0,\dots, p-1,\right .\\
&\left .i\in I_a,\,j \in I_b, \, X \in B^{\rm id}_{C(pl+r)^{\delta}}\right \}
\subset 
{\cal F}(\tsident_{V_0}) \otimes \mathbb{C}^{s},
\end{align*}
where $s=2(p-1)+\sum_{a=0}^{p-1}|I_a|$ (see section \ref{sub:limitmodels} for the definition of the involved objects).

\vspace{2mm}
Lemma \ref{lem:inner}, equation  \eqref{eq:inner2} and Lemma \ref{lemma:zeta.inner.prod} imply that for any fixed $r=0,\dots, p-1$, the inner products of the vectors in 
$\widetilde{\bf Q}(pl+r)$ converge uniformly to the inner products of the vectors in 
$\widetilde{\bf G}(pl+r)$ with a rate of the order $l^{3\delta-1/2}$. 

\vspace{2mm}

Using that
$$\ket{W(X_0)\Omega_0} \otimes \ket{\zeta_{b\ominus a \oplus r}(X^\perp)}=\mathbf{1} \otimes Q_{-b\ominus a \oplus r} W(X)\ket{\Omega},
$$
one has that the distance in norm $1$ between the states (as density matrices)
$$\ket{W(X_0)\Omega_0} \otimes \ket{\zeta_{b\ominus a \oplus r}(X^\perp)}\otimes \ket{a,b,i,j}
$$
and
$$\ket{W(Y_0)\Omega_0} \otimes \ket{\zeta_{a\ominus b \ominus r}(Y^\perp)}\otimes \ket{a,b,i,j}
$$
is less or equal that
$$\|\ket{W(X)\Omega}    \bra{W(X)\Omega}-\ket{W(Y)\Omega}    \bra{W(Y)\Omega}\|_1 \leq \beta_{V_0}(X-Y,X-Y).$$
Therefore one can apply Lemma \ref{lem:wts1}, one has that
\begin{equation} \label{eq:convam}
    \lim_{l\rightarrow \infty}\Delta(\widetilde{\mathbf{Q}}(pl+r),\widetilde{\mathbf{G}}(pl+r))=0.
\end{equation}

We arrived at the last step of the proof.
Consider the stationary output model (with $n=pl+r$)
$$
\mathbf{Q}^{\rm out}_{V_0}(pl+r):=\{\rho_{[V(X(pl+r)^{-1/2})]}^{\rm out}(pl+r) ~:~ X \in B^{\rm id}_{C(pl+r)^\delta}(V_0)\}
$$
and the modified Gaussian mixture model
\begin{equation} \label{eq:model1}
    \widetilde{\mathbf{GM}}(pl+r)
    :=\left \{ \omega_X:=\sum_{m=0}^{p-1} W(X_0)\ket{\Omega_0} \bra{\Omega_0} W(X_0)^* \otimes \ket{\zeta_m(X^\perp)} \bra{\zeta_m(X^\perp)}\otimes \omega_m \right \},
\end{equation}
where
$$\omega_m:= \sum_{a,i,j} \pi_{i}^a\pi_{j}^{a \ominus m\ominus r} p \ket{a,a \ominus m\ominus r,i,j} \bra{a,a \ominus m\ominus r,i,j}.$$
is a state for every $X\in B^{\rm id}_{C(pl+r)^\delta}(V_0)$:
$$\tr(\omega_m)=p\sum_{a=0}^{p-1} \left (\sum_{i}\pi_i^{a} \right )\left (\sum_{j}\pi_j^{ a\ominus m \ominus r} \right )=p \sum_{a=0}^{p-1} \frac{1}{p^2}=1.
$$

Let us now show that equation  \eqref{eq:LeCam.mixed} in the theorem statement is a corollary of \eqref{eq:convam}.

Notice that
\begin{align*}\rho^{\rm out}_X(n)&=\sum_{a,b,i,j} \ket{\sqrt{\pi}_i^a \psi^{ab}_{ij,X}(n)}\bra{\sqrt{\pi}_i^a \psi^{ab}_{ij,X}(n)}, \\
\omega_X&=\sum_{a,b,i,j}\ket{\varphi^{ab}_{ij,X}}\bra{\varphi^{ab}_{ij,X}},\end{align*}
therefore thanks to equation \eqref{eq:convam} and Lemma \ref{lemm:mixed} we get that
$$\lim_{l \rightarrow +\infty}\Delta(\mathbf{Q}^{\rm out}_{V_0}(pl+r),\widetilde{\mathbf{GM}}(pl+r))=0.$$
Equation \eqref{eq:LeCam.mixed} follows from the fact that $\mathbf{GM}(pl+r))$ and $\widetilde{\mathbf{GM}}(pl+r))$ are equivalent: indeed, let us consider the two quantum channels given by
\begin{align*}{\cal S}_*:L^1({\cal F}(T^{\rm id}))\otimes L^1(\cc^s) &\rightarrow L^1({\cal F}(T^{\rm id}))\\
\rho &\mapsto \tr_{\cc^s}(\rho)
\end{align*}
and
\begin{align*}{\cal R}_*:L^1({\cal F}(T^{\rm id}))&\rightarrow L^1({\cal F}(T^{\rm id}))\otimes L^1(\cc^s)  \\
\rho &\mapsto \sum_{m=0}^{p-1} (\mathbf{1}\otimes Q_m )\rho (\mathbf{1}\otimes Q_m ) \otimes \omega_m.
\end{align*}
It is easy to see that for every $X \in T^{\rm id}$ one has
$${\cal S}_*(\omega_X)=\rho(X) \text{ and } {\cal R}_*(\rho(X))=\omega_X.$$

The opposite direction can be shown by performing a similar construction.


\end{proof}

\section{Proofs of Lemma \ref{lemma.SNR} } 
\label{app:proof.lemma.SNR}

\begin{proof}
    One can explicitly check that $\TT_0(\cdot)-\tr(\cdot)\mathbf{1}/2$ is diagonalisable with the following eigenvalues and eigenvectors:
    \begin{align*}&(\lambda_1(0)=0, R_1(0)=\mathbf{1}/2), \, (\lambda_2(0)=-1, R_1(0)=Z), \\
    &\left (\lambda_3(0)=1/\sqrt{3}+1/\sqrt{6}, R_3(0)=\begin{pmatrix}0 & 1 \\ 1 & 0 \end{pmatrix}\right ), \\
    &\left (\lambda_4(0)=-(1/\sqrt{3}+1/\sqrt{6}), R_4(0)=\begin{pmatrix}0 & 1 \\ -1 & 0 \end{pmatrix}\right ).\end{align*}
Perturbation theory ensures that we can define smooth functions $\lambda_j(\theta), L_j(\theta), R_j(\theta)$ such that
\begin{enumerate}
    \item $\TT_\theta(R_j(\theta))=\lambda_j(\theta)R_j(\theta),$
    \item $\TT_\theta^\dagger(L_j)=\overline{\lambda_j(\theta)}L_j(\theta)$, where $\TT_\theta^\dagger$ is the adjoint of $\TT_\theta$ with respect to the Hilbert-Schmidt inner product,
    \item $\langle L_j(\theta)| R_j(\theta)\rangle_{HS}\equiv 1$.
\end{enumerate}
We recall that for every $\theta$,
$$\TT_{0*}(\mathbf{1}/2)=\TT_{0}(\mathbf{1}/2)=\mathbf{1}/2,$$
moreover one can easily check that $\TT_\theta(Z)=(-1+2 \sin^2(\theta))Z$ as well. Therefore, since $|\lambda_1(0)|,|\lambda_3(0)|, |\lambda_4(0)| < |\lambda_2(0)|=1,$ the spectral radius of $\TT_\theta-\tr(\cdot) \mathbf{1}/2$ is given by $|\lambda_2(\theta)|=1-2\sin^2(\theta)$ and, consequently, the spectral radius of $\widetilde{\TT}_\theta=\TT^2_\theta-\tr(\cdot) \mathbf{1}/4$ is $(1-2\sin^2(\theta))^2$. Moreover, one can easily see that for every $k\geq 0$
$$(\TT_\theta-\tr(\cdot)\mathbf{1}/2)^k={\cal V}_\theta^{-1}{\cal D}^k_\theta{\cal V}_\theta, \quad  \text{ where }{\cal D}_\theta:=\begin{pmatrix} 0 & 0 & 0 & 0 \\
0 & \lambda_2(\theta) & 0 & 0 \\ 
0 & 0 & \lambda_3(\theta) & 0 \\
0 & 0 &0 &  \lambda_4(\theta)  \\\end{pmatrix},$$
and ${\cal V}_\theta$ is the transformation that allows to pass from the basis composed by the right eigenvectors $R_1(\theta), \dots, R_4(\theta)$ to the canonical basis, hence whose rows are composed by the conjugate of the coordinates of $L_j(\theta)$s in the canonical basis. We can easily see that
$$\|\widetilde{\TT}^k_\theta\|_{2 \rightarrow 2}=\|{\cal V}_\theta^{-1}{\cal D}^{2k}_\theta{\cal V}_\theta\|_{2 \rightarrow 2} \leq \|{\cal V}_\theta^{-1}\|_{2 \rightarrow 2}\|{\cal V}_\theta\|_{2 \rightarrow 2}\|{\cal D}^{2k}_\theta\|_{2 \rightarrow 2}\leq C|\lambda_2(\theta)|^{2k},$$
where
$$C:=\max_{\theta \in [0,\overline{\theta}]}\|{\cal V}_\theta^{-1}\|_{2 \rightarrow 2}\|{\cal V}_\theta\|_{2 \rightarrow 2}>0.$$
The maximum in the previous equation is well defined and strictly positive due to the continuity of the norm and the compactness of the interval.
\end{proof}

\end{appendix}

\bibliographystyle{imsart-number.bst} 
\bibliography{blibio.bib,biblioII.bib}  



\end{document}